\documentclass[hidelinks,onefignum,reqno,onetabnum]{siamart220329}

\usepackage{lipsum}
\usepackage{amsfonts}
\usepackage{graphicx}
\usepackage{epstopdf}
\usepackage{algorithmic}
\ifpdf
  \DeclareGraphicsExtensions{.eps,.pdf,.png,.jpg}
\else
  \DeclareGraphicsExtensions{.eps}
\fi

\newsiamremark{remark}{Remark}
\newsiamremark{hypothesis}{Hypothesis}
\crefname{hypothesis}{Hypothesis}{Hypotheses}
\newsiamthm{claim}{Claim}

\headers{Linear and Non linear stability for the kinetic plasma \textit{sheath}}{M.Badsi}

\title{Linear and Non linear stability for the kinetic plasma \textit{sheath}  on a bounded interval\thanks{Submitted to the editors DATE.
\funding{This work was funded by the ANR project Muffin ( ANR-19-CE46-0004) }}}

\author{Mehdi Badsi\thanks{Laboratoire de Mathématiques Jean Leray, Nantes Université, UMR CNRS 6629, 2 Chemin de la Houssinière, 44322 Nantes Cedex 03
  (\email{mehdi.badsi@univ-nantes.fr}).}}

\usepackage{amsopn}

\usepackage{mathrsfs}
\DeclareMathOperator{\supp}{supp}

\DeclareMathOperator{\supess}{ess\, sup}

 \def\NN{{\mathbb N}}  
  \def\RR{{\mathbb R}}

\def\dd{\textnormal{d}}

\ifpdf
\hypersetup{
  pdftitle={Linear and Non linear stability for the kinetic plasma sheath},
  pdfauthor={M.Badsi}
}
\fi

\begin{document}

\maketitle

\begin{abstract}
Plasma sheaths are inhomogeneous stationary states that form when a plasma is in contact with an absorbing wall. We prove linear and non linear stability of a kinetic sheath stationary state for a Vlasov-Poisson type system in a bounded interval. Notably, in the linear setting, we obtain exponential decay of the fluctuation provided the rate of injection of particles at equilibrium is smaller than the rate of absorption at the wall. In the non linear setting, we prove a similar result for small enough equilibrium and small localized perturbation of the equilibrium.
\end{abstract}



\section{Introduction}

In this paper, we study the dynamical behavior of solutions to a kinetic model of plasma-wall interaction.
When a plasma made of positive ions is in contact with a partially absorbing wall, the relative difference in mobility between ions and electrons yields the formation of a boundary layer near the wall which is characterized by the formation of a positive space charge. This charge creates an electric field which attracts the ions and repels the electrons at a rate which balances the loss of charges at the wall. It results in a permanent regime. This phenomenon is known as the plasma sheath \cite{Chen,Stangeby_2000}. The study of the plasma sheath is of significant importance in the design of laboratory plasma devices such as Tokamaks \cite{Munschy-Sheath, Manfredi,Valsaque}.  
In this work, we consider a bounded and uni-dimensional plasma made of only one kind of ions in which is immersed an infinite metallic wall. The time scale of interest in our study is that of ions so that electrons  are  assumed to be at a thermodynamical equilibrium.  We shall model the mesoscopic behavior of the ions using a kinetic approach where positions and velocities of the particles in the phase space $Q = (0,1) \times \RR$ are denoted $x$ and $v$. We then denote by $f(t,\cdot,\cdot) \geq 0$  the density of ions in the phase space $Q$ at time $t \geq 0$ and by $\phi(t,\cdot)$ the electrostatic potential at time $t.$
We are interested in long time stability properties of $(f,\phi)$ solution to the Vlasov-Poisson equations,
\begin{align}
    \partial_t f  + v \partial_x f  - \partial_{x}\phi(t,x) \partial_v f  = 0, \quad t > 0, (x,v) \in Q, \label{Vlasov-i}\\
    -\lambda^2 \partial_{xx} \phi(t,x) = \int_{\RR} f(t,x,v) \dd v - n_e(\phi(t,x)), \quad  t \in \RR^{+}, \quad x \in (0,1), \label{Poisson}
\end{align}
where $ \lambda > 0$ is a normalized Debye length. Here,
\begin{align}
n_e  : \RR \rightarrow \RR^{+}_{*} \textnormal{ is  } \mathscr{C}^{1} \textnormal{ and } n_{e}' > 0. \label{hyp_ne_1}
\end{align}
A typical example of such a function in the context of plasma physics is the so called Maxwell-Boltzmann density \cite{Chen} which is given by $n_e(\psi) = n_{0} e^{\psi}$ where $n_{0} > 0$ is a normalized density of electrons. To close the system \eqref{Vlasov-i}-\eqref{Poisson}, we have to define boundary conditions and to prescribe the initial data.
We thus define the set of incoming particles, the set of outgoing particles and the singular set by
\begin{align}
&\Sigma^{\textnormal{inc}} := \lbrace (0, v)  \: : \: v > 0 \rbrace \cup \lbrace (1, v) \: : \: v < 0 \rbrace,\\
&\Sigma^{\textnormal{out}} := \lbrace (0, v)  \: : \: v < 0 \rbrace \cup \lbrace (1, v) \: : \: v  > 0 \rbrace,  \\
&\Sigma^{0} := \lbrace (0,0) \rbrace \cup \lbrace (1,0 ) \rbrace.
\end{align}
We therefore prescribe $f(t,\cdot,\cdot)$ on $\Sigma^{\textnormal{inc}}$ as follows
\begin{align}
\forall t > 0,  \quad f(t,0,v) = \mu(v) \textnormal{ if } v > 0  \quad f(t,1,v) = 0 \quad \textnormal{ if } v < 0, \label{bc-f}
\end{align}
where $\mu$ denotes a stationary density of incoming particles  that comes from the plasma core while at the wall ($x =1$) particles are absorbed.
As for the electric potential, we impose the following Dirichlet boundary conditions
\begin{align}
\phi(t,0) = 0 , \quad \phi(t,1) = \phi_{b}, \label{bc-phi}
\end{align}
where the reference of potential is chosen to be at $x = 0$ (the plasma core) and $\phi_{b}\in \RR$ denotes the voltage at the metallic wall. The system \eqref{Vlasov-i}-\eqref{bc-phi} is eventually completed with the initial condition:
\begin{align}
    f(0,x,v) = f_{0}(x,v) \label{initial_condition}
\end{align}
where $f_{0}$ is the initial density of ions. 

Up to our knowledge, the wellposedness theory of Vlasov-Poisson equations in bounded domain has not been investigated in full details. Stationary solutions have been obtained in \cite{Raviart, Rein-stationary, Degond_Jaffard_Poupaud_Raviart_1996, badsi_krm, Badsi_Godard-Cadillac_2022}.  As for the time dependent problem, existence of weak solutions has been obtained by Ben Abdallah in \cite{Ben-Abdallah}, while uniqueness of the mild solution in the one dimensional  case was proven by Bostan in \cite{Bostan} for decreasing electric fields. In the case of the specular reflection boundary condition,  global wellposedness of classical solutions has been obtained by H-J Hwang and J.Velazquez in \cite{Velazquez_2008} using the assumption that the electric field points outward at the boundary. Recently,  Cesbron and Iacobelli in \cite{cesbron-iacobelli} proved a similar result for the Vlasov-Poisson equations with massless electrons. In the case of the half-space, Y. Guo in \cite{guo_1994} proved for the Vlasov-Poisson equations,  existence and uniqueness of weak solutions in the class of essentially bounded and of bounded variations functions.

As far as the stability of stationary solutions is concerned. The Vlasov-Poisson equations have a long history which dates back to the seminal paper by Landau \cite{Landau}. Indeed, in the absence of spatial boundaries, the stability of homogeneous (in space) equilibria  (either in the linear or in the non linear setting) has been studied by many authors. A non exhaustive list is \cite{Degond-Spectral-Theory, Villani_2012,Despres_VP,Han-Kwan-Rousset-Nguyen}. It is known that for a homogenous plasma whose equilibrium density of particles verifies the so called Penrose condition,  a slight perturbation of the equilibrium density creates an electric field which tends to zero in the long time asymptotic. This phenomenon is known as the Landau Damping.
As far as we know, in the presence of spatial boundaries and with inhomogeneous equilibria, there is no analogue to the Landau Damping.  Nevertheless, spectral stability or Lyapunov type stability using the Energy Casimir method have been carried in  \cite{Strauss_2013,Rein_1994,Ben_Artzi_2011,Badsi_2016, Hwang_2016,Guo-BGK-Waves}. We mention that in the use of the Energy Casimir method, very often a structural assumption on the equilibrium density is assumed. Namely, it has to be a decreasing function of the microscopic energy. As for the spectral analysis, a symmetric structure is needed to study the spectrum of the appropriate operator.  Unfortunately, in the case of plasma sheath equilibria, we have not been able to use these tools in a satisfactory way, notably because the so called Bohm condition  \eqref{Bohm} prevents the non trivial equilibrium from being a decreasing function of the microscopic energy.

The approach used in this work is different and is much more in the spirit of the work of Glass, HanKwan and Moussa \cite{HK-Glass-Moussa} to study the stability of equilibria for the Vlasov-Navier-Stokes system. Our approach fully exploits the structure of the transport operator for the Vlasov equation, and the fact that the non trivial sheath equilibrium yields an electric field which is uniformly positive. At equilibrium, all the trajectories leave the phase space in some uniform time. This yields a kind of first order delayed Grönwall inequality satisfied by the $L^{1}$ norm of the perturbation for times larger than the time of exit of all the trajectories. Our linear stability Theorem \ref{stability_linear_eq}  then roughly says that if the rate of injection of particles at equilibrium in the plasma core is smaller than the rate of absorption of particles at the wall then any perturbation is prone to decay exponentially fast. Our non linear stability Theorem \ref{non_linear_stab_thm} is a perturbative variant of the linear stability result. We prove that if the equilibrium density is small enough then a small localized perturbation of the equilibrium density  yields that all the characteristics associated to the non linear Vlasov-Poisson equations still leave the phase space in some uniforme time. We obtain as in the linear analysis an exponential decay of the perturbation for times larger than the time of exit of all the perturbed trajectories. The difficult part in the non linear analysis consists in propagating the local in time stability estimates. This is done by a continuation argument: using the delayed Grönwall inequality we show that if the initial fluctuation is small enough then the perturbation on the electric field stays bounded by a constant, say $C$, up  to a critical time. We then establish that this bound on the electric field propagates after the critical time provided the semi-norm in $W^{1,1}$ of the equilibrium density is chosen small enough. So far, we mention that in the non linear analysis, only stability estimates are given, existence and uniqueness of the global mild solution for the non linear system is not addressed directly in this paper and is left for a future work.

This work is organized as follows. We study the stationary solutions in Section \ref{sec:stationary_solution}. Then,   we study the linearized equations in Section \ref{sec:linear_stab}. We notably prove the linear stability result given in Theorem \ref{stability_linear_eq}. The proof relies on linear elliptic estimates for the linearized Poisson equation that are given in Section \ref{sec:ll_est},  a careful study of the stationary phase portrait that is done in Section \ref{sec:study_of_the_stationary_characteristics} and a delayed Grönwall inequality satisfied by the $L^{1}$ norm of the perturbation for large times that is studied in Section \ref{sec:delayed_gronwall}. We then study the non linear stability in Section \ref{sec:non_linear_stab} using exactly the same strategy. We establish wellposedness and give non linear elliptic estimates for the non linear Poisson equation in Section \ref{sec:nl_ell_est}. Then we study, in Section  \ref{sec:gen_characteristics_study},  the stability of the stationary phase portrait with respect to small perturbations on the equilibrium electric field. We prove local stability estimates in Section \ref{sec:loc_stab_est}  and conclude the proof of the non linear stability result by a continuation argument in Section \ref{sec:continuation}.

\section{The stationary \textit{sheath} for \textbf{(VP)}}\label{sec:stationary_solution}
In this section, we recall the construction of the non trivial stationary \textit{sheath} obtained in \cite{badsi_krm}. We provide additional estimates that were not in \cite{badsi_krm}.  Let us begin by introducing the concept of solutions we consider. Due to the boundary conditions, solutions to transport equations in bounded domain are rarely classical.  We shall generically consider weak solutions for the stationary Vlasov equation in the following sense.
\begin{definition} \label{def_ws_stat} Let $ \mu \in L^{1}(\RR^{+})$, $\phi_{b} \in \RR.$ Let $(f, \phi) \in L^{1}(Q) \times W^{2,\infty}(0,1).$ We say that $(f,\phi)$ is a stationary solution for the Vlasov-Poisson system \eqref{Vlasov-i}-\eqref{bc-phi} if

\begin{itemize}
    \item[a)] $f$ is a weak solution of the Vlasov equation: for all $\psi \in \mathscr{C}^{1}_{c}\big(\overline{Q}\big)$ such that $\psi_{|\Sigma^{\textnormal{out}}} = 0$ 
    \begin{align}
    \int_{Q} f(x,v) \Psi(x,v) \dd x \dd v = \int_{0}^{+\infty} \mu(v) v \psi(0,v) \dd v,
    \end{align}
where $\Psi(x,v) = v \partial_{x} \psi(x,v) - \partial_{x} \phi(x) \partial_{v} \psi(x,v)$.
\item[b)] $\phi$ verifies the boundary conditions $\phi(0) = 0$, $\phi(1) = \phi_{b}$ and it satisfies the Poisson equation: 
\begin{align}
     \:  -\lambda^2 \partial_{xx} \phi(x) = \int_{\RR} f(x,v)\dd v - n_e(\phi(x)) \textnormal{ for almost every } x \in (0,1).
\end{align}
\end{itemize}
\end{definition}
In our setting, the terminology \textit{sheath} refers to a non trivial stationary solution in the sense of the previous definition such that:
\begin{itemize}
\item $x \in [0,1] \longmapsto \displaystyle \int_{\RR} f (x,v) \dd v$ is continuous.
\item the plasma is neutral at $x = 0,$ that is $\int_{\RR} f(0,v) \dd v = n_{e}( \phi(0)).$
\item  the electric potential is monotone decreasing, concave, and it varies strongly in a neighborhood, of  size $\lambda$, of $x=1.$ 
\end{itemize}
This is what we obtain in the following.
\begin{theorem}[The stationary \textit{sheath}] \label{theorem_equilibrium}Let $\phi_{b} < 0.$ Assume \eqref{hyp_ne_1} and that the function
\begin{align}
s \in \RR^{-} \longmapsto \exp(-s) n_{e}(s) \label{hyp_ne_2} \textnormal{ is concave. }
\end{align}
 Let $\mu \in \mathscr{C}^{1}_{c}(\RR^{+}; \RR^{+})$ be such that 
\begin{align} \label{neutrality}
    \int_{0}^{+\infty} \mu(v)dv = n_e(0),
\end{align}
and
\begin{align} \label{Bohm}
    \int_{0}^{+\infty} \frac{\mu(v)}{v^2} dv < n_e'(0).
\end{align}
Then, for any $\lambda > 0$, there exists a unique weak stationary solution $(f^{\infty},\phi^{\infty}) \in \mathscr{C}^{0}\Big( [0,1] ; L^{1}(\RR) \Big) \times \mathscr{C}^{2}[0,1]$ to \textbf{(VP)} in the class of decreasing and concave potentials. Moreover the solution satisfies
\begin{itemize}
    \item[a)] 
    \begin{align}
      f^{\infty}(x,v) &= \mathbf{1}_{D^{+}}(x,v) \mu\Big( \sqrt{v^2 + 2 \phi^{\infty}(x) } \Big),  \label{f_inf} \\
       \textnormal{ where } \quad D^{+} &= \Big \lbrace (x,v) \in Q \: : \: v > \sqrt{-2 \phi^{\infty}(x) } \Big \rbrace\nonumber.  
    \end{align}
    \item[b)] There are constants $0 < \alpha \leq \beta $ which depend only $\mu, n_e$ and $\phi_{b}$ but not on $\lambda$ such that the following estimates hold
    \begin{align}
        \int_{0}^{1} \frac{\lambda^2}{2} \Big | \partial_{x} \phi^{\infty}(x) \Big |^2+ \frac{\alpha}{2} | \phi^{\infty}(x) |^2 \dd x \underset{\lambda \rightarrow 0^{+}} = \mathcal{O}(\lambda) \label{energy_estimate}, \\
         \forall x \in [0,1], \:   \vert \phi_{b} \vert \frac{\sinh(\frac{\sqrt{\beta}x}{\lambda})}{\sinh(\frac{\sqrt{\beta}}{\lambda})} \leq \vert \phi^{\infty}(x) \vert \leq \vert \phi_{b} \vert \frac{\sinh(\frac{\sqrt{\alpha}x}{\lambda})}{\sinh(\frac{\sqrt{\alpha}}{\lambda})}. \label{estimate_ponct_phi_eq}\\
         \partial_{x} \phi^{\infty}(0) < 0, \quad |\partial_{x} \phi^{\infty}(0) | \leq  \frac{ \frac{\sqrt{\alpha}}{\lambda} }{ \sinh(\frac{\sqrt{\alpha}}{\lambda})} \label{dx_phi_0}.
    \end{align}
    \item[c)] $\phi^{\infty}$ converges to zero locally uniformly on $[0,1)$ as $\lambda \longrightarrow 0^{+}.$
    \item[d)] The plasma is quasi-neutral
    \begin{align}
        \forall p \in [1, +\infty), \quad \lim_{\lambda \rightarrow 0^{+}}  \limits \Big \Vert \int_{\RR} f^{\infty}(\cdot,v) \dd v - n_e(\phi^{\infty}) \Big \Vert_{L^{p}(0,1)} = 0.
    \end{align}
\end{itemize}
\end{theorem}
\begin{proof} The existence and uniqueness of a pair $(f^{\infty},\phi^{\infty}) \in \mathscr{C}^{0}\left([0,1] ; L^{1}(\RR) \right) \times \mathscr{C}^{2}[0,1]$ with $f^{\infty}$ given by \eqref{f_inf}  and $\phi^{\infty}$ non increasing and concave follows essentially the lines of \cite{badsi_krm}. The regularity $f^{\infty} \in \mathscr{C}([0,1];L^{1}(\RR))$ follows by direct estimates using the explicit form \eqref{f_inf} and using the fact that $\mu$ is compactly supported. We thus only prove the claim b), c) and d) which are new.  The starting point in the analysis consists in noticing that  $\phi^{\infty}$ minimizes the functional given by
\begin{align}
\forall \psi \in \mathcal{C}, \quad  J(\psi) = \displaystyle \int_{0}^{1} \frac{\lambda^2}{2} | \psi'(x) |^2 + q(\psi(x)) \dd x
\end{align}
where $\mathcal{C} = \lbrace u \in H^{1}[0,1] \: : \: \phi_{b} \leq u \leq 0 \text{ in } [0,1], \: \: u(0) = 0, \: u(1) = \phi_{b} \rbrace.$
Here, $q \in \mathscr{C}^{2}[\phi_{b},0]$ is a potential that verifies
\begin{align}
\forall s \in [\phi_{b},0], \quad q'(s) = -\int_{0}^{+\infty} \frac{\mu(w)w}{\sqrt{w^2-2s}} \dd w +n_{e}(s),\\
q(0) = q'(0) = 0, \: q''(0) > 0,\\
 \phi_{b} \leq s < 0 \Longleftrightarrow \quad q'(s) < 0.
\end{align}
With such a properties of the potential $q$, the function $s \in [\phi_{b},0[ \mapsto \displaystyle \frac{q'(s)}{s} $ is  continuous and positive, and since $q'(0) = 0$ it extends by continuity at $s = 0$ with $\lim_{s \rightarrow 0^{-}} \limits \frac{q'(s)}{s} = q''(0) > 0$.  We then set $\alpha := \underset{ s \in [\phi_{b},0]  } \inf \frac{q'(s)}{s} $ and $\beta := \underset{ s \in [\phi_{b},0] } \sup \frac{q'(s)}{s} $ which are well-defined and positive numbers with $0 < \alpha \leq \beta.$  We then get by definition of the numbers $\alpha$ and $\beta$ that
\begin{align}
\forall s \in [\phi_{b},0], \quad \beta s \leq q'(s) \leq \alpha s, \label{linear_term}\\
\forall s \in [\phi_{b},0], \quad  \frac{\alpha s^2}{2} \leq q(s) \leq  \frac{\beta s^2}{2}, \label{quadratic_potential}
\end{align}
where the second inequality follows from the first inequality by integration. We are now ready to prove  claims b),c) and d).

\emph{ Proof of b).}
Let us start with  \eqref{energy_estimate}.
Using the inequality \eqref{quadratic_potential} and the fact $\phi^{\infty}$ minimizes $J$ on $\mathcal{C}$ we have
\begin{align}
 \int_{0}^{1} \frac{\lambda^2}{2} \Big | \partial_{x} \phi^{\infty}(x) \Big |^2+ \frac{\alpha}{2} | \phi^{\infty}(x) |^2 dx \leq J(\phi^{\infty}) \leq J(\psi) 
\end{align}
where $\psi \in \mathcal{C}$ is the solution to the linear elliptic equation 
\begin{align}
-\lambda^2 \partial_{xx} \psi + \beta \psi = 0 \text{ a.e  in } (0,1), \label{toto}
\end{align}
 with the boundary conditions $\psi(0) = 0$ and $\psi(1) = \phi_{b}$. The solution is given explicitly by
\begin{align}
 \forall x \in [0,1],\: \psi(x) = \phi_{b} \frac{\sinh(\frac{\sqrt{\beta}x}{\lambda})}{\sinh(\frac{\sqrt{\beta}}{\lambda})}.
\end{align}
Taking the $L^2$ inner product of the linear elliptic equation  \eqref{toto} with $\frac{\psi}{2}$ and using an integration by parts, we obtain 
\begin{align*}
&J(\psi) = \displaystyle \int_{0}^{1} -\frac{\beta  |\psi(x)|^2}{2} + q(\psi(x))\dd x  + \frac{\lambda^2}{2} \partial_{x} \psi(1) \phi_{b} \leq \\
&\frac{\lambda^2}{2} \partial_{x} \psi(1) \phi_{b} = \frac{\lambda}{2} \phi_{b}^2 \sqrt{\beta}  \frac{1}{\tanh \left(  \frac{\sqrt{\beta}}{\lambda}  \right) }.
\end{align*}
For $\lambda >0$ small enough, we have $1 \leq \frac{1}{ \tanh \left(  \frac{\sqrt{\beta}}{\lambda}  \right)} \leq \frac{3}{2}$. Therefore since $\phi_{b}$ is independent of $\lambda$ we obtain  $\frac{\lambda^2}{2} \partial_{x} \psi(1) \phi_{b} \underset{\lambda \rightarrow 0^{+}} {=} \mathcal{O}(\lambda)$ as expected.
We now prove \eqref{estimate_ponct_phi_eq}. Consider the  function $e= \phi^{\infty} - \psi$ where $\psi$ is the solution to \eqref{toto}. Then the difference verifies $e \in H^{2}(0,1) \cap H^{1}_{0}(0,1)$ and
\[
 -\lambda^2 \partial_{xx} e + \beta e \leq 0 \textnormal{ a.e in } (0,1).
\]
A maximum principle yields $e \leq 0$ everywhere in $[0,1]$ and therefore $\phi^{\infty} \leq \psi$ which is the first inequality. Arguing similarly, one gets $\phi^{\infty}(x) \geq \phi_{b} \frac{\sinh(\frac{\sqrt{\alpha}x}{\lambda})}{\sinh(\frac{\sqrt{\alpha}}{\lambda})}$ for all $x \in [0,1].$ Hence
\[
\forall x \in [0,1], \quad \phi_{b} \frac{\sinh(\frac{\sqrt{\alpha}x}{\lambda})}{\sinh(\frac{\sqrt{\alpha}}{\lambda})}  \leq \phi^{\infty}(x) \leq \phi_{b} \frac{\sinh(\frac{\sqrt{\beta}x}{\lambda})}{\sinh(\frac{\sqrt{\beta}}{\lambda})}.
\]
Then for $x \in (0,1],$
\[
 \phi_{b} \frac{\sinh(\frac{\sqrt{\alpha}x}{\lambda})}{x \sinh(\frac{\sqrt{\alpha}}{\lambda})}  \leq \frac{ \phi^{\infty}(x) }{x}  \leq \phi_{b} \frac{\sinh(\frac{\sqrt{\beta}x}{\lambda})}{x\sinh(\frac{\sqrt{\beta}}{\lambda})}.
\]
Since $\phi^{\infty}(0) = 0$ and $\phi^{\infty}$ is differentiable at $x = 0$, taking the limit as $x \rightarrow 0^{+}$ yields 
\[
\phi_{b} \frac{ \sqrt{\alpha} }{ \lambda \sinh(\frac{\sqrt{\alpha}}{\lambda})} \leq \partial_{x} \phi^{\infty}(0) \leq \phi_{b} \frac{ \sqrt{\beta} }{ \lambda \sinh(\frac{\sqrt{\beta}}{\lambda})}.
\]
This inequality shows in particular that for $\lambda > 0$ and $\phi_{b} < 0$,  $\partial_{x} \phi^{\infty}(0) < 0$ and that $| \partial_{x} \phi^{\infty}(0) | \leq \vert \phi_{b} \vert \frac{ \sqrt{\alpha} }{ \lambda \sinh(\frac{\sqrt{\alpha}}{\lambda})}.$

\emph{Proof of c).} Thanks to \eqref{estimate_ponct_phi_eq},  for every $ 0 < r <  1,$
\[
\underset{ x \in [0,r]}{ \sup } \vert \phi^{\infty}(x) \vert \leq \vert \phi_{b} \vert \frac{\sinh(\frac{\sqrt{\alpha} r }{\lambda})}{\sinh(\frac{\sqrt{\alpha}}{\lambda})}  \longrightarrow 0 \textnormal{ as } \lambda \longrightarrow 0^{+}.
\]
\emph{ Proof of d).} Note that we have the simple bound for all $p \in [1,+\infty)$ and $\psi \in [\phi_{b},0],$

\[
\: \Big \vert \int_{\RR^{+}} \frac{\mu(v) v }{\sqrt{v^2 - 2\psi}} \dd v - n_{e}( \psi) \Big \vert ^{p} \leq 2^{p-1} \left(  \Big \vert \int_{\RR^{+}} \mu(v) \dd v \Big \vert^{p} +\Big( \underset{ \psi \in [\phi_{b},0] }{\sup}  |n_{e}(\psi)| \Big)^{p}  \right).
\]
Also observe thanks to c) that for $x \in [0,1)$, $\phi^{\infty}(x) \longrightarrow 0$ as $\lambda \longrightarrow 0^{+}.$ Then, since $q'$ is continuous on $[\phi_{b},0]$, we have $q'(\phi^{\infty}(x))\longrightarrow q'(0) = 0$ as $\lambda \longrightarrow 0^{+}.$ An application of the Lebesgue dominated convergence theorem yields the expected result.
\end{proof}
\begin{remark} In \cite{badsi_krm} the solution of the Vlasov equation is constructed by the method of the characteristics: it is in fact a mild solution. 
\end{remark}
Since $\mu \in \mathscr{C}^{1}_{c}(\RR^{+})$ the inequality \eqref{Bohm} yields the necessary condition
\begin{align}
\mu(0) = \mu'(0) = 0. \label{mu_local_quadra}
\end{align}
Consequently, the equilibrium density has the regularity $f^{\infty} \in \mathscr{C}^{1}\big(Q) \cap W^{1,1}(Q)$ and its partial derivative in the vertical direction is given for $(x,v) \in Q$ by \begin{align} \label{dv_feq}
  \partial_{v}f^{\infty}(x,v) = \mathbf{1}_{D^{+}}(x,v)  \frac{v \mu'\Big( \sqrt{v^2 + 2 \phi^{\infty}(x)} \Big)}{\sqrt{v^2 + 2 \phi^{\infty}(x)} }.
\end{align}
From now on, we set ourself in the framework of Theorem \ref{theorem_equilibrium}. In particular, the two parameters $\lambda >0$ and $\phi_{b} < 0$ are supposed to be fixed and we assume that we are given a stationary sheath solution for \textbf{(VP)}. We thus discard these two parameters in the following mathematical statements unless it is specifically needed for the understanding.

\section{Linear stability} \label{sec:linear_stab}
We firstly investigate the linear stability of the equilibrium $(f^{\infty},\phi^{\infty})$ given by the Theorem \ref{theorem_equilibrium}.
For small enough $\varepsilon > 0$ we want to investigate the dynamic of the first order perturbation of the equilibrium, namely we write
\begin{equation}
\begin{cases}
f(t,x,v)   &=  f^{\infty}(x,v) +  \varepsilon h(t,x,v) + \rm{o}(\varepsilon),\\
\phi(t,x) & =  \phi^{\infty}(x) + \varepsilon U(t,x) + \rm{o}(\varepsilon),
\end{cases}
\end{equation}
where $(h,U)$ denotes the first order fluctuation and $\rm{o}(\varepsilon)$ denotes formal higher order fluctuations. Formally, the first order fluctuation verifies the linearized Vlasov-Poisson equations
\begin{align}
    \partial_{t} h + v \partial_{x} h - \partial_{x} \phi^{\infty} \partial_{v} h =  \partial_x U(t,x) \partial_{v}f^{\infty}, \quad t > 0, \quad (x,v) \in Q, \label{linear-Vlasov-i}\\
    -\lambda^2 \partial_{xx} U(t)+ n_e'(\phi^{\infty}) U(t) = \int_{\RR} h(t,\cdot,v) \dd v, \quad  t \geq 0, \quad x \in (0,1) \label{linear-Poisson}.
\end{align}
We prescribe $h(t,\cdot,\cdot)$ on $\Sigma^{\textnormal{inc}}$ as follows
\begin{align}
  \forall t > 0, \quad  h(t,0, v) = 0 \textnormal{ if } v > 0, \quad  h(t,1,v) = 0 \textnormal{ if } v < 0.
\end{align}
It means that there is no fluctuation on the incoming density of particles at the boundary. For the fluctuating potential, we also neglect the fluctuation at the boundary so that it is assumed to verify the homogeneous Dirichlet boundary conditions
\begin{align}
U(t,0) = 0, \quad U(t,1) = 0.
\end{align}
The system is supplemented with the initial condition
\begin{align} 
    h(0,x,v) = h_{0}(x,v) \label{IC}
\end{align}
where $h_{0}$ represents the initial fluctuation of density in the phase space $Q$.
We shall denote \textbf{(LVP)} the set of linearized equations \eqref{linear-Vlasov-i}-\eqref{IC}. The wellposedness study will be performed in the following Banach spaces. For $T > 0$ and $\gamma > 0$, we define the spaces
\begin{align} \label{functional_spaces}
X_{T,\gamma} := L^{\infty} \Big( [0,T] ; L^{1}(Q) \Big), \quad Y_{T,\gamma}:= L^{\infty} \Big( [0,T]; W^{2,1}(0,1)  \Big),
\end{align}
endowed with the norms defined for any couple of functions $(f,g) \in X_{T,\gamma} \times Y_{T,\gamma}$ by
\begin{align}
\| f\|_{X_{T,\gamma}} := \underset{ 0 \leq t \leq T}{ \supess }  \: e^{-\gamma t} \| f(t) \|_{L^{1}(Q)}, \quad \| \phi \|_{Y_{T,\gamma}} :=  \underset{ 0 \leq t \leq T}{ \supess }  \: e^{-\gamma t} \| \phi(t) \|_{W^{2,1}(0,1)}. \label{norms}
\end{align}
 Before going deeper in the analysis, let us recall that for a generic measurable function $f : Q \longrightarrow \RR$ we call support of $f$ and denote $\supp f$ the complement in $Q$ of the largest open set $O \subset Q$ on which $f$ is almost everywhere equal to zero. In the sequel we will use the abreviation a.e in place of almost everywhere.  For $r \geq 0$, we define the following subset of $Q$
\begin{align} \label{D_r}
D^{+}_{r} :=  \Big \lbrace (x,v) \in Q \: : \: v > \sqrt{-2\phi^{\infty}(x) + r^2} \Big \rbrace
\end{align}
and define the associated time
\begin{align} \label{T_r}
 T_{r} := \begin{cases}
 \frac{\sqrt{-2\phi_{b}}}{\vert \partial_{x} \phi^{\infty}(0)\vert} \quad \textnormal{ if } r = 0, \\
 \frac{1}{r} \quad \textnormal{ if } r > 0.
 \end{cases}
\end{align}
The number $T_{r}$ is an upper bound for the time of travel of a particle in the domain $D^{+}_{r}$.
Note that for $r = 0$ we have $D^{+}_{r} = D^{+}$ where $D^{+}$ is given in \eqref{f_inf}. We eventually define the following  constant 
\begin{align} \label{def:kappa_star}
\kappa^{\star}_{L} = - \frac{1}{T_{r} } \log\Big( \frac{ 2\| \partial_{v} f^{\infty} \|_{L^{1}(D^{+}_{r})} T_{r} }{ \lambda^2}  \Big)
\end{align}
which is a lower bound for the rate of decay of the fluctuation. Observe that if $ \frac{ 2\| \partial_{v} f^{\infty} \|_{L^{1}(D^{+}_{r})} T_{r} }{ \lambda^2}   < 1$ then $\kappa^{\star}_{L} > 0$.

The main result of this section is the following.
\begin{theorem}[Linear stability] \label{stability_linear_eq}
Let $r \geq 0$. Consider the equilibrium $(f^{\infty},\phi^{\infty})$ given by Theorem \eqref{theorem_equilibrium} where $\mu$ is such that $\supp \mu \subset (r,+\infty)$.  For all $h_{0} \in L^{1}(Q)$, the linearized system \textbf{(LVP)} admits a unique global mild-strong solution $(h,U)$ in the sense of Definition \eqref{def_mild_sol}. In addition, we have
 \begin{align}
 (h,U) \in \mathscr{C}\Big( \RR^{+}; L^{1}(Q) \Big) \times \mathscr{C}\Big( \RR^{+}; W^{2,1}(0,1) \Big)\label{regularity_in_time}
 \end{align}
and the solution satisfies
 \begin{itemize}
 \item[a)] \emph{Decay of the $L^{1}$ norm in  $Q \setminus D^{+} $}.
There is a family of non increasing Borel sets $(\mathcal{O}_{t})_{ t \geq 0 } \subset Q \setminus  D^{+}$ which become empty in finite time such that for every $t \geq 0$
 \begin{align}
 \| h(t) \|_{L^{1} \Big( Q\setminus D^{+} \Big) } =\| h_{0} \|_{L^{1}\Big( \mathcal{O}_{t} \Big)}. \label{l1_decrease}
 \end{align}

\item[b)] 
\emph{Exponential decay  of the $L^{1}$ norm in $D^{+}_{r}$}.
Provided
 \begin{align}
\frac{ 2\| \partial_{v} f^{\infty} \|_{L^{1}(D^{+}_{r})} T_{r} }{ \lambda^2} < 1 \label{stability_condition},
\end{align}
and  $\supp h_{0} \subset D^{+}_{r}$ we have for every $t \geq 0$
 \begin{align}
\supp h(t) \subset  D^{+}_{r}
 \end{align}
 and there are constants $\kappa > \kappa^{\star}_{L} > 0$ and $C \geq 0$ such that for every $t \geq 0$ 
  \begin{align} 
& \| h(t) \|_{L^{1}(D^{+}_{r})} \leq C \exp(- \kappa t),\label{exponential_decay_h}\\
 &\| \partial_{x} U(t) \|_{L^{\infty}(0,1)} \leq \frac{2 C }{\lambda^2}  \exp(-\kappa t). \label{exponential_decay_u}
 \end{align}
\end{itemize}
where $\kappa^{\star}_{L}$ is given in \eqref{def:kappa_star}.
\end{theorem}
Several comments are in order about this result.

\begin{remark}
\begin{itemize}
\item The wellposedness of the linearized equations will follow from the classical Banach-Picard fixed point theorem. In the proof, a key ingredient is the use of an elliptic estimate for the linearized Poisson equation which enables to obtain a closed estimate on the growth of the $L^{1}$ norm of $h$.
\item  As for the stability  result, we heavily rely on the method of the characteristics and a detailed study of the associated flow. The proof of stability in $ Q \setminus D^{+} $ uses the fact that the equilibrium is supported on $D^{+}$. As a matter of fact, the source term vanishes in $Q \setminus D^{+} $ for the linearized Vlasov equation \eqref{linear-Vlasov-i} and thus the $L^{1}$ norm in $Q \setminus D^{+} $ is expected to decay thanks to the absorbing boundary conditions.
\item The proof of the exponential decay in $D^{+}_{r}$ for $r \geq 0$ heavily relies on the fact the stationary electric field $-\partial_{x} \phi^{\infty}$ is uniformly positive on $[0,1]$ and thus satisfies the two exit geometric conditions \ref{egc1} and \ref{egc2} in time $T_{r}$. A consequence is that the map $t \in \RR^{+} \longmapsto \| h(t) \|_{L^{1}(D^{+}_{r})}$ will verify a delayed Grönwall inequality of the form \eqref{fafa}.
Thus the inequality \eqref{stability_condition}  is a sufficient condition for this type of delayed Grönwall type inequality to possess solutions that are bounded by exponentially decaying functions. The inequality \eqref{stability_condition} also conveys the idea that the source term in the linearized Vlasov equation must produce particles at a rate which is less than the rate of absorption of the particles at the boundary. We mention that in the case when $r > 0$ the inequality \eqref{stability_condition} is equivalent to
\[
\frac{2\| \mu ' \|_{L^{1}(r,+\infty)} }{ \lambda^2 r} < 1
\]
which is trivially satisfied if $\mu' \in L^{1}(\RR^{+})$ is either small enough or $r$ is large enough. In the case $r = 0$, the time given in \eqref{T_r} is an upper bound for the time of exit of a particle which originally starts from the point of coordinate $(0,0)$. It is an estimate of the longest time of exit of a particle in the domain $D^{+}.$ As  $\lambda \longrightarrow 0^{+}$  this bound diverges exponentially fast because the electric vanishes exponentially fast. In addition, note that the bound \eqref{dx_phi_0} on $\partial_{x} \phi^{\infty}(0)$  has an implicit dependence on $\mu$. The inequality \eqref{stability_condition} in the case $r = 0$ is more intricate to verify.
\end{itemize}
\end{remark}

\subsection{The stationary characteristics}
For the analysis, it is convenient to consider an extension of $\phi^{\infty}$ to $\RR$ in such a way its extension still denoted $\phi^{\infty}$ verifies 
\begin{align}
\phi^{\infty} \in W^{2,\infty}(\RR), \quad  \| \partial_{x} \phi^{\infty}\|_{L^{\infty}(\RR)} \leq \| \partial_{x} \phi^{\infty} \|_{L^{\infty}(0,1)}. \label{extension_1}
\end{align}
As we work in dimension $d = 1$ and $\phi^{\infty} \in \mathscr{C}^{1}([0,1])$,  such an extension is easily constructed by a first order Taylor extrapolation of $\phi$ to $\RR \setminus [0,1]$. For $(t,x,v) \in \RR \times \RR^2$, we define the characteristics which passes through the point $(x,v)$ at time $t$ as the unique solution of the ordinary differential system 
\begin{align}
\begin{cases} \label{ode-char}
\frac{\dd}{\dd s} X_{\infty}\Big(s;t,x,v \Big) = V_{\infty}\Big( s;t,x,v \Big),\\
\frac{\dd}{\dd s} V_{\infty}\Big(s;t,x,v \Big) = - \partial_{x} \phi^{\infty}\Big( X_{\infty}\Big(s;t,x,v\Big)\Big),\\
X_{\infty}\Big(t;t,x,v\Big) = x , \: V_{\infty}\Big(t;t,x,v\Big) = v.
\end{cases}
\end{align}
Since $\phi^{\infty} \in W^{2,\infty}(\RR)$, thus $\partial_{x} \phi^{\infty}$ is Lipschitz continuous. Therefore the ordinary differential system \eqref{ode-char} has a unique solution with the regularity
\[
s \longmapsto\Big( X_{\infty}\Big(s;t,x,v\Big),V_{\infty}\Big(s;t,x,v\Big) \Big)  \in \mathscr{C}^{1}(\RR) \cap W_{\textnormal{loc}}^{2,\infty}(\RR).
\]
As for the spatial regularity,  for fixed $(s,t) \in \RR^2$, the flow map 
\begin{align}
\mathcal{F}^{\infty}_{s,t} : (x,v) \in \RR^2 \mapsto \Big( X_{\infty}\Big(s;t,x,v\Big),V_{\infty}\Big(s;t,x,v\Big) \Big) \label{diffeo}
\end{align}
is a measure preserving diffeomorphism. For $(t,x,v) \in \RR \times Q$, we may often be interested in the restriction of the the solution to the interval $\big( t_{\infty}^{\textnormal{inc}}(t,x,v), t_{\infty}^{\textnormal{out}}(t,x,v) \big)$ where
\begin{align}
t^{\textnormal{inc}}_{\infty}(t,x,v) = \inf \big \lbrace s \leq t \: : \: X_{\infty} \Big(s';t,x,v\Big) \in (0,1) \:  \forall s' \in (s,t) \big \rbrace, \label{inc_time}\\
t^{\textnormal{out}}_{\infty}(t,x,v) = \sup \big \lbrace s \geq t \: : \: X_{\infty} \Big(s';t,x,v\Big) \in (0,1) \:  \forall s' \in (t,s) \big \rbrace. \label{out_time}
\end{align}
It is the largest open interval on which the characteristic starting at time $t$ from a point   $(x,v) \in Q$ stays in $Q.$
\eqref{inc_time} (respectively  \eqref{out_time}) is called the incoming  (respectively the outgoing) time. The life time of the characteristic which started at time $t$ from the point $(x,v)$ is the difference $t^{\textnormal{out}}_{\infty}(t,x,v) - t^{\textnormal{inc}}_{\infty}(t,x,v)$. Note that the definitions \eqref{inc_time}-\eqref{out_time} do not depend on the way $\phi$ is extended on $\RR \setminus [0,1].$
The differential system \eqref{ode-char} being autonomous, one readily justifies by a uniqueness argument of the solution  that
\begin{align} \label{t_inc_t_0}
    t^{\textnormal{inc}}_{\infty}(t,x,v) = t + t^{\textnormal{inc}}_{\infty}(0,x,v), \quad t^{\textnormal{out}}_{\infty}(t,x,v) = t + t^{\textnormal{out}}_{\infty}(0,x,v).
\end{align}
Moreover, the characteristics verify the conservation of the energy
\begin{align}
    \forall s \in \RR,  \quad \frac{\dd }{\dd s} \Big( \frac{1}{2} V_{\infty}^{2}\Big( s;t,x,v \Big) + \phi^{\infty} \Big( X_{\infty}\Big(s;t,x,v \Big)\Big)  \Big ) = 0. \label{first_integral}
\end{align}
We shall study more in details the characteristics in Section \ref{sec:study_of_the_stationary_characteristics}.

To solve the Vlasov equation \eqref{linear-Vlasov-i}, we will use its mild-formulation which consists in integrating backward in time the equation along the characteristics. In doing so, there is two cases : either the characteristics has crossed $\Sigma^{\textnormal{inc}}$ at some positive time or it has crossed the time axis $\lbrace t = 0 \rbrace$. 

We may denote and define when $t_{\infty}^{\textnormal{inc}}(t,x,v) > - \infty,$
\begin{align}
&X_{\infty}^{\textnormal{inc}}(t,x,v) = X_{\infty}\Big(  t^{\textnormal{inc}}_{\infty}(t,x,v); t,x,v \Big) \in \lbrace 0,1 \rbrace, \\
&V_{\infty}^{\textnormal{inc}}(t,x,v) = V_{\infty}\Big(  t^{\textnormal{inc}}_{\infty}(t,x,v); t,x,v \Big).
\end{align}
We are now ready to define the concept of solutions we consider for \textbf{(LVP)}.
\begin{definition}[Global mild-strong solution] \label{def_mild_sol} Let $h_{0} \in L^{1}(Q).$ We say that a couple $(h,U)$ is a global mild-strong solution to \textbf{(LVP)} if there exists $\gamma > 0$ such that for every $T > 0$
\begin{itemize}
\item[a)] $(h,U) \in  X_{T, \gamma} \times Y_{T, \gamma}$ where $X_{T,\gamma}$ and $Y_{T,\gamma}$ are defined in \eqref{functional_spaces}.
\item[b)] $h$ is a mild solution of the Vlasov equation, in the sense that it satisfies for a.e $(t,x,v) \in [0,T] \times Q,$
\begin{align}
&h(t,x,v) =   \mathbf{1}_{t^{\textnormal{inc}}_{\infty}(t,x,v) \leq 0 } h_{0} \Big( X_{\infty}\Big(0;t,x,v \Big), V_{\infty}\Big(0;t,x,v\Big) \Big)  \label{def_f_linear} \\
    &+ \mathcal{G}(x,v)  \int_{0}^{t}  \mathbf{1}_{t_{\infty}^{\textnormal{inc}}(t,x,v) < s } \partial_{x} U \Big(s, X_{\infty}\Big(s;t,x,v \Big) \Big)  V_{\infty}\Big( s;t,x,v \Big) \dd s\nonumber  
\end{align}
where the function $\mathcal{G}$ is given by
\begin{align}
    \mathcal{G}(x,v) = \mathbf{1}_{D^{+}}(x,v)  \frac{\mu'\Big( \sqrt{v^2 + 2 \phi^{\infty}(x)} \Big)}{\sqrt{v^2 + 2 \phi^{\infty}(x)}},
\end{align}
and verifies for $(x,v) \in Q,$
\begin{align} \label{f_inf_and_G}
 \mathcal{G}(x,v) v = \partial_{v} f^{\infty}(x,v).
\end{align}

\item[c)]  For a.e $0 \leq t  \leq T$, $U(t)$ is a strong solution to
\begin{align}
-\lambda^2 \partial_{xx} U(t) + n_e'(\phi^{\infty}) U(t) = \int_{\RR} h(t,\cdot,v) \dd v \label{Poisson_f}
\end{align}
and satisfies the homogeneous Dirichlet boundary conditions $U(t)(0) = 0$ and $U(t)(1) = 0.$
\end{itemize}
\end{definition}
The fact that the function $\mathcal{G}$ is outside the integral term is due to the conservation of the energy \eqref{first_integral} and the fact that the domain $D^{+}$ is left invariant by the stationary flow as we prove in Lemma \ref{invariant_region}.
For the stability analysis, we need to introduce the concept of exit geometric conditions. The aim is to quantify how long a particle stays into a given subset of the phase space. We will used this concept to estimate the time spent by all the particles in the phase space before they reach the boundary. The two following definitions are not new since they are borrowed from \cite{HK-Glass-Moussa} . We thus define.

\begin{definition}[Internal Exit Geometric Condition] \label{egc1}Let $K \subset Q$ and $J$ a subinterval of $\RR_{+} $. We say that the equilibrium electric field $-\partial_{x} \phi^{\infty}$ verifies the first exit geometric condition in time $T$  with respect to $K$ on $J$ if
\begin{align}
\underset{ (t,x,v) \in J \times K } \sup( t_{\infty}^{\textnormal{out}}(t,x,v) - t)  \leq T, \\
\forall (t,x,v) \in J \times K, \quad (X^{\textnormal{out}}_{\infty},V^{\textnormal{out}}_{\infty} )(t,x,v) \in \Sigma^{\textnormal{out}} .\label{first_exit_geometric_condition}
\end{align}
\end{definition}

\begin{definition}[Initial Exit Geometric Condition] \label{egc2}Let $K \subset Q$. We say that the equilibrium electric field $- \partial_{x} \phi^{\infty}$ verifies the initial exit geometric condition in time $T$ with respect to $K$ if
\begin{align}
\underset{ (x,v) \in K} \sup t_{\infty}^{\textnormal{out}}(0,x,v) \leq T , \label{second_exit_geometric_condition}\\
\forall (x,v) \in K, \quad  (X^{\textnormal{out}}_{\infty}, V^{\textnormal{out}}_{\infty})(0,x,v) \in \Sigma^{\textnormal{out}} .
\end{align}
\end{definition}
\subsection{ The linear elliptic estimates} \label{sec:ll_est}
Let  $\rho \in L^{1}(0,1)$.  In this section we provide $W^{1,\infty}$ estimates for the solution of the linear elliptic problem 
\begin{equation}\label{problem_p}
\begin{cases}
-\lambda^2 \partial_{xx} V + n_e'(\phi^{\infty}) V  = \rho \text{ a.e in } (0,1),\\
V(0) = 0, \quad V(1) = 0.
\end{cases}
\end{equation}
where we recall that $\phi^{\infty} \in \mathscr{C}^{2}[0,1]$ and $n_{e}' $ is continuous and positive on $\RR$. Therefore $n_{e}' \circ \phi$ is a continuous function in the compact interval $[0,1]$ so it is in $L^{\infty}(0,1)$ and we have
\begin{align}
\forall x \in [0,1], \quad n_{e}'(\phi^{\infty}(x)) \geq \underset{u \in [0,1]}\min n_{e}'(\phi^{\infty}(u)) > 0.
\end{align}

We then deduce the following.
\begin{proposition}[$W^{1,\infty}$-estimates] \label{estimate_Linf_E}
Let  $\rho \in L^{1}(0,1)$. Then the problem \eqref{problem_p}  admits a unique strong solution $V \in W^{2,1}(0,1)$.  In addition, the solution satisfies the estimates
\begin{align}
	& \| \partial_{xx} V \|_{L^{1}(0,1)} \leq \frac{2 \| \rho \|_{L^{1}(0,1)}}{\lambda^2}, \label{dxxV}\\
	&\| \partial_{x} V \|_{L^{\infty}(0,1)} \leq \frac{2 \| \rho \|_{L^{1}(0,1)}}{\lambda^2},\label{dxV}\\
    	& \|(n_e' \circ \phi^{\infty}) V \|_{L^1(0,1)} \leq \| \rho \|_{L^1(0,1)}.\label{V}
\end{align}

\begin{proof} 
In one dimension, the existence and uniqueness in $W^{2,1}(0,1)$ of a strong solution to \eqref{problem_p} follows directly from the Lax-Milgram theorem and the elliptic regularity theory. We only establish the continuity in $H^{1}_{0}(0,1)$ of the linear form $\varphi \in H^{1}_{0}(0,1) \longmapsto \int_{0}^{1} \varphi(x) \rho(x) \dd x$. One has using a Hölder inequality that
\[
\forall \varphi \in H^{1}_{0}(0,1), \quad \left \vert  \int_{0}^{1} \varphi(x) \rho(x) \dd x \right \vert \leq \| \varphi \|_{L^{\infty}(0,1)} \| \rho \|_{L^{1}(0,1)} \leq \| \varphi \|_{H^{1}_{0}(0,1)}  \| \rho  \|_{L^{1}(0,1)},
\]
where we used the continuous imbedding of $H^{1}_{0}(0,1)$ in $L^{\infty}(0,1)$ to obtain the last inequality.
We now focus on the estimates.
\emph{ First estimate.} We obtain directly from the equation \eqref{problem_p} that
\[
\|\partial_{xx} V \|_{L^{1}(0,1)} \leq \frac{1}{\lambda^2} \| \rho - (n_{e}' \circ \phi^{\infty}) V\|_{L^1} \leq \frac{ \| \rho \|_{L^1} + \| (n_e' \circ \phi^{\infty} )V \|_{L^1}}{\lambda^2}.
\]
So it remains to estimate  the $L^1$-norm of $(n_e' \circ \phi^{\infty})V.$ Let us proceed by regularization of the sign function. Let $(\varphi_{n})_{n \in \NN} \subset \mathscr{C}^{1}(\RR)$ a sequence of functions such that for all $n \in \NN:$
\begin{align*}
    \varphi_n(0) = 0,\\
    \forall  u \in \RR, \varphi_{n}'(u) > 0, \:  |\varphi_n(u)| \leq 1\\
    \forall u \neq 0, \: \varphi_{n}(u) \longrightarrow \textnormal{sgn}(u) \textnormal{ as } n \rightarrow +\infty.
\end{align*}
For each $n \in \NN$, the function $\varphi_{n} \circ V$ belongs to $W^{1,1}(0,1)$ and it verifies $\varphi_{n} \circ V(0) = \varphi_{n} \circ V(1) = 0.$ We then multiply the equation \eqref{problem_p} by $\varphi_{n} \circ V$ and integrate by parts to obtain
\begin{align*}
    \int_{0}^{1} \lambda^2 |\partial_{x}V|^2 \varphi_{n}' \circ V \dd x + \int_{0}^{1} (n_e' \circ \phi^{\infty} ) V \varphi_n \circ V\dd x = \int_{0}^{1} \rho \varphi_{n} \circ V \dd x.
\end{align*}
Since $\varphi_{n}'$ is positive on $\RR$ the first term is non negative and therefore
\begin{align*}
    \int_{0}^{1} (n_e' \circ \phi^{\infty} )V \varphi_n \circ wV\dd x \leq \int_{0}^{1} \rho  \varphi_{n} \circ V  \dd x \leq  \| \rho \|_{L^1(0,1)}
\end{align*}
where the last inequality is obtained using both a Hölder inequality and the fact that $|\varphi_{n}| \leq 1$ on $\RR.$ Using the Lebesgue dominated convergence theorem, one has besides

\begin{align*}
    \int_{0}^{1} (n_e' \circ \phi^{\infty} )V \varphi_n \circ V \dd x \longrightarrow \int_{0}^{1} n_{e}' \circ \phi^{\infty} |V| \dd x .
\end{align*}
So, we get by passing to the limit in the previous inequality that
\[
\int_{0}^{1} n_{e}' \circ \phi^{\infty} |V|  \dd x \leq \| \rho \|_{L^1(0,1)}
\]
which yields \eqref{V} since $n_e'$ is positive and thus \eqref{dxxV}.

\emph{ Second estimate.} Since $V  \in  W^{2,1}(0,1)$ then  $V \in \mathscr{C}^{1}[0,1]$. Besides, \newline
$ \int_{0}^{1} \partial_{x} V(x)\dd x = 0$ because $V(0)= V(1) = 0$. By the mean value theorem there exists $x_0 \in (0,1)$ such that $\partial_{x} V(x_0) = 0.$ Then, for $x \in [0,1]$ one  has  $ \partial_{x}V(x) = \int_{x_0}^{x} \partial_{xx}V(t) \dd t$. Using a triangular inequality and the elliptic equation \eqref{problem_p}, we get for all $x \in [0,1],$
\begin{align*}
|\partial_{x}V(x)| \leq \|\partial_{xx} V \|_{L^{1}(0,1)} \leq \frac{2 \| \rho \|_{L^{1}(0,1)}}{\lambda^2}. 
\end{align*}

\end{proof}
\end{proposition}

\subsection{Study of the stationnary characteristics}\label{sec:study_of_the_stationary_characteristics}
Consider the following partition of the phase-space
\begin{align}
    Q = D^{+} \cup  D^{\pm} \cup D^{-} \cup S,
\end{align}
where $D^{+}$ is given by \eqref{f_inf},
\begin{align}
    S = \Big \lbrace (x,v) \in Q \: : v^2 + 2\phi^{\infty}(x) = 0 \Big \rbrace, \label{S}\\
    D^{\pm} = \Big \lbrace (x,v) \in Q \: :  |v| <\sqrt{-2\phi^{\infty}(x)} \Big \rbrace, \label{D_pm}\\
    D^{-} = \Big \lbrace (x,v) \in Q \: : v < -\sqrt{-2\phi^{\infty}(x)} \Big \rbrace \label{D_m}.
\end{align}
We recall here that $\phi^{\infty}$ is decreasing on $[0,1]$ with $\phi^{\infty}(0) = 0$, hence and the above sets are well defined. These sets corresponds to sub or super level sets of the microscopic energy $(x,v) \mapsto \frac{v^2}{2} + \phi^{\infty}(x)$ which is a conserved quantity along the flow \eqref{first_integral}. In this section we are interested in studying in details the characteristics, notably the invariance of the above sets and providing quasi-explicit formulas for the times of exit of the characteristics. A sketch of the phase portrait is given in Figure \eqref{phase-portrait-fig}.
\begin{figure}[!ht]
\center
\includegraphics[scale=8]{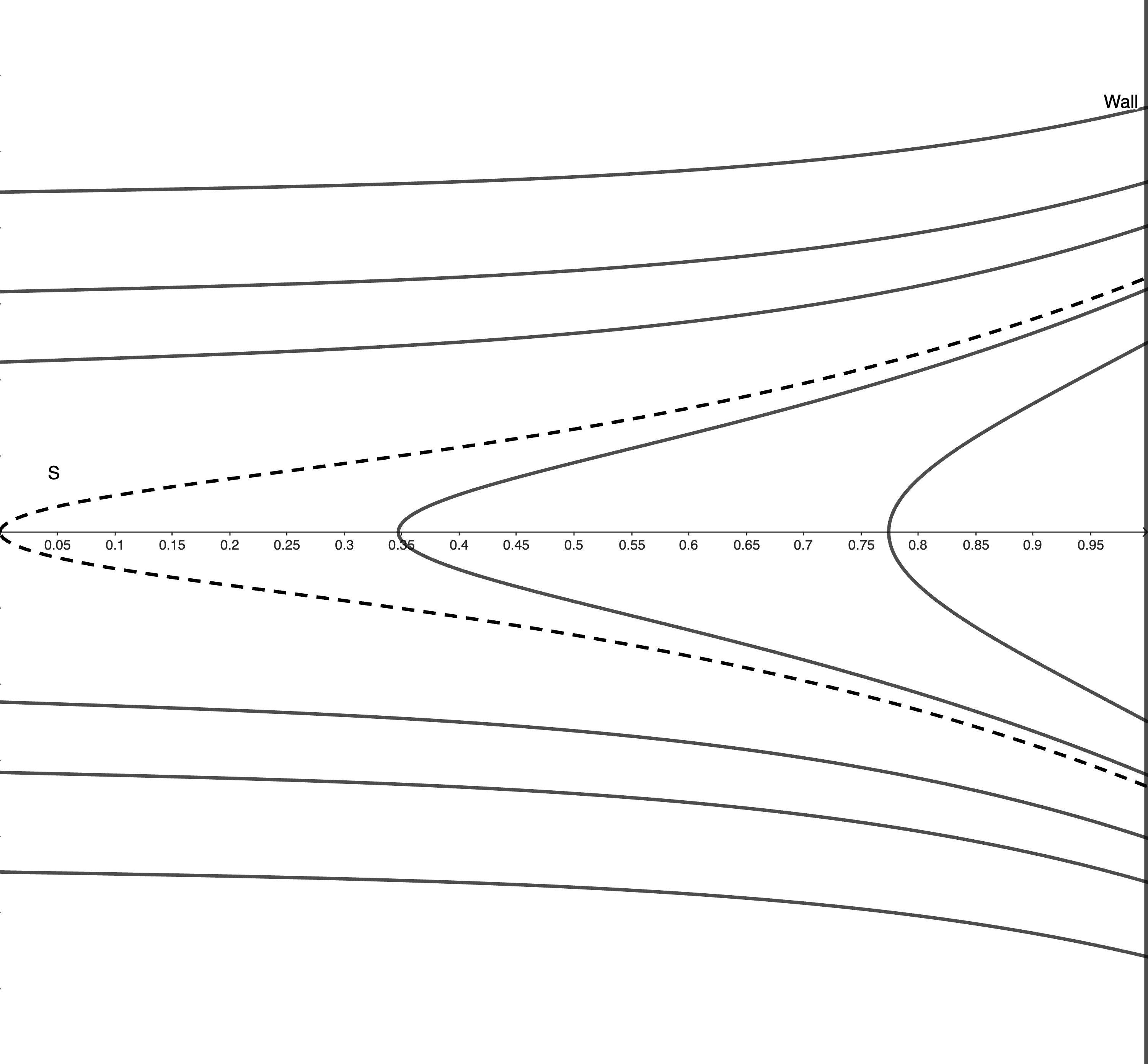}
\caption{Sketch of the stationary phase portrait. The dashed line is the curve of equation $\frac{v^2}{2} + \phi(x) = 0.$}
\label{phase-portrait-fig}
\end{figure}
As a consequence of this analysis we shall obtain.
\begin{proposition}[Exit Geometric Conditions]\label{egc_linear} Let $r \geq 0$. The equilibrium electric field  given by Theorem \ref{theorem_equilibrium} verifies the internal exit geometric condition in time $T_{r}$ with respect to $D^{+}_{r}$ on $\RR^{+}$ and the initial exit geometric condition in time $T_{r}$ with respect to $D^{+}_{r}.$
\end{proposition}
We begin with following first result which in particular states that the stationary characteristics are leaving the phase space $Q$ in finite time.
\begin{lemma}[Finiteness of the exit times]\label{finite_inc_out_time} Let $(t,x,v) \in \RR \times Q$. We have the following
\begin{itemize}
\item[a)] $t_{\infty}^{\textnormal{inc}}(t,x,v) > -\infty$ and $t_{\infty}^{\textnormal{out}}(t,x,v) < +\infty.$
\item[b)] $\textnormal{sgn}\Big(  X_{\infty}^{\textnormal{inc}}(t,x,v) - \frac{1}{2} \Big) V_{\infty}^{\textnormal{inc}}(t,x,v) \leq 0.$
\end{itemize}
\begin{proof} \emph{Proof of a).}
We only do the proof for $t_{\infty}^{\textnormal{inc}}(t,x,v)$ since the proof for $t_{\infty}^{\textnormal{out}}(t,x,v) $ is similar. Observe first that by construction, the stationary solution verifies  $\partial_{x} \phi^{\infty} < 0 $ in $[0,1]$ and $- \partial_{x} \phi^{\infty}$ is  non decreasing on $[0,1]$. We now argue by contradiction and assume that $t_{\infty}^{\textnormal{inc}}(t,x,v) = -\infty$. Then one has for all $s \in (-\infty, t]$,  $0 < X_{\infty}\big(s;t,x,v\big) < 1.$ Using the monotony of $\partial_{x} \phi^{\infty}$ we deduce from \eqref{ode-char} that for all $s \in (-\infty,t]$, $-\partial_{x} \phi^{\infty}(0) \leq  \frac{\dd}{\dd s} V_{\infty}\big(s;t,x,v \big) \leq -\partial_{x} \phi^{\infty}(1).$ By integration of the previous inequality for $s \in [u,t]$ with $u < t$ we get
\[
\partial_{x} \phi^{\infty}(1) (t-u) + v \leq V(u;t,x,v) \leq \partial_{x}\phi^{\infty}(0)(t-u) + v.
\]
Integrating again the previous inequality on $[s,t]$ for $s < t$ , we get
\[
-\partial_{x} \phi^{\infty}(0) \frac{ (t-s)^2}{2} - v(t-s) + x \leq X_{\infty}\big(s;t,x,v \big) \leq -\partial_{x} \phi^{\infty}(1) \frac{ (t-s)^2}{2} - v(t-s) + x.
\]
Note that $-\partial_{x} \phi^{\infty} > 0$ on $[0,1]$ therefore  $\lim_{s \rightarrow -\infty}  \limits-\partial_{x} \phi^{\infty}(0) \frac{ (t-s)^2}{2} - v(t-s) + x = +\infty$. We then deduce by comparison that for $|s|$ large enough we have $X_{\infty}(s;t,x,v) \geq 1$ which yields the contradiction.

\emph{Proof of b).}  By virtue of the previous point, the incoming position and velocity are well-defined. Then b) amounts to prove that if $X^{\textnormal{inc}}_{\infty}(t,x,v) = 0$ then $V_{\infty}^{\textnormal{inc}}(t,x,v) \geq 0$ and if $X^{\textnormal{inc}}_{\infty}(t,x,v) = 1$ then $V_{\infty}^{\textnormal{inc}}(t,x,v) \leq 0.$ We only treat the case when $X^{\textnormal{inc}}_{\infty}(t,x,v) = 0$ since the reasoning is similar for the other case. So suppose that $X^{\textnormal{inc}}_{\infty}(t,x,v) = 0.$ We want to prove that $V_{\infty}^{\textnormal{inc}}(t,x,v) \geq 0$. So assume for the sake of the contradiction that $V_{\infty}^{\textnormal{inc}}(t,x,v) < 0$. Then, since the solution to \eqref{ode-char} is continuous in time, for $0 < \eta < 1$ there exists $\delta > 0$ such that for all $s \in (t^{\textnormal{inc}}_{\infty}(t,x,v) -  \delta,  t^{\textnormal{inc}}_{\infty}(t,x,v) + \delta ) $
\[
V_{\infty}\big(s;t,x,v \big) < \frac{V_{\infty}^{\textnormal{inc}}\big(t,x,v\big)}{2} < 0, \quad  -\eta < X_{\infty}(s;t,x,v) < \eta.
\]
Since $\frac{\dd X_{\infty}}{\dd s}(s;t,x,v) = V_{\infty}(s;t,x,v) $, $X_{\infty}\big(\cdot;t,x,v\big)$ is decreasing on the interval\newline$(t^{\textnormal{inc}}_{\infty}(t,x,v) -  \delta, t^{\textnormal{inc}}_{\infty}(t,x,v) ].$ Therefore,
\[
t^{\textnormal{inc}}_{\infty}(t,x,v) -  \delta< s < t^{\textnormal{inc}}_{\infty}(t,x,v)  \Longrightarrow 0 < X_{\infty}\big(s;t,x,v \big) < \eta.
\]
It eventually contradicts the minimality  of the incoming time \eqref{inc_time}.
\end{proof}
\end{lemma}
We now study the regularity of the incoming time and the associated exit point.
\begin{lemma}[Regularity and characterization of the singular set] \label{regularity_lemma}Let $t \in \RR.$  
\begin{itemize}
\item [a)] The three functions $(x,v) \in Q \mapsto t^{\textnormal{inc}}_{\infty}(t,x,v)$, $(x,v) \in Q \mapsto X^{\textnormal{inc}}_{\infty}(t,x,v)$, $(x,v) \in Q \mapsto V^{\textnormal{inc}}_{\infty}(t,x,v)$ are continuous at every point $(x,v) \in Q $ such that $(X_{\infty}^{\textnormal{inc}}(t,x,v),V_{\infty}^{\textnormal{inc}}(t,x,v)) \notin \Sigma^{0}.$
\item[b)] One has for all $(x,v) \in Q,$
\[
 \Big( (X_{\infty}^{\textnormal{inc}}(t,x,v),V_{\infty}^{\textnormal{inc}}(t,x,v)) \in \Sigma^{0} \Longleftrightarrow X_{\infty}^{\textnormal{inc}}(t,x,v)= 0 \textnormal{ and } (x,v) \in S \Big).
\]
\end{itemize}
\begin{proof} \emph{Proof of a)}. A proof of continuity can be found in \cite{HK-Glass-Moussa} Lemma 3.3.  We nevertheless propose one here for the sake selfcontainedness of this document with slightly more details.  We fix $t \in \RR$. Observe firstly that we can decompose the incoming position and velocitiy as functions in the form
\begin{align}
X_{\infty}^{\textnormal{inc}}(t,\cdot,\cdot) = \big( X_{\infty,t} \circ S_{t} \big)(\cdot,\cdot)  \quad V_{\infty}^{\textnormal{inc}}(t,\cdot,\cdot) = \big( V_{\infty,t} \circ S_{t} \big) (\cdot,\cdot) \label{decomp_id}
\end{align}
with
\begin{align*}
&X_{\infty,t}  : (s,x,v) \in  \RR^{3} \longmapsto X_{\infty}(s;t,x,v),  \quad V_{\infty,t}  : (s,x,v)  \in \RR^{3} \longmapsto V_{\infty}(s;t,x,v),\\
&S_{t} : (x,v) \in \RR^{2} \longmapsto (t_{\infty}^{\textnormal{inc}}(t,x,v),x,v).
\end{align*}
Let us recall briefly why $  (X_{\infty,t},V_{\infty,t})$ is  continuous.
For $(x,v), (x',v') \in \RR^{2},$ let us set $Y_{\infty}(\cdot;t) = X_{\infty}(\cdot;t,x,v) - X_{\infty}(\cdot; t , x',v')$, $W_{\infty}(\cdot;t) = V_{\infty}(\cdot;t,x,v) - V_{\infty}(\cdot;t, x',v')$. Then integrating the differential system \eqref{ode-char} and using a triangle inequality, one gets for $s \leq t$
\[
\vert Y_{\infty}(s,t) \vert + \vert W_{\infty}(s,t) \vert \leq \vert x-x' \vert + \vert v-v' \vert + C \int_{s}^{t} \vert Y_{\infty}(\tau,t) \vert  + \vert W_{\infty}(\tau,t) \vert d\tau
\]
with $C = \max( 1 , \| \partial_{xx} \phi^{\infty} \|_{L^{\infty}(\RR)})$. A similar inequality holds if $s > t$ so that  a Gronwall Lemma then yields for $s\in \RR$
\begin{align}
\vert Y_{\infty}(s,t) \vert + \vert W_{\infty}(s,t) \vert \leq \Big( \vert x - x' \vert + \vert v - v' \vert \Big) e^{C|t-s|}. \label{first_Gronwall_estimate}
\end{align}
This estimate associated with the Cauchy-Lipschitz theorem shows that for all $t \in \RR$, the map $(X_{\infty,t},V_{\infty,t})$ is continuous. Then, thanks to the decomposition \eqref{decomp_id}, to establish the expected continuity,  it is now sufficent to prove the continuity  of the function $(x,v) \in Q \longmapsto t_{\infty}^{\textnormal{inc}}(t,x,v)$ at every point $(x,v) \in Q$ such that \newline$(X_{\infty}^{\textnormal{inc}}(t,x,v),V_{\infty}^{\textnormal{inc}}(t,x,v)) \notin \Sigma^{0}.$

We fix $(x,v) \in Q$ such that $(X_{\infty}^{\textnormal{inc}}(t,x,v),V_{\infty}^{\textnormal{inc}}(t,x,v)) \notin \Sigma^{0}$  and suppose without loss of generality that $V_{\infty}^{\textnormal{inc}}(t,x,v) > 0$, and as a consequence of Lemma \ref{finite_inc_out_time} b), that $X_{\infty}^{\textnormal{inc}}(t,x,v) = 0$. Then we have $t_{\infty}^{\textnormal{inc}}(t,x,v) > - \infty $ and by definition of the incoming time we also have $X_{\infty}(s; t,x,v) \in (0,1)$ for all $t_{\infty}^{\textnormal{inc}}(t,x,v) < s < t $.
The proof is in three steps.
\newline
\emph{ Step 1: Localization in time.}
We are going to prove the following claim: \emph{ there exists $\varepsilon _{0}> 0$ such that for any $0 < \varepsilon < \frac{\varepsilon_{0}}{2}$, we have $X_{\infty}(s;t,x,v) < 0$ for $s \in (t_{\infty}^{\textnormal{inc}}(t,x,v) - \varepsilon_{0}, t_{\infty}^{\textnormal{inc}}(t,x,v) - \varepsilon)$.} Indeed, since $X_{\infty}(\cdot;t,x,v)$ is $\mathscr{C}^{1}$ a Taylor expansion at the point $t_{\infty}^{\textnormal{inc}}(t,x,v)$ yields for $s < t_{\infty}^{\textnormal{inc}}(t,x,v),$
$
X_{\infty}(s;t,x,v)  = (s-t_{\infty}^{\textnormal{inc}}(t,x,v)) \left( V_{\infty}^{\textnormal{inc}}(t,x,v) + \rm{o}(1)(s) \right)
$
where $\rm{o}(1) : \RR \rightarrow \RR$ is function such that \newline$\rm{o}(1)(s) \longrightarrow 0$ as $s \longrightarrow t_{\infty}^{\textnormal{inc}}(t,x,v).$
Since $ V_{\infty}^{\textnormal{inc}}(t,x,v)  > 0$, it therefore exists $\varepsilon_{0} > 0$ such that  we have $V_{\infty}^{\textnormal{inc}}(t,x,v) + \rm{o}(1) > 0$ for $ t_{\infty}^{\textnormal{inc}}(t,x,v) - \varepsilon_{0} <s < t_{\infty}^{\textnormal{inc}}(t,x,v)$ and it yields the claim.
\newline

\emph{ Step 2: Localization in space.} Let $0 < \varepsilon < \min \lbrace  \frac{\varepsilon_{0}}{2} ; t-t_{\infty}^{\textnormal{inc}}(t,x,v)  \rbrace$. Using the Gronwall estimate \eqref{first_Gronwall_estimate} at $s = t_{\infty}^{\textnormal{inc}}(t,x,v) -2 \varepsilon$ and the fact that by virtue of the previous claim $X_{\infty}\big( t_{\infty}^{\textnormal{inc}}(t,x,v) -2 \varepsilon;t,x,v) < 0$ we have that there exists $\delta \equiv \delta_{t,x,v,\varepsilon}  > 0$ such that for all $(x',v') \in Q,$
\[
| x-x'| + |v - v'| < \delta \Longrightarrow X_{\infty}\big( t_{\infty}^{\textnormal{inc}}(t,x,v) -2 \varepsilon;t,x',v') < 0.
\]
It implies that for $(x',v')$ in a ball centered at $(x,v)$ and of radius $\delta$, we have  $ t_{\infty}^{\textnormal{inc}}(t,x',v') > -\infty$ and that the associated incoming time is such that $t_{\infty}^{\textnormal{inc}}(t,x,v) -2 \varepsilon < t_{\infty}^{\textnormal{inc}}(t,x',v').$  We are now going to prove that $t_{\infty}^{\textnormal{inc}}(t,x',v') < t_{\infty}^{\textnormal{inc}}(t,x,v) +2 \varepsilon.$ Observe that by definition of $t_{\infty}^{\textnormal{inc}}(t,x,v)$ and since $\varepsilon$ is small enough we have $X_{\infty}\big(  t_{\infty}^{\textnormal{inc}}(t,x,v) +2 \varepsilon;t,x,v) \in (0,1).$ Now using again the Gronwall estimate at \eqref{first_Gronwall_estimate} $s =  t_{\infty}^{\textnormal{inc}}(t,x,v) +2 \varepsilon$,  there exists $\eta \equiv \eta_{t,x,v,\varepsilon}  > 0$ such  that for all $(x',v') \in Q,$
\[
| x-x'| + |v - v'| < \eta \Longrightarrow X_{\infty}\big( t_{\infty}^{\textnormal{inc}}(t,x,v) +2 \varepsilon;t,x',v') \in (0,1).
\]
It implies that for $(x',v')$ in a ball centered at $(x,v)$ of radius $\eta$ we have  $t_{\infty}^{\textnormal{inc}}(t,x',v') < t_{\infty}^{\textnormal{inc}}(t,x,v) +2 \varepsilon.$

\emph{ Conclusion.} We have proven that for $0 < \varepsilon < \min \lbrace  \frac{\varepsilon_{0}}{2} ; t-t_{\infty}^{\textnormal{inc}}(t,x,v)  \rbrace$ there exists $r = \min\lbrace \delta_{t,x,v,\varepsilon}, \eta_{t,x,v,\varepsilon}  \rbrace > 0$ such that for all $(x',v') \in Q$
\[
| x - x'| + |v - v'| < r \Longrightarrow | t_{\infty}^{\textnormal{inc}}(t,x',v') - t_{\infty}^{\textnormal{inc}}(t,x,v) | < 2 \varepsilon. 
\]
It shows that the map $(x,v) \in Q \mapsto t_{\infty}^{\textnormal{inc}}(t,x,v)$ is continuous at any point $(x,v) \in Q$ such that $(X_{\infty}^{\textnormal{inc}}(t,x,v),V_{\infty}^{\textnormal{inc}}(t,x,v)) \notin \Sigma^{0}.$ 

\emph{ Proof of b).} We show the necessary condition. Fix $t \in \RR$ and $(x,v) \in Q$ such that $(X_{\infty}^{\textnormal{inc}}(t,x,v),V_{\infty}^{\textnormal{inc}}(t,x,v)) \in \Sigma^{0}$. Then by conservation of the energy \eqref{first_integral} we have
\[
\phi^{\infty}\Big( X_{\infty}^{\textnormal{inc}}(t,x,v)  \Big) + \frac{V_{\infty}^{\textnormal{inc}}(t,x,v)^2}{2} = \phi^{\infty}(x) + \frac{v^2}{2}.
\]
Note that the case $( X_{\infty}^{\textnormal{inc}}(t,x,v),  V_{\infty}^{\textnormal{inc}}(t,x,v)) = (0,0)$ readily yields $(x,v) \in S,$ while the other case $( X_{\infty}^{\textnormal{inc}}(t,x,v),  V_{\infty}^{\textnormal{inc}}(t,x,v)) = (1,0)$  yields $\phi^{\infty}(1) - \phi^{\infty}(x) = \frac{v^2}{2}.$ Since $\phi^{\infty}$ is decreasing  on $[0,1]$ one has $\frac{v^2}{2} = \phi^{\infty}(1) - \phi^{\infty}(x) < 0$. It yields a contradiction. The sufficient condition is trivial by conservation of energy.
\end{proof}
\end{lemma}
We now study the regions that remain invariant by \eqref{diffeo} and give the associated formula for the incoming and outgoing times.
Note that by \eqref{t_inc_t_0} we have
\[
\forall (t,x,v) \in \RR \times Q, \: t_{\infty}^{\textnormal{inc}}(t,x,v) = t +  t_{\infty}^{\textnormal{inc}}(0,x,v), \quad t_{\infty}^{\textnormal{out}}(t,x,v) = t +  t_{\infty}^{\textnormal{out}}(0,x,v),
\]
which encodes the fact that for an autonomous differential equation the dynamic is invariant by translation in time. So it is sufficient to study the characteristics which started at time $t = 0.$
We prove the following.
\begin{lemma}[Invariant regions]\label{invariant_region} Let $r \geq 0$.  The sets $D^{+}_{r}$, $D^{-}$,$D^{\pm}$ and $S$ are left invariant by the flow \eqref{diffeo}. More precisely, for any set $E \in \lbrace D^{+}_{r},D^{-},D^{\pm}, S \rbrace$ we have
\begin{align}
     \quad (x,v) \in E \Longrightarrow \forall s \in \big ( t_{\infty}^{\textnormal{inc}}(0,x,v), t_{\infty}^{\textnormal{out}}(0,x,v) \big), \: (X_{\infty},V_{\infty})(s;0,x,v) \in E.
\end{align}
\begin{proof} 
\emph{Invariance in $D^{+}_{r}$}. Let $(x,v) \in D^{+}_{r}.$ By conservation of the energy  \eqref{first_integral} one has for all $s \in \big( t_{\infty}^{\textnormal{inc}}(0,x,v), t_{\infty}^{\textnormal{out}}(0,x,v) \big) $, 
\[
V_{\infty}^2\big(s;0,x,v\big) = v^2 + 2\phi^{\infty}(x) -2 \phi^{\infty}\Big( X_{\infty}\big(s;0,x, v\big) \Big)
\] 
and since $v^2 + 2\phi^{\infty}(x) > r^2$ we obtain  
\begin{align*}
&| V_{\infty}(s;0,x,v) | = \sqrt{v^2 + 2(\phi^{\infty}(x) -\phi^{\infty}(X_{\infty}(s;0,x,v)) )} >\\
& \sqrt{r^2 -2\phi^{\infty}(X_{\infty}(s;0,x,v))}.
\end{align*}
Let us now prove that $V_{\infty}\big(s;0,x,v\big)$ is positive. Since $\phi^{\infty}$ is decreasing and $\mathscr{C}^{1}$, by \eqref{ode-char} we deduce that $s \mapsto V_{\infty}\big(s;0,x,v\big)$ is increasing. Therefore 
if $
s \geq t_{\infty}^{\textnormal{inc}}(0,x,v)$ then $
V_{\infty}\big(s;0,x,v\big) \geq V_{\infty}^{\textnormal{inc}}(0,x,v).
$
We now claim that $V_{\infty}^{\textnormal{inc}}(0,x,v) \geq 0.$ For the sake of the contradiction, let us assume the opposite. Since $V_{\infty}\big(t;0,x,v\big) = v$ and $v > 0$, and the function $V_{\infty}\big( \cdot ; t,x,v \big)$ is continuous on its interval of definition, there exists $t^{0} \in (t_{\infty}^{\textnormal{inc}}(t,x,v),t)$ such that $V_{\infty}\big(t^{0};0,x,v \big) = 0.$ Using the conservation of the energy, observe that $\phi^{\infty}(X_{\infty}\big(t^{0};0,x,v\big))  = \frac{v^2}{2} + \phi^{\infty}(x) $ and the right hand side is positive since $(x,v) \in D^{+}_{r}$. It contradicts the negativity of $\phi^{\infty}.$ Hence $ V_{\infty}^{\textnormal{inc}}(0,x,v)\geq 0$ and thus $|V_{\infty}\big(s;0,x,v\big)| = V_{\infty}\big(s;0,x,v \big) > \sqrt{-2\phi^{\infty}\Big(X_{\infty}\big(s;0,x,v \big)\Big) + r^2}.$ It proves the claim for the set $D^{+}_{r}.$ 

\emph{Invariance in $D^{-}.$}
Consider the symmetry $\mathcal{R} : (x,v) \in \RR^2 \mapsto (x,-v)$ and observe that $\mathcal{R}(D^{+}) = D^{-}.$ Arguing the uniqueness of the characteristics, one has the identiy
\begin{equation}\label{symmetry_comp_flow}
\mathcal{R} \circ \mathcal{F}_{-s,-t}^{\infty} \circ \mathcal{R}  = \mathcal{F}^{\infty}_{s,t}
\end{equation}
where $\mathcal{F}_{s,t}^{\infty}$ is the flow defined in \eqref{diffeo}.
Let $(x,v) \in D^{-}$.  One has $\mathcal{R}(x,v) \in D^{+}$ and $\mathcal{F}_{-s,-t}^{\infty}(\mathcal{R}(x,v)) \in D^{+}$ for $s \in \big( t_{\infty}^{\textnormal{inc}}(0,x,v), t_{\infty}^{\textnormal{out}}(0,x,v) \big)$ by symmetry and invariance in $D^{+}.$ It implies
$\mathcal{R} \circ \mathcal{F}_{-s,t}^{\infty} \circ \mathcal{R} (x,v) \in \mathcal{R}(D^{+}) = D^{-}.$ Using the identity \eqref{symmetry_comp_flow} it yields $\mathcal{F}_{s,t}^{\infty}(x,v) \in D^{-}$ which is the expected claim.

\emph{Invariance in $D^{\pm}$.}
Let $(x,v) \in D^{\pm}.$ By conservation of the energy one has for all $s \in \big( t_{\infty}^{\textnormal{inc}}(0,x,v), t_{\infty}^{\textnormal{out}}(0,x,v) \big)$, $\frac{V_{\infty}^{2}(s;0,x,v)}{2} + \phi^{\infty}(X_{\infty}(s;0,x,v)) = \frac{v^2}{2} + \phi^{\infty}(x) <  0$ which is the expected inequality.
Eventually the invariance in $S$ is trivial. The proof is achieved.
\end{proof}
\end{lemma}
We now provide  quasi-explicit formula for the incoming and the outgoing time defined in \eqref{inc_time},\eqref{out_time}.  
\begin{lemma}[Explicit formulas for the incoming and the outgoing time] \label{explicit_times}
One has the following  formulas for the incoming and the outgoing times in $Q \setminus S,$
\begin{align} \label{tinc_explicit}
&t_{\infty}^{\textnormal{inc}}(0,x,v) = \begin{cases}  - \int_{0}^{x} \frac{\dd u}{\sqrt{v^2 + 2( \phi^{\infty}(x) - \phi^{\infty}(u))}} & (x,v) \in D^{+}
\\
-\int_{x_0}^{x} \frac{\dd u}{\sqrt{v^2 + 2(\phi^{\infty}(x)-\phi^{\infty}(u))}} - \int_{x_0}^{1} \frac{\dd u}{\sqrt{2( \phi^{\infty}(x_0)- \phi^{\infty}(u))}}, & (x,v) \in D^{\pm}, v > 0, \\
-\int_{x}^{1} \frac{\dd u}{\sqrt{v^2 + 2 (\phi^{\infty}(x)-\phi^{\infty}(u))}}, & (x,v) \in D^{\pm}, v \leq 0,\\
- \int_{x}^{1} \frac{\dd u}{\sqrt{v^2 + 2( \phi^{\infty}(x) - \phi^{\infty}(u))}} & (x,v) \in D^{-},
\end{cases}
\end{align}
\begin{align} \label{tout_explicit}
&t_{\infty}^{\textnormal{out}}(0,x,v) = \begin{cases}   \int_{x}^{1} \frac{\dd u}{\sqrt{v^2 + 2( \phi^{\infty}(x) - \phi^{\infty}(u))}} & (x,v) \in D^{+},
\\
\int_{x_0}^{x} \frac{\dd u}{\sqrt{v^2 + 2(\phi^{\infty}(x)-\phi^{\infty}(u))}} + \int_{x_0}^{1} \frac{\dd u}{\sqrt{2( \phi^{\infty}(x_0)- \phi^{\infty}(u))}}, & (x,v) \in D^{\pm}, v < 0, \\
\int_{x}^{1} \frac{\dd u}{\sqrt{v^2 + 2 (\phi^{\infty}(x)-\phi^{\infty}(u))}}, & (x,v) \in D^{\pm}, v \geq 0\\
\int_{0}^{x} \frac{\dd u}{\sqrt{v^2 + 2( \phi^{\infty}(x) - \phi^{\infty}(u))}} & (x,v) \in D^{-},
\end{cases}
\end{align}
where $x_0 =  \big(\phi^{\infty}\big)^{-1}\big( \frac{v^2}{2} + \phi^{\infty}(x) \big)$. 
\begin{proof} 
\emph{ Formula in $D^{+} \cup D^{-}$.}
Let $(x,v) \in D^{+}.$ \newline
Set $I(x,v) := (t_{\infty}^{\textnormal{inc}}(0,x,v), t_{\infty}^{\textnormal{out}}(0,x,v))$. Observe that by the Lemma \ref{invariant_region}, for all $s \in I(x,v)$, $V_{\infty}\big(s;0,x,v \big) > 0.$ Then, the function $s \in \overline{I(x,v)}\mapsto X_{\infty}(s;0,x,v)$ is monotone increasing and $\mathscr{C}^{1}$ on $\overline{I(x,v)}$. Therefore $X_{\infty}(\cdot;0,x,v)$ is onto and into from $\overline{I(x,v)}$ to $[0,1]$. By the bijection theorem, it has a unique inverse denoted $\tau : [0,1] \rightarrow \overline{I}(x,v)$ which is continuous and verifies
\begin{align} \label{inverse_relation}
    \forall u \in (0,1), \:  X_{\infty}(\tau(u); 0,x,v) = u.
\end{align} Since $X_{\infty}(\cdot;0,x,v)$ is $\mathscr{C}^{1}$ we deduce from  the global inverse mapping theorem that the function $\tau$ is $\mathscr{C}^{1}(0,1)$. One therefore differentiates \eqref{inverse_relation} to get the Cauchy problem
\[
\begin{cases}

\forall u \in (0,1), \: \frac{\dd}{\dd u} \tau(u)  = \frac{1}{V_{\infty}(\tau(u); 0, x,v)},\\
\tau(x) = 0.
\end{cases}
\]
Invoking the conservation of the energy and the fact that $V_{\infty}(\cdot;0,x,v)$ is positive, it yields $V_{\infty}(\tau(u);0,x,v) = \sqrt{v^2 + 2(\phi^{\infty}(x)-\phi^{\infty}(u))}.$ By definition, $\tau(0) =t_{\infty}^{\textnormal{inc}}(0,x,v)$ and $\tau(1) = t_{\infty}^{\textnormal{out}}(0,x,v)$. Since $s \in \RR \mapsto V_{\infty}(s;0,x,v)$ is continuous and $\tau \in \mathscr{C}^{0}[0,1],$ the function $\tau$ is also $\mathscr{C}^{1}[0,1]$. Integrating the Cauchy problem both on $(0,x)$ and on $(x,1)$  yields the expected formulas for $t_{\infty}^{\textnormal{inc}}(0,x,v)$ and $t_{\infty}^{\textnormal{out}}(0,x,v).$ A similar reasoning yields the formulas in $D^{-}.$

\emph{Formula in $D^{\pm}$.} Let $(x,v) \in D^{\pm}.$ Set $I(x,v) := (t_{\infty}^{\textnormal{inc}}(0,x,v), t_{\infty}^{\textnormal{out}}(0,x,v))$.  There is two cases.
If $v > 0$ then by a standard continuity and monotony argument there exists a unique $t^{0}(x,v) \in (t_{\infty}^{\textnormal{inc}}(0,x,v), 0) $ such that $V_{\infty}(t^{0}(x,v);0,x,v) = 0.$ Then set $x_{0} = X_{\infty}(t^{0}(x,v);0,x,v).$ Using the conservation of the energy and the fact that $\phi^{\infty}$ is a bijection from  $[0,1]$ to $[\phi_{b},0]$ one has
$x_0  = \big(\phi^{\infty}\big)^{-1}\big( \frac{v^2}{2} + \phi^{\infty}(x) \big).$ Following backward in time the characteristic  which started at time $t= 0$ from the point of coordinate $(x,v)$, one decomposes the incoming time as follows,
\[
t_{\infty}^{\textnormal{inc}}(0;x,v) = t^{0}(x,v) + t_{\infty}^{\textnormal{inc}}(t^{0}(x,v);x_0,0).
\]
Observe that for $t > t^{0}(x,v)$, $V_{\infty}(t; 0,x,v) > 0$ and for $t < t^{0}(x,v), V_{\infty}(t; 0,x,v) < 0.$ One may again use the global inversion theorem on each branch of the characteristic (for $t > t^{0}(x,v)$ and for $t < t^{0}(x,v)$)  to obtain
\begin{align}
&t^{0}(x,v) = -\int_{x_0}^{x} \frac{\dd u}{\sqrt{v^2 + 2(\phi^{\infty}(x)-\phi^{\infty}(u))}},\\
&t_{\infty}^{\textnormal{inc}}(t^{0}(x,v);x_0,0) = -\int_{x_0}^{1} \frac{\dd u}{\sqrt{v^2 + 2 (\phi^{\infty}(x)-\phi^{\infty}(u))}}
\end{align}
where the singularity at $ u = x_{0}$ is integrable because one has by concavity and monotony of $\phi^{\infty}$, for all $u \in [0,1]$, $ \phi^{\infty}(u) \leq \phi^{\infty}(x_0) + (u-x_{0}) \partial_{x} \phi^{\infty}(x_{0})$ with $\partial_{x} \phi^{\infty}(x_{0}) < 0.$
Combining both expressions, one obtains the expected formula for $t_{\infty}^{\textnormal{inc}}(0,x,v).$ To get the formula for $t_{\infty}^{\textnormal{out}}(0,x,v)$, it suffices to follow the same characteristic but forward in time.
Eventually, the case $v \leq 0$, is treated analogously to the case when $(x,v) \in D^{-}.$
\end{proof}
\end{lemma}
As a corollary of the above explicit formulas, we obtain the following bounds which depend on the equilibrium potential $\phi^{\infty}.$ \begin{corollary}[Bound on the incoming and outgoing times] \label{estimate_incoming_time} 
We have the following bounds
\begin{itemize}
\item[a)] For every $r > 0$ we have the bounds in $D^{+}_{r}$,
\begin{align}
-\frac{1}{r} \leq  \underset{ (x,v) \in D^{+}_{r}}{\inf}  t_{\infty}^{\textnormal{inc}}(0,x,v),  \quad \underset{ (x,v) \in D^{+}_{r}}{\sup}  t_{\infty}^{\textnormal{out}}(0,x,v)  \leq  \frac{1}{r}. \label{estimate_on_D_r}
\end{align}
\item[b)] More generally, we have for all $(x,v) \in Q \setminus S$
\begin{align}
 \left \vert t_{\infty}^{\textnormal{inc}}(0,x,v) \right \vert \leq t^{-}(x,v), \quad \left \vert t_{\infty}^{\textnormal{out}}(0,x,v) \right \vert \leq t^{+}(x,v),
\end{align}
where 
\begin{equation}
t^{-}(x,v) = \begin{cases}
\frac{\sqrt{-2\phi^{\infty}(x)}}{\left \vert \partial_{x} \phi^{\infty}(0) \right \vert} & (x,v) \in D^{+},\\
\frac{\sqrt{2(\phi^{\infty}(x_0)-\phi^{\infty}(x))} }{\left \vert \partial_{x} \phi^{\infty}(x_0) \right \vert} + \frac{\sqrt{2(\phi^{\infty}(x_0)-\phi^{\infty}(1))} }{\left \vert \partial_{x} \phi^{\infty}(x_0) \right \vert}&  (x,v) \in D^{\pm}, v > 0,\\
\frac{\sqrt{2(\phi^{\infty}(x)-\phi^{\infty}(1))} }{\left \vert \partial_{x} \phi^{\infty}(x) \right \vert}  & (x,v) \in D^{\pm}, v \leq 0,\\
\frac{\sqrt{-2\phi^{\infty}(1)} - \sqrt{2 \phi^{\infty}(x)}}{\left \vert \partial_{x} \phi^{\infty}(0) \right \vert} &  (x,v) \in D^{-},\\
\end{cases}
\end{equation}
\begin{equation}
t^{+}(x,v) = \begin{cases}
\frac{\sqrt{-2\phi^{\infty}(1)} - \sqrt{2 \phi^{\infty}(x)}}{\left \vert \partial_{x} \phi^{\infty}(0) \right \vert}  & (x,v) \in D^{+},\\
\frac{\sqrt{2(\phi^{\infty}(x_0)-\phi^{\infty}(x))} }{\left \vert \partial_{x} \phi^{\infty}(x_0) \right \vert} + \frac{\sqrt{2(\phi^{\infty}(x_0)-\phi^{\infty}(1))} }{\left \vert \partial_{x} \phi^{\infty}(x_0) \right \vert}&  (x,v) \in D^{\pm}, v < 0,\\
\frac{\sqrt{2(\phi^{\infty}(x)-\phi^{\infty}(1))} }{\left \vert \partial_{x} \phi^{\infty}(x) \right \vert}  & (x,v) \in D^{\pm}, v \geq 0,\\
\frac{\sqrt{-2\phi^{\infty}(x)}}{\left \vert \partial_{x} \phi^{\infty}(0) \right \vert}  &  (x,v) \in D^{-}.\\
\end{cases}
\end{equation}

with $x_0 =  \big(\phi^{\infty}\big)^{-1}\big( \frac{v^2}{2} + \phi^{\infty}(x) \big)$ are bounded functions. Moreover we have the uniform bounds in $D^{+}$
\begin{align}
\underset{ (x,v) \in D^{+} } {\sup }\big | t^{-} (x,v) \big |   \leq \frac{\sqrt{-2\phi_{b}}}{\vert \partial_{x} \phi^{\infty}(0) \vert}, \quad \underset{ (x,v) \in D^{+}} {\sup }\big | t^{+} (x,v) \big |   \leq   \frac{\sqrt{-2\phi_{b}}}{\vert \partial_{x} \phi^{\infty}(0) \vert} \label{bound_time}
\end{align}
where we recall that $\phi_{b} < 0 $ is independent of $\lambda > 0$ and $\partial_{x} \phi^{\infty}(0) < 0$ thanks to $\eqref{estimate_ponct_phi_eq}.$ 

\end{itemize}
\begin{proof}
\emph{Proof of a)} Let $r  > 0$ and $(x,v) \in D^{+}_{r}$ where $D^{+}_{r}$ is defined in \eqref{D_r}. Then we have $v^2 + 2 \phi^{\infty}(x)  > r^2.$  Since $\phi$ is non positive we obtain
\begin{align*}
t^{\textnormal{inc}}_{\infty}(0,x,v)  = -  \int_{0}^{x} \frac{ \dd u }{\sqrt{ v^2 + 2 \phi^{\infty}(x) - 2\phi^{\infty}(u) } } > - \frac{x}{r} >  -\frac{1}{r}.
\end{align*}
Taking the infimum on $D^{+}_{r}$ yields $\underset{ (x,v) \in D^{+}_{r} }{\inf} t^{\textnormal{inc}}_{\infty}(0,x,v) \geq - \frac{1}{r}.$
The other bound is obtained in the same manner.
\emph{Proof of b)}
We only do the proof for the incoming time since the proof is similar for the outgoing time. We have three cases: 
if $(x,v) \in D^{+} \cup D^{-}$, then  $v^2 + 2 \phi^{\infty}(x) > 0$. Therefore one has,
\begin{align*}
\left \vert t_{\infty}^{\textnormal{inc}}(0,x,v)\right \vert  \leq \begin{cases}
\int_{0}^{x} \frac{\dd u}{\sqrt{-2\phi^{\infty}(u)}} & (x,v) \in D^{+},\\
\int_{x}^{1}  \frac{\dd u}{\sqrt{-2\phi^{\infty}(u)}} & (x,v) \in D^{-}.
\end{cases}
\end{align*}
Using the change of variable $u \mapsto \varphi = - \phi^{\infty}(u)$ and the fact that $\partial_{x} \phi^{\infty}$ is decreasing, we obtain the expected estimate.
If $(x,v) \in D^{\pm}$ with $v > 0$ then  $2 \phi^{\infty}(x_{0}) = v^2 + 2\phi^{\infty}(x)$. So,
\begin{align*}
\left \vert t_{\infty}^{\textnormal{inc}}(0,x,v)\right \vert  \leq \int_{x_0}^{x} \frac{\dd u}{\sqrt{2(\phi^{\infty}(x_0)-\phi^{\infty}(u))}} +  \int_{x_0}^{1} \frac{\dd u}{\sqrt{2(\phi^{\infty}(x_0)-\phi^{\infty}(u))}} .
\end{align*}
Using again the change of variable $u \mapsto \varphi = -\phi^{\infty}(u)$ and the fact  that $ \phi^{\infty}$ is decreasing and that $\partial_{x} \phi^{\infty} $ is also decreasing with $\partial_{x} \phi^{\infty}< 0 $ in $[0,1]$, we get the expected estimate. Eventually, if $(x,v) \in D^{\pm}$ with $v \leq 0$, we use the fact that $v^2 + 2(\phi^{\infty}(x) - \phi^{\infty}(u)) \geq 2(\phi^{\infty}(x) - \phi^{\infty}(u))$ for all $u \in [x,1].$ So using again the change of variable $u \mapsto -\phi^{\infty}(u)$ and the fact that $\partial_{x} \phi^{\infty}$ is decreasing we get the desired estimate.

\end{proof}

\end{corollary}
Using the fact that for any time $t \in \RR$ and $(x,v) \in Q$,  $t^{\textnormal{inc}}_{\infty}(t,x,v) = t + t^{\textnormal{inc}}_{\infty}(0,x,v)$ , $t^{\textnormal{out}}_{\infty}(t,x,v) = t + t^{\textnormal{out}}_{\infty}(0,x,v)$ we obtain that Proposition \ref{egc_linear} is a consequence of the above corollary.

\subsection{Measurability and change of variables in the integrals}
We introduce different subsets of the phase-space defined for $t \geq 0$ by
\begin{align}
\mathcal{A}_{t} := \lbrace (x,v) \in Q\setminus S \: : \: t_{\infty}^{\textnormal{inc}}(t,x,v) < 0 \rbrace , \label{ba}\\
\mathcal{B}_{t} := \lbrace (x,v) \in Q \setminus S \: : \: t_{\infty}^{\textnormal{inc}}(t,x,v) = 0 \rbrace.\label{be}
\end{align}
Points in $\mathcal{A}_{t}$  are on characteristics that do not leave $Q$ on the interval $[0,t]$.
Points in  $\mathcal{B}_{t}$  are on characteristics which reach $\Sigma^{\textnormal{inc}}$ at  time zero.
We do not consider points which are in $S$ because the incoming and outgoing times \eqref{inc_time} and \eqref{out_time} are not continuous in $S$. Since $S$ is  a set of measure zero,  it does not pose any difficulty to define the solution $h$ to \eqref{linear-Vlasov-i} outside $S$ since we are concerned with a mild solution and therefore a function which is defined a.e.

\begin{lemma}[Measurability and change of variable] \label{propreties_of_the_flow} Let $t \geq 0$. We have
\begin{itemize}
\item[a)] The sets $\mathcal{A}_{t}, \mathcal{B}_{t},$  are Borel sets and the set $\mathcal{B}_{t}$ is of Lebesgue measure zero. 
\item[b)] $\mathcal{F}^{\infty}_{0,t}(\mathcal{A}_{t}) = \Big \lbrace (x',v') \in Q \setminus S \: : \: t_{\infty}^{\textnormal{out}}(0,x',v') > t \Big\rbrace.$

\item[c)] For any mesurable function $f : Q \longrightarrow \RR^{+}$, we have
\begin{align*}
\int_{\mathcal{A}_{t}} f\Big( \mathcal{F}^{\infty}_{0,t}(x,v) \Big) \dd x \dd v = \int_{\mathcal{F}^{\infty}_{0,t}(\mathcal{A}_{t})} f( x',v') \dd x' \dd v' \leq \int_{Q} f( x',v') \dd x' \dd v'.
\end{align*}

\end{itemize}
\begin{proof} Let $t \geq 0.$
\newline
\emph{ Proof of a).} By virtue of the Lemma \ref{regularity_lemma}, the function $(x,v) \in Q  \setminus S \mapsto t_{\infty}^{\textnormal{inc}}(t,x,v)$ is continuous. Therefore the sets $\mathcal{A}_{t}$ and $\mathcal{B}_{t}$ are Borel sets since they are respectively open and closed. Since $\mathcal{F}^{\infty}_{0,t}$ is measure preserving the  measure of $\mathcal{B}_{t}$ is equal to the measure of $\mathcal{F}_{0,t}^{\infty}(\mathcal{B}_{t}).$ By definition of $\mathcal{B}_{t}$ we have $\mathcal{F}_{0,t}^{\infty}(\mathcal{B}_{t}) \subset \Sigma^{\textnormal{inc}}$ where $\Sigma^{\textnormal{inc}}$ is a set of measure zero.

\emph{Proof of b).}
We begin with the  embedding $\big \lbrace \mathcal{F}^{\infty}_{0,t}(x,v) \: : \: (x,v) \in \mathcal{A}_{t} \big \rbrace \subset \big \lbrace (x',v') \in Q \setminus S \: : \: t_{\infty}^{\textnormal{out}}(0,x',v') > t \big\rbrace $. Let $(x,v) \in \mathcal{A}_{t}$ and set $(x',v') = \mathcal{F}^{\infty}_{0,t}(x,v).$ For $ s \in [0,t]$ we have, 
\[
\mathcal{F}_{s,0}^{\infty}(x',v') = \mathcal{F}_{s,0}^{\infty} \circ \mathcal{F}_{0,t}^{\infty}(x,v) = \mathcal{F}_{s,t}(x,v).
\]
By definition of $\mathcal{A}_{t}$, we have $t_{\infty}^{\textnormal{inc}}(t,x,v) < 0$. So, $\mathcal{F}^{\infty}_{s,t}(x,v) \in Q $ for all $s \in [0,t]$. By energy conservation \eqref{first_integral} we have $\mathcal{F}^{\infty}_{s,t}(x,v) \notin S$ for all $s \in [0,t]$ since $(x,v) \notin S.$  Thus, for $s \in [0,t]$ we have $\mathcal{F}^{\infty}_{s,0}(x',v') \in Q \setminus S$ and consequently $t^{\textnormal{out}}_{\infty}(0,x',v') > t.$ We now show the reverse embedding. Let $(x',v') \in Q \setminus S$ such that $t_{\infty}^{\textnormal{out}}(0,x',v') > t.$ Therefore $\mathcal{F}_{s,0}(x',v') \in Q$ for all $s \in [0,t]$. By conservation of energy \eqref{first_integral} we have $\mathcal{F}_{s,0}(x',v') \notin S$  for all $s \in [0,t]$ since $(x',v') \notin S.$ Then set $(x,v) = \mathcal{F}_{t,0}(x',v').$ Since, for all $s \in [0,t]$ we have $\mathcal{F}_{s,t}(x,v) = \mathcal{F}_{s,0}(x',v') \in Q \setminus S$ we deduce that $t_{\infty}^{\textnormal{inc}}(t,x,v) < 0$. Therefore $(x,v) \in \mathcal{A}_{t}$ and $(x',v') = \mathcal{F}_{0,t}(x,v)$. It shows the claim.

\emph{ Proof of c).} It is just a consequence of the measure preserving change of variables $(x',v') \longmapsto (X_{\infty},V_{\infty})(0;t,x,v)$ in the integral.
\end{proof}
\end{lemma}
\subsection{The proof of wellposedness for \textbf{(LVP)}}

Consider the assumptions of Theorem \ref{stability_linear_eq}. 
Let  $T  > 0$ and fix $\gamma >   \frac{2 \| \partial_{v} f^{\infty} \|_{L^1(Q)} }{\lambda^2}$ and note that $\gamma$ is chosen independently of $T.$ Consider $\mathscr{S}$ the operator defined for $h \in X_{T,\gamma}$ by 
\begin{align}
& \mathscr{S}(h)(t,x,v) = \mathbf{1}_{\mathcal{A}_{t} \cup \mathcal{B}_{t} }(x,v) h_{0} \Big( X_{\infty}\Big(0;t,x,v \Big), V_{\infty}\Big(0;t,x,v\Big) \Big)\label{def_S} \\
    &+ \mathcal{G}(x,v)  \int_{0}^{t} \mathbf{1}_{t_{\infty}^{\textnormal{inc}}(t,x,v) < s }\partial_{x} U (h) \Big(s, X_{\infty}\Big(s;t,x,v \Big) \Big)  V_{\infty}\Big( s;t,x,v \Big) \dd s\nonumber 
\end{align}
for a.e $(t,x,v) \in [0,T] \times Q$, where for a.e $0 \leq s \leq T$, $ U(h)(s, \cdot)$ is the solution to the linearized Poisson equation \eqref{Poisson_f} with the source term $\rho_{h}(s,\cdot) = \int_{\RR} h(s,\cdot,v) \dd v$ and the sets $\mathcal{A}_{t}$ and $\mathcal{B}_{t}$ are given in \eqref{ba}-\eqref{be}.  Remark that the indicatrix functions are a.e equal to those appearing in the mild formulation \eqref{def_f_linear} so that both formulations are in fact equivalent. Set for a.e $(t,x,v) \in [0,T] \times Q$,
\begin{align*}
    &I_1(t,x,v) =   \mathbf{1}_{\mathcal{A}_{t} \cup \mathcal{B}_{t}}(x,v) h_{0} \Big( X_{\infty}\Big(0;t,x,v \Big), V_{\infty}\Big(0;t,x,v\Big) \Big),\\
   & I_2(t,x,v) = \mathcal{G}(x,v)  \int_{0}^{t}   \mathbf{1}_{t_{\infty}^{\textnormal{inc}}(t,x,v) < s } \partial_{x} U (h) \Big(s, X_{\infty}\Big(s;t,x,v \Big) \Big)  V_{\infty}\Big( s;t,x,v \Big) \dd s.
\end{align*}

\emph{ Step 1: Stability estimate in $X_{T,\gamma}.$}
We will show separately that $I_{1}$ and $I_{2}$ belong to $X_{T,\gamma}$.
According to Lemma \ref{propreties_of_the_flow} a), the sets $\mathcal{A}_{t}, \mathcal{B}_{t}$ are Borel sets for every $t \in [0,T]$. Therefore, $I_{1}(t,\cdot,\cdot) ,I_{2}(t,\cdot,\cdot)$ are measurable functions.
Using Lemma \ref{propreties_of_the_flow} c) we have for a.e $t \in [0,T],$

\begin{align*}
\int_{Q} \vert I_{1}(t,x,v) \vert \dd x \dd v = \int_{\mathcal{A}_{t}}  \left \vert h_{0} \Big( \mathcal{F}^{\infty}_{0,t}(x,v) \Big) \right \vert  \dd x \dd v \leq \int_{Q} \left \vert h_{0}(x',v') \right \vert \dd x' \dd v',
\end{align*}
where we have used the fact that $\mathcal{B}_{t}$ is a set of measure zero.
Therefore $I_{1} \in X_{T,\gamma}.$
We now treat $I_{2}.$  Using the fact that $\mathcal{G}$ is supported in $D^{+}$, we have for a.e $t \in [0,T]$
\begin{align*}
&\int_{Q} \left \vert I_{2}(t,x,v) \right \vert \dd x \dd v \\
&= \int_{D^{+}} \left \vert \mathcal{G}(x,v) \right \vert   \left \vert \int_{0}^{t} \mathbf{1}_{t^{\infty}_{\textnormal{inc}}(t,x,v)  < s }  \partial_{x} U(h)(s; X_{\infty}(s;t,x,v))  V_{\infty}(s;t,x,v) \dd s \right \vert  \dd x \dd v.
\end{align*}
Using a triangular inequality and the Fubini-Tonelli theorem we have,
\begin{align*}
&\int_{Q} \left \vert I_{2}(t,x,v) \right \vert \dd x \dd v \\
&\leq \int_{0}^{t} \int_{D^{+}}  \mathbf{1}_{ t^{\infty}_{\textnormal{inc}}(t,x,v)< s } \left \vert \mathcal{G}(x,v) \right \vert  \left \vert  \partial_{x} U(h)(s; X_{\infty}(s;t,x,v))   \right \vert \left \vert V_{\infty}(s;t,x,v)  \right \vert \dd x \dd v \dd s.
\end{align*}
For $s \in (0,t)$, we use the measure preserving change of variable $(x,v) \in D^{+} \longmapsto (x',v') = \mathcal{F}^{\infty}_{s,t}(x,v)$ we therefore obtain
\begin{align*}
&\int_{Q} \left \vert I_{2}(t,x,v) \right \vert \dd x \dd v\\
& \leq \int_{0}^{t} \int_{\mathcal{F}^{\infty}_{s,t}(D^{+})}  \mathbf{1}_{t^{\infty}_{\textnormal{inc}}(t,\mathcal{F}_{t,s}^{\infty}(x',v')) < s } \left \vert \mathcal{G}(\mathcal{F}^{\infty}_{t,s}(x',v')) \right \vert  \left \vert  \partial_{x} U(h)(s; x') \right \vert \vert v' \vert  \dd x' \dd v' \dd s.
\end{align*}
By the linear elliptic estimates of Proposition \ref{estimate_Linf_E}, we have that for a.e $s \in (0,t),$ $\partial_{x} U(h) (s,\cdot) \in L^{\infty}(0,1)$ and therefore
\begin{align*}
& \int_{Q}  \left \vert I_{2}(t,x,v) \right \vert \dd x \dd v \\
& \leq \int_{0}^{t} \| \partial_{x} U (h)(s) \|_{L^{\infty}(0,1)} \int_{\mathcal{F}^{\infty}_{s,t}(D^{+})}  \mathbf{1}_{t^{\infty}_{\textnormal{inc}}(t,\mathcal{F}_{t,s}^{\infty}(x',v')) < s} \left \vert \mathcal{G}(\mathcal{F}^{\infty}_{t,s}(x',v'))\right \vert  \vert v' \vert  \dd x' \dd v'  \dd s.
\end{align*}
By Lemma \ref{invariant_region}, $D^{+}$ is invariant by the flow. Therefore on the set defined for $s \in (0,t)$ by 
\[
\mathcal{D}_{s,t} = \Big \lbrace (x',v') \in \mathcal{F}^{\infty}_{s,t}(D^{+}) \: : \: t^{\infty}_{\textnormal{inc}}(t,\mathcal{F}_{t,s}^{\infty}(x',v')) < s  \Big \rbrace \subset D^{+}
\]
we have $\mathcal{G}(\mathcal{F}^{\infty}_{t,s}(x',v')) = \mathcal{G}(x',v')$ because $\mathcal{G}$ depends only on the microscopic energy which is a conserved quantity along the flow \eqref{first_integral} and $v' > 0.$ We thus get
\begin{align*}
 \int_{Q}  \left \vert I_{2}(t,x,v) \right \vert \dd x \dd v  \leq \int_{0}^{t} \| \partial_{x} U (h)(s) \|_{L^{\infty}(0,1)} \int_{\mathcal{D}_{s,t}}   \left \vert \mathcal{G}(x',v')\right \vert  v'  \dd x' \dd v'  \dd s.
\end{align*}
Since $\mathcal{D}_{s,t} \subset D^{+}$,  by monotony of the integral we obtain the crude bound
\begin{align*}
 \int_{Q}  \left \vert I_{2}(t,x,v) \right \vert \dd x \dd v  \leq \int_{0}^{t} \| \partial_{x} U (h)(s) \|_{L^{\infty}(0,1)} \int_{D^{+}}  \left \vert \mathcal{G}(x',v') \right \vert   v'   \dd x' \dd v'  \dd s.
\end{align*}
Remark that $\| \partial_{v} f^{\infty} \|_{L^{1}(Q)} = \int_{D^{+}} \left \vert \mathcal{G}(x',v')\right \vert  v' \dd x' \dd v'$, we eventually get

\begin{align*}
 \int_{Q}  \left \vert I_{2}(t,x,v) \right \vert \dd x \dd v  \leq  \| \partial_{v} f^{\infty} \|_{L^{1}(Q)}  \int_{0}^{t} \| \partial_{x} U (h)(s) \|_{L^{\infty}(0,1)}  \dd s.
\end{align*}
Combining this upper bound with the linear elliptic estimates \eqref{dxV} and the fact that $\| \rho_{h}(s) \|_{L^{1}(0,1)} \leq \| h(s) \|_{L^{1}(Q)}$, we thus glean
\begin{align}
\int_{Q}  \left \vert I_{2}(t,x,v) \right \vert \dd x \dd v \leq  \frac{2 \| \partial_{v} f^{\infty} \|_{L^{1}(Q)}}{\lambda^2}  \int_{0}^{t}  \|h(s) \|_{L^{1}(Q)} \dd s. \label{linear_estimate}
\end{align}
Now observe that for $ 0 \leq  t \leq  T,$
\begin{align*}
\int_{0}^{t}  \| h(s) \|_{L^{1}(Q)} \dd s \leq  \| h \|_{X_{T,\gamma}} \frac{e^{\gamma t}-1}{\gamma}, 
\end{align*}
So that we obtain that for a.e $ 0 \leq t \leq T,$
\begin{align}
e^{-\gamma t} \int_{Q} \left \vert I_{2}(t,x,v) \right \vert \dd x \dd v \leq  \frac{2 \| \partial_{v} f^{\infty} \|_{L^{1}(Q)} }{ \lambda^2 \gamma}  \| h \|_{X_{T,\gamma}} \label{linear_estimate_final} 
\end{align}
This last estimate shows that $I_{2} \in X_{T,\gamma}.$

\emph{ Step 2: Contraction.}
Since $\gamma > \frac{2 \| \partial_{v} f^{\infty} \|_{L^{1}(Q)} }{\lambda^2} $ and $\mathscr{S}$ is an affine map, the estimate \eqref{linear_estimate_final} shows that for $(f,g) \in X_{T,\gamma} \times X_{T,\gamma},$  we have $\| \mathscr{S}(f) - \mathscr{S}(g) \|_{X_{T,\gamma}} \leq a \| f- g \|_{X_{T,\gamma}}$ with $a=\frac{2 \| \partial_{v} f^{\infty} \|_{L^{1}(Q)} }{\lambda^2 \gamma} < 1. $ Since $\gamma$ is independent of $T$, where $T > 0$ was chosen arbitrarily,  the Banach-Picard theorem applies in $X_{T,\gamma}$ for any $ T > 0$. Therefore,  for any $T > 0$, there exists a unique $h \in X_{T,\gamma}$ such that $\mathscr{S}(h) = h $ in $X_{T,\gamma}.$  It shows that $h$ is a mild solution of the Vlasov equation \eqref{def_f_linear}. By definition of $h$, we have that $s \mapsto \rho_{h}(s,\cdot) = \int_{\RR} h(s,\cdot,v) \dd v \in L^{\infty}([0,T]; L^{1}(0,1))$. Then,  we know from Proposition \ref{estimate_Linf_E} that there is a unique $U(h) \in Y_{T,\gamma}$ which solves the linearized Poisson equation \eqref{Poisson_f} with the source term $\rho_{h}.$  
The proof of continuity in time is tedious and thus omitted.
Since the result holds for arbitrary $T > 0$, we obtain the conclusion.

\subsection{ A delayed Grönwall  lemma}\label{sec:delayed_gronwall}
A consequence of the exit geometric conditions \eqref{egc_linear} will be that $t \longmapsto \|h(t)\|_{L^{1}(Q)}$ verifies for $t \geq T_{r}$ a first order delayed integral equation.  We show in the following that such an integral equation has exponential decaying solutions provided that the rate of increase of the solution is small enough.
We have.
\begin{lemma}[A delayed Grönwall lemma]\label{delayed_Gronwall_lemma}Let $ T > 0$ and $\alpha > 0$ such that $\alpha  T < 1.$ Let $y \in \mathscr{C}(\RR^{+})$ such that for $t \geq T$
\begin{align}
y(t) \leq \alpha \int_{t-T}^{t} y(s) \dd s. \label{fafa}
\end{align}
Then there exists a constant $\kappa \equiv \kappa(\alpha,T) > 0$ such that for $ t \geq 0$
\begin{align}
 y(t) \leq C_{T,\kappa} \exp(-\kappa t)
\end{align}
where  $C_{T,\kappa} = \underset{ t \in [0,T] }{ \sup } \vert y(t) \vert  \exp{(\kappa t)}.$ In addition, we have the lower bound
\begin{align}
\kappa > - \frac{ \log(\alpha T)}{T}.
\end{align}

\begin{proof} Let us first remark that if $y$ is negative on $[0,T]$ then $y$ is also negative on $[T,+\infty).$ Indeed,  if it is not the case then let $t^{*} := \sup \lbrace t > T \: : \: y(s) < 0 \: \forall s \in [0,t) \rbrace$. By continuity of $y$ this point exists and $y(t^{*}) = 0.$ Then applying \eqref{fafa} for
$t = t^{*}$ we have
\begin{align}
y(t^{*}) \leq \alpha \int_{t^{*}-T}^{t^{*}} y(s) \dd s < 0
\end{align}
which yields a contradiction. We now refine this result by proving that if $y$ is non positive on $[0,T]$ then it is also non positive on $[T,+\infty).$ Indeed, consider the auxiliary sequence of functions defined for all $n \in \NN$ by  $y_{n} : t \in \RR^{+} \longmapsto y(t) - \frac{1}{n+1}$. By virtue of the previous result, for each $n \in \mathbb{N}$ we have that $y_{n}$ is negative on $[0,T]$ and therefore all $t \geq 0$,
$y(t) < \frac{1}{n+1}.$ Passing to the limit as $n \longrightarrow +\infty$ yields $y(t) \leq 0$. This shows that $y$ is non positive on $[0,+\infty).$

Consequently, we deduce that if $z \in \mathscr{C}^{1}(\RR^{+})$ verifies the case of equality in \eqref{fafa} and is such that $y_{|[0,T]} \leq z_{|[0,T]}$ then for all $t \geq 0$, $y(t) \leq z(t).$ We then look for a function of the form $t \in \RR^{+} \longmapsto \exp(-\kappa t)$ where $\kappa > 0$ verifies the equality in \eqref{fafa}. It yields the equation for $\kappa$
\begin{align}
-\kappa = \alpha ( 1 - \exp(\kappa T)). \label{fifi}
\end{align}
The function $a : u \in \RR^{+} \longmapsto \alpha (1- \exp(uT)) + u$ is continuous with $a(0) =0$ and $\lim_{ u \rightarrow +\infty} \limits a(u) = -\infty$. It therefore suffices that $a'(0) > 0$ to obtain the existence of $\kappa > 0$ such that $a(\kappa) = 0.$ Note that $a'(0) > 0$ is equivalent to $\alpha T < 1$. Thus, we conclude  that if $\alpha T < 1$ then there is $\kappa  > 0$ which solves \eqref{fifi} and such that the function $z : t \in \RR^{+} \longmapsto C_{T} \exp(-\kappa t)$ verifies $z_{|[0,T]} \geq y_{|[0,T]}$ (by definition of $C_{T}$) and $z(t) = \alpha \int_{t-T}^{t} z(s) \dd s$ for all $t \geq T$.
We then conclude that for all $t \geq 0,$ $y(t) \leq z(t)$ which shows the first part of the claim. 
Eventually, we remark that the function $a$  is increasing  on $[0, - \frac{\log(\alpha T)}{T}]$ with $a(0)= 0$ and decreasing on $\big[- \frac{\log(\alpha T)}{T}, +\infty \big)$. Therefore it is positive on $\Big( 0,- \frac{\log(\alpha T)}{T} \big]$. Since $\kappa$ is a zero of the function $a$ on $\RR^{+}_{\star}$, we have $\kappa > - \frac{\log(\alpha T)}{T}.$
\end{proof}
\end{lemma}

\subsection{ The proof of stability for \textbf{(LVP)}}
We are now ready to prove the linear stability theorem.
We begin with the \emph{Decay of the $L^{1}$ norm in  $Q \setminus D^{+}.$}
For every $ t \geq 0$ and a.e $(x,v) \in Q$ we have,
\begin{align*}
& h(t,x,v) =  \mathbf{1}_{\mathcal{A}_{t} \cup \mathcal{B}_{t} }(x,v) h_{0} \Big( X_{\infty}\Big(0;t,x,v \Big), V_{\infty}\Big(0;t,x,v\Big) \Big)\nonumber \\
    &+ \mathcal{G}(x,v)  \int_{0}^{t}\mathbf{1}_{t_{\infty}^{\textnormal{inc}}(t,x,v) < s }  \partial_{x} U (h) \Big(s, X_{\infty}\Big(s;t,x,v \Big) \Big)  V_{\infty}\Big( s;t,x,v \Big) \dd s\nonumber  
\end{align*}
where the set $\mathcal{A}_{t},\mathcal{B}_{t}$  are given in \eqref{ba}-\eqref{be}.
Since $\supp \mathcal{G} \subset D^{+},$ we obtain for a.e $(x,v) \in Q \setminus D^{+}$ that
\begin{align*}
h(t,x,v) =\mathbf{1}_{\mathcal{A}_{t} \cup \mathcal{B}_{t} }(x,v)h_{0} \Big( X_{\infty}\Big(0;t,x,v \Big), V_{\infty}\Big(0;t,x,v\Big) \Big).
\end{align*}
Then
\begin{align*}
\| h(t) \|_{L^{1}\big( Q \setminus D^{+} \big)}  = \int_{(Q \setminus (D^{+} \cap S))\cap \mathcal{A}_{t}}  \Big \vert h_{0} \Big( \mathcal{F}^{\infty}_{0,t}(x,v) \Big) \Big \vert  \dd x \dd v
\end{align*}
where we have used again the fact $\mathcal{B}_{t}$ is of measure zero.  Using the measure preserving change of variable $(x,v) \mapsto \Big( x' = X_{\infty}(0;t,x,v), v' = V_{\infty}(0;t,x,v) \Big)$ we get
\begin{align*}
\| h(t) \|_{L^{1}\Big( Q \setminus D^{+} \Big)}  = \int_{ \mathcal{F}_{0,t}^{\infty} \Big( (Q \setminus (D^{+} \cap S) )\cap \mathcal{A}_{t} \Big) }  \vert  h_{0}(x',v')  \vert \dd x' \dd v'.
\end{align*}
Set 
\[
\mathcal{O}_{t} := \mathcal{F}_{0,t}^{\infty} \Big( (Q \setminus (D^{+} \cap S) )\cap \mathcal{A}_{t} \Big).
\]
Following Lemma  \ref{propreties_of_the_flow} b) we deduce that,
\begin{align}
\mathcal{O}_{t}  =  \Big \lbrace (x,v) \in Q \setminus (D^{+} \cap S)  \: : t^{\textnormal{out}}(0,x,v) > t  \Big \rbrace.
\end{align}
Since the outgoing time is a continuous function in $Q \setminus S$, the set $\mathcal{O}_{t}$ is a Borel set.
For  $t = 0$, we have obviously $ \mathcal{O}_{0} = Q \setminus (D^{+} \cap S)$.  If  $t_{1}$ and $t_{2}$ are two positive numbers such that $t_{1} <  t_{2}$ then $ \mathcal{O}_{t_{2}} \subset \mathcal{O}_{t_{1}}$. Therefore, the family $(\mathcal{O}_{t})_{t \in \RR^{+}}$ is non increasing. According to Corollary \eqref{estimate_incoming_time}, the outgoing time is bounded in $Q \setminus (D^{+} \cap S)$ and therefore  if $t \geq \underset{(x,v) \in Q \setminus (D^{+} \cap S)} \sup t_{\infty}^{\textnormal{out}}(0,x,v)$ then $ \mathcal{O}_{t}= \emptyset.$ It yields the claim.

We now prove the \emph{Exponential decay  in $D^{+}_{r}$}. Let $r \geq 0$ and assume that $\supp h_{0} \subset D^{+}_{r}$. Since $\supp \mu' \subset (r,+\infty)$, we have that $\partial_{v} f^{\infty}$ given in \eqref{dv_feq} is such that $\supp  \partial_{v} f^{\infty} \subset D^{+}_{r}.$ Thanks to the representation formula \eqref{def_f_linear} and the invariance of $D^{+}_{r}$ by the stationary flow, we have that
\begin{align}
 \forall t \geq 0, \quad \supp h(t) \subset D^{+}_{r}.
\end{align}
We are now going to prove that $\| h(t) \|_{L^{1}(D^{+}_{r})}$ verifies a delayed Gronwall type inequality of the form \eqref{fafa}. We study separately the contribution of the initial data and the source term. 
\newline
\emph{Study of the initial data.}
Let  $t > 0$ and $(x,v) \in Q$ such that $t_{\infty}^{\textnormal{inc}}(t,x,v) \leq 0.$ First, we observe thanks to the invariance of $D^{+}_{r}$ that:
\begin{align*}
&(X_{\infty},V_{\infty})(0;t,x,v) \in D^{+}_{r} \Longrightarrow \\
 &\exists (x',v') \in D^{+}_{r}, t_{\infty}^{\textnormal{out}}(0,x',v') > t, (x,v) = (X_{\infty},V_{\infty})(t;0,x',v'). 
\end{align*}
Yet, according to the Proposition \ref{egc_linear}, the equilibrium electric field $-\partial_{x} \phi^{\infty}$ satisfies the initial exit geometric condition in time $T_{r}$ with respect to $D^{+}_{r}$. Therefore, \newline $t_{\infty}^{\textnormal{out}}(0,x',v') \leq T_{r}$ for all $(x',v') \in D^{+}_{r}$.  Since $\supp h_{0} \subset D^{+}_{r}$ we deduce that
\[
t \geq T_{r}  \Longrightarrow \mathbf{1}_{t_{\infty}^{\textnormal{inc}}(t,x,v) \leq 0} h_{0} \Big( X_{\infty}\Big(0;t,x,v \Big), V_{\infty}\Big(0;t,x,v\Big) \Big) = 0.
\]

Hence, for $t \geq T_{r}$ and a.e $(x,v) \in Q,$
\begin{align*}
& h(t,x,v) =  \mathcal{G}(x,v)  \int_{0}^{t} \mathbf{1}_{t_{\infty}^{\textnormal{inc}}(t,x,v) < s} \partial_{x} U (h) \Big(s, X_{\infty}\Big(s;t,x,v \Big) \Big)  V_{\infty}\Big( s;t,x,v \Big) \dd s\nonumber .
\end{align*}
\emph{Study of the source term.}
Let $t \geq T_{r}$.  We now estimate the $L^{1}$ norm of $h$ on $D^{+}_{r}$ by a direct computation.
Using a triangular inequality and an $L^{\infty}$ bound on $\partial_{x} U$ we have,
\begin{align*}
&\int_{D^{+}_{r}} \vert h(t,x,v) \vert \dd x \dd v \\
&\leq \int_{D^{+}_{r}}  \Big \vert \mathcal{G}(x,v) \Big \vert \int_{0}^{t}  \mathbf{1}_{t_{\infty}^{\textnormal{inc}}(t,x,v) < s}  \| \partial_{x} U(h)(s) \|_{L^{\infty}(0,1)} \Big \vert V_{\infty}(s;t,x,v)  \Big \vert  \dd s \dd x \dd v\\
& = \int_{0}^{t}  \| \partial_{x} U(h)(s) \|_{L^{\infty}(0,1)}   \int_{D^{+}_{r}} \Big \vert \mathcal{G}(x,v) \Big \vert \mathbf{1}_{  t^{\textnormal{inc}}_{\infty}(t,x,v) < s }\Big \vert V_{\infty}(s;t,x,v)  \Big \vert  \dd x \dd v \dd s
\end{align*}
where we have used Fubini's theorem for the last equality. Define for $s \in (0,t)$
\begin{align}
\mathcal{D}^{r}_{s,t} = \lbrace (x,v) \in D^{+}_{r} \: : \:   t_{\infty}^{\textnormal{inc}}(t,x,v)  < s \rbrace.
\end{align}
Using the measure preserving change of variable $(x,v) \in \mathcal{D}^{r}_{s,t} \longmapsto (x',v') = \mathcal{F}_{s,t}^{\infty}(x,v)$, we then obtain
\begin{align*}
 \int_{D^{+}_{r}} \Big \vert \mathcal{G}(x,v) \Big \vert \mathbf{1}_{  t^{\textnormal{inc}}_{\infty}(t,x,v) < s}\Big \vert V_{\infty}(s;t,x,v)  \Big \vert  \dd x \dd v  = \int_{\mathcal{F}^{\infty}_{s,t}(\mathcal{D}^{r}_{s,t})}  \Big \vert \mathcal{G}(x',v') \Big \vert v'  \dd x' \dd v'
\end{align*}
where, arguing as in the proof Lemma \ref{propreties_of_the_flow} b),  we have
\begin{align}
\mathcal{F}_{s,t}^{\infty}(\mathcal{D}^{r}_{s,t}) = \Big \lbrace (x',v') \in D^{+}_{r} \: : \: t_{\infty}^{\textnormal{out}}(s,x',v') > t \Big \rbrace,
\end{align}
and where we have used the fact that the function $(x,v) \in D^{+} \longmapsto \mathcal{G}(x,v)$ depends only on the microscopic energy (which is a conserved quantity \eqref{first_integral}) and the fact that if $(x',v') \in \mathcal{F}^{\infty}_{s,t}(\mathcal{D}^{r}_{s,t})$ then $v' > 0$.  Thus,
\begin{align*}
 \int_{D^{+}_{r}} \Big \vert \mathcal{G}(x,v) \Big \vert \mathbf{1}_{ s >  t^{\textnormal{inc}}_{\infty}(t,x,v)}\Big \vert V_{\infty}(s;t,x,v)  \Big \vert  \dd x \dd v  = \int_{D^{+}_{r}}   \mathbf{1}_{t_{\infty}^{\textnormal{out}}(s,x',v') > t} \Big \vert \mathcal{G}(x',v') \Big \vert v'  \dd x' \dd v'.
\end{align*}
Gathering this term with the other, we obtain
\begin{align*}
&\int_{D^{+}_{r}} \left \vert h(t,x,v) \right \vert \dd x \dd v  \leq  \int_{0}^{t}   \| \partial_{x} U(h)(s) \|_{L^{\infty}(0,1)}   \int_{D^{+}_{r}} \mathbf{1}_{t_{\infty}^{\textnormal{out}}(s,x',v') > t} \Big \vert \mathcal{G}(x',v') \Big \vert v'  \dd x' \dd v' \dd s\\
&= \int_{D^{+}_{r}} \Big \vert \mathcal{G}(x',v') \Big \vert v'   \int_{0}^{t} \mathbf{1}_{t_{\infty}^{\textnormal{out}}(s,x',v') > t}  \| \partial_{x} U(h)(s) \|_{L^{\infty}(0,1)}  \dd s  \dd x' \dd v' \\
& \leq \int_{D^{+}_{r}} \Big \vert \mathcal{G}(x',v') \Big \vert v' \dd x' \dd v'  \times \underset{ (x',v') \in D^{+}_{r}} {\sup} \int_{0}^{t} \mathbf{1}_{t_{\infty}^{\textnormal{out}}(s,x',v') > t}  \| \partial_{x} U(h)(s) \|_{L^{\infty}(0,1)}  \dd s.
\end{align*}
According to the Proposition \ref{egc_linear}, the equilibrium electric field $-\partial_{x} \phi^{\infty}$ satisfies the internal exit geometric condition  in time $T_{r}$ with respect to $D^{+}_{r}$ on $\RR^{+}$.  So,
\[
\underset{ (s,x,v) \in \RR^{+} \times D^{+}_{r} } \sup (t_{\infty}^{\textnormal{out}}(s,x,v) - s )\leq T_{r}.
\]
Therefore for every $s \in (0,t]$ and $(x',v') \in D^{+}_{r}$,  we have
\[
t_{\infty}^{\textnormal{out}}(s,x',v') > t \Longrightarrow s > t  - T_{r}.
\]
Therefore,
\begin{align*}
 \underset{ (x',v') \in D^{+}} \sup  \int_{0}^{t} \mathbf{1}_{t_{\infty}^{\textnormal{out}}(s,x',v') \geq t}  \| \partial_{x} U(h)(s) \|_{L^{\infty}(0,1)}  \dd s \leq \int_{t - T_{r}}^{t}  \| \partial_{x} U(h)(s) \|_{L^{\infty}(0,1)}  \dd s.
\end{align*}
Using  the elliptic estimate \eqref{dxV} and recalling that $\supp h(t) \subset D^{+}_{r}$ we eventually obtain a closed estimate for $\| h(t) \|_{L^{1}(D^{+}_{r})}$ which writes
\begin{align}
\forall t \geq T_{r}, \quad \| h(t) \|_{L^{1}(D^{+}_{r})} \leq \frac{2 \| \partial_{v} f^{\infty} \|_{L^{1}(D^{+}_{r})}}{\lambda^{2}} \int_{t-T_{r}}^{t}   \| h(s) \|_{L^{1}(D^{+}_{r})} \dd s.
\end{align}
where $\| \partial_{v} f^{\infty} \|_{L^{1}(D^{+}_{r})} = \int_{D^{+}_{r}} \Big \vert \mathcal{G}(x',v') \Big \vert v' \dd x' \dd v' = \int_{r}^{+\infty} \vert \mu'(v) \vert \dd v$. By virtue of Lemma \ref{delayed_Gronwall_lemma} we thus get the expected exponential decay provided that $ \frac{2 \| \partial_{v} f^{\infty} \|_{L^{1}(D^{+}_{r})} T_{r}}{\lambda^{2}} < 1.$ We thus infer \eqref{exponential_decay_h}, and from the elliptic estimates \eqref{dxV} we also have \eqref{exponential_decay_u}. It concludes the study of the linearized Vlasov-Poisson equations.

\section{Non linear stability}\label{sec:non_linear_stab}
We  now investigate the non linear stability of the Vlasov-Poisson system \eqref{Vlasov-i}-\eqref{bc-phi} written in a perturbative form.
Using the same notation, we denote the fluctuation
\begin{align}
h(t,x,v) = f(t,x,v) - f^{\infty}(x,v), \quad  U(t,x) = \phi(t,x) - \phi^{\infty}(x).
\end{align}
It satisfies the perturbative form of \textbf{(VP)} that we denote in short \textbf{(PVP):}

\[
\begin{cases}
&\partial_{t} h  + v \partial_{x} h - \partial_{x} (\phi^{\infty}+U) \partial_{v} h =  \partial_{x} U \partial_{v} f^{\infty}, \quad (t,x,v) \in (0,+\infty) \times Q, \\
&-\lambda^2 \partial_{xx} U +   n_{e}\big( \phi^{\infty}+U \big) - n_{e}\big(\phi^{\infty} \big) = \int_{\RR} h(\cdot, \cdot,v) \dd v, \quad  (t,x) \in \RR^{+} \times (0,1), \\
& h(t,x,v) = 0, \quad (t,x,v) \in (0,+\infty) \times \Sigma^{\textnormal{inc}},\\
&U(t,0) = 0, \quad U(t,1) = 0, \quad t \geq 0,\\
&h(t=0, x,v) = h_{0}(x,v), \quad (x,v) \in Q .
\end{cases}
\]
We shall prove a non linear version of Theorem \ref{stability_linear_eq} b) for \textbf{(PVP)} under smallness condition on the equilibrium and on the initial fluctuation. Though it is not the main purpose of this work to prove it, we need $U \in \mathscr{C}\big( \RR^{+}; W^{2,\infty}(0,1)\big)$ to define properly the characteristics. Following the work of BenAbdallah \cite{Ben-Abdallah} we restrict the set of admissible initial conditions. We define
\begin{align}
\mathscr{A} = \Big \lbrace h \in (L^{1} \cap L^{\infty})(Q) \: : \: \int_{Q} h(x,v) \vert v \vert ^2 \dd x \dd v < +\infty , \: \vert v\vert ^2 h(x,v) \in L^{\infty}(Q) \Big \rbrace.
\end{align}

For the non linear stability analysis we shall consider initial fluctuation $h_{0}$ that are supported in $D^{+}_{r}$ for $r >0$ large enough and we will localize the solution in $D^{+}_{\frac{r}{2}}$ provided the norm of $h_{0}$ is small enough in $L^{1}(Q)$. To do so, we introduce the fonction defined on $\RR^{+}_{\star}$ by
\begin{equation}
 \forall r > 0, \quad \delta(r):= r^{2}\Big( \exp(-2T_{r}) - \frac{1}{8}\Big) - 2T_{r} \vert \phi_{b} \vert  \label{delta_r},
\end{equation}
where we recall that $T_{r} = \frac{1}{r}$ for $r > 0$. Since $\vert \phi_{b} \vert > 0$, the function $\delta$ is increasing on $\RR^{+}_{\star}$, it is moreover continuous and such that $\lim_{r \rightarrow 0^{+}}  \limits\delta (r) = -\infty$ and $\lim_{r \rightarrow +\infty} \limits \delta(r) = +\infty$. Therefore, we  consider $r^{\star} \equiv r^{\star}(\phi_{b}) > 0$ to be the unique positive zero of the function $\delta$ and we have
\begin{align}
\forall r > r^{\star}, \quad \delta(r) > 0.
\end{align}
As for the linear analysis, we define an explicit lower bound of the rate of decay:
\begin{align} \label{def:kappa_star_NL}
\kappa_{NL}^{\star} = - \frac{1}{T_{r} } \log\Big( \frac{ 4\| \partial_{v} f^{\infty} \|_{L^{1}(D^{+}_{r})} T_{r} }{ \lambda^2}  \Big)
\end{align}
and we do observe that if $ \frac{ 4\| \partial_{v} f^{\infty} \|_{L^{1}(D^{+}_{r})} T_{r} }{ \lambda^2}   < 1$ then $\kappa^{\star}_{NL} > 0$.

Our main result is the following.
\begin{theorem}[Non linear stability] \label{non_linear_stab_thm} Let $\phi_{b} < 0$. Let $r > r^{\star}(\phi_{b})$ and $\mu$ be as  Theorem \ref{theorem_equilibrium} with $\mu$ such that $\supp \mu \subset (r; +\infty)$. For any $\lambda > 0$, there are positive constants $\varepsilon_{\infty} \equiv \varepsilon_{\infty}(r,\lambda^2)  < \frac{\lambda^2}{4 T_{r}}$ and $ \varepsilon_{0} \equiv \varepsilon(r,\lambda^2, \varepsilon_{\infty})$
such that if 
\begin{align} \label{smallness_mu_prime}
\| \mu ' \|_{L^{1}(r,+\infty)} < \varepsilon_{\infty}
\end{align}
and $h_{0} \in \mathscr{A}$ satisfies
\[
\supp h_{0} \subset D^{+}_{r}, \quad  \| h_{0} \|_{L^{1}(Q)} < \varepsilon_{0},
\]
then the stationary solution $(f^{\infty},\phi^{\infty})$ given by Theorem  \ref{theorem_equilibrium} is such that
\begin{align}
\| \partial_{v} f^{\infty} \|_{L^{1}(D^{+}_{r})} < \varepsilon_{\infty} \label{smallness_nl_cond_1},\\
\supp \partial_{v} f^{\infty} \subset D^{+}_{r}.
\end{align}
Moreover, any global mild-strong solution $(h,U) \in \mathscr{C}(\RR^{+};L^{1}(Q)) \times \mathscr{C}(\RR^{+};W^{2,\infty}(0,1))$ to \textbf{(PVP)} in the sense of Definition \eqref{mild_strong_non_linear_sol} associated with the initial data $h_{0}$ and the equilibrium $(f^{\infty},\phi^{\infty})$ satisfies
\begin{itemize}
\item[a)] \emph{Propagation of the support.} For every $t \geq 0$
\begin{align}
\supp h(t) \subset D^{+}_{\frac{r}{2}}.
\end{align}
\item[b)] \emph{Exponential Decay.} There are constants $\kappa > \kappa^{\star}_{NL} > 0$ and $C \geq 0$ such that for every $t \geq 0$ 
  \begin{align} 
& \| h(t) \|_{L^{1}\Big( D^{+}_{\frac{r}{2} }\Big) }\leq C \exp(- \kappa t),\label{Graal_1}\\
 &\| \partial_{x} U(t) \|_{L^{\infty}(0,1)} \leq \frac{2 C }{\lambda^2}  \exp(-\kappa t). \label{Graal_2}
 \end{align}
 \end{itemize}
\end{theorem}

Several comments are in order about this result
\begin{remark}
\begin{itemize}
\item The inequality \eqref{smallness_mu_prime} and \eqref{smallness_nl_cond_1} are equivalent. The constant $\varepsilon_{\infty}$ is chosen in such a way that it also implies the natural condition (regarding the linear analysis)
\[
\frac{4\| \partial_{v}f^{\infty} \|_{ L^{1}(D^{+}_{r})} T_{r}}{\lambda^2 } < 1.
\]
We recall that such an inequality will be needed for the first order delayed integral equation of the type \eqref{fafa} to have exponentially decaying solutions.
As such, the constant $\varepsilon_{\infty}$ is not explicit but it could be estimated at least numerically by solving the inequality \eqref{mickey}. Once this constant is fixed, we can estimate $\varepsilon_{0}$ numerically by solving the inequality \eqref{sufficient_condition_h0}.

\item We propagate the support of the initial data and the equilibrium up to the domain $D^{+}_{\frac{r}{2}} \supset D^{+}_{r}$. The supports were exactly preserved in the linear case thanks to the energy conservation \eqref{first_integral}. In the non linear case, we prove a stability result of the stationary phase portrait in Lemma \ref{stability_micro_energy}. This result is a key ingredient in the proof, both for the stability of the phase-portrait and for the propagation of the exit geometric conditions for the perturbed characteristics.
\end{itemize}
\end{remark}
To prove this  theorem, we shall again rely on the concept of mild-strong solutions to \textbf{(PVP)}, we therefore need to define the characteristics associated with a generic time dependent potential.
\subsection{The generic characteristics}
Let be $\phi \in \mathscr{C}\big( \RR^{+} ; W^{2,\infty}(0,1) \big)$ a generic potential. For every $s  \in \RR^{+}$, we consider and extension of $\phi(s,\cdot)$ to $\RR$ still denoted $\phi(s, \cdot) $ such that:
\begin{align}
\forall s \in \RR^{+}, \quad \phi(s,\cdot) \in W^{2,\infty}(\RR), \quad \| \partial_{x} \phi(s) \|_{L^{\infty}(\RR)} \leq  \| \partial_{x} \phi(s) \|_{L^{\infty}(0,1)}. \label{extension_2}
\end{align}
Such an extension is easily constructed in one dimension by affine extrapolation to $\RR \setminus [0,1]$ of the boundary value of $\phi(s,\cdot)$.

For $(t,x,v) \in \RR^{+} \times \RR^{2}$ we define the characteristics which passes through $(x,v)$ at time $t$ as the solution to the system of differential equations
\begin{equation} \label{general-char}
\begin{cases}
\frac{\dd }{\dd s}X_{\phi}(s; t,x,v) = V_{\phi}(s; t,x,v),\\
\frac{\dd }{\dd s}V_{\phi} (s; t,x,v) =  -\partial_{x} \phi \Big( s, X_{\phi}(s;t,x,v) \Big).\\
X_{\phi}(t;t,x,v) = x, V_{\phi}(t;t,x,v) = v.
\end{cases}
\end{equation}
Invoking the Cauchy-Lipschitz theorem,  there exists a unique solution 
\[
s \in \RR^{+} \longmapsto  (X_{\phi}(s;t,x,v),V_{\phi}(s;t,x,v)) \in \mathscr{C}^{1}(\RR^{+}) \cap W^{2,\infty}_{\textnormal{loc}}(\RR^{+}).
\]
When $(t,x,v) \in \RR^{+} \times Q$, we define the incoming time in $Q$ and the outgoing times of $Q$ by
\begin{align}
t^{\textnormal{inc}}_{\phi}(t,x,v) = \inf \lbrace s \in (-\infty,t) \cap \RR^{+}  \: : X_{\phi}(s';t,x,v) \in (0,1) \:  \forall s' \in (s,t) \rbrace, \label{inc_time_phi}\\
t^{\textnormal{out}}_{\phi}(t,x,v) = \sup \lbrace s \in (t,+\infty) \: : X_{\phi}(s';t,x,v) \in (0,1) \:  \forall s' \in (t,s) \rbrace. \label{out_time_phi}
\end{align}
The interval $(t^{\textnormal{inc}}_{\phi}(t,x,v), t^{\textnormal{out}}_{\phi}(t,x,v))$ is the largest open interval contained in $\RR^{+}$ on which the characteristic which started at time $t$ from the point $(x,v)$ is in $Q.$ By definition,
\begin{align}
t^{\textnormal{inc}}_{\phi}(t,x,v) > 0 \Longrightarrow X_{\phi}^{ \textnormal{inc} }(t,x,v)   \in \lbrace 0,1 \rbrace, \label{incoming_position}\\
t^{\textnormal{out}}_{\phi}(t,x,v) < +\infty \Longrightarrow X_{\phi}^{ \textnormal{out} }(t,x,v)   \in \lbrace 0,1 \rbrace.
\end{align}
We also define for notational convenience,
\begin{align}
&X_{\phi}^{ \textnormal{inc} }(t,x,v) = X_{\phi}\Big( t^{\textnormal{inc}}_{\phi}(t,x,v) ; t, x , v \Big),\\
&V_{\phi}^{\textnormal{inc}}(t,x,v) = V_{\phi}\Big( t^{\textnormal{inc}}_{\phi}(t,x,v) ;t,x,v \Big),
\end{align}
and  when $t^{\textnormal{out}}_{\phi}(t,x,v) < +\infty,$
\begin{align}
&X_{\phi}^{ \textnormal{out} }(t,x,v) = X_{\phi}\Big( t^{\textnormal{out}}_{\phi}(t,x,v) ; t, x , v \Big),\\
&V_{\phi}^{\textnormal{out}}(t,x,v) = V_{\phi}\Big( t^{\textnormal{out}}_{\phi}(t,x,v) ;t,x,v \Big).
\end{align}
We also denote for $(s,t) \in \RR^{+} \times \RR^{+}$ the associated flow by
\begin{align}
\mathcal{F}_{s,t}^{\phi} : (x,v) \in \RR^{2}  \longmapsto \Big (X_{\phi}(s;t,x,v), V_{\phi}(s;t,x,v) \Big). \label{flow_gen}
\end{align}
It is a measure preserving diffeomorphism.

We now define the notion of solutions we consider for \textbf{(PVP)}. 
\begin{definition}[Global mild-strong solution to \textbf{(PVP)}]  \label{mild_strong_non_linear_sol} Let $h_{0} \in \mathscr{A} $. We say that $(h,U)$ is a global mild-strong solution to \textbf{(PVP)} if
\begin{itemize}
\item[a)] $(h,U) \in \mathscr{C}(\RR^{+}; L^{1}(Q)) \times \mathscr{C}(\RR^{+}; W^{2,\infty}(0,1)).$
\item[b)] $h$ is a mild solution of the non linear Vlasov equation, in the sense that is satisfies for every $t \geq 0$ and a.e $(x,v) \in Q$,
\begin{align}
&h(t,x,v) = \mathbf{1}_{t^{\textnormal{inc}}_{\phi}(t,x,v) = 0 } h_{0} \Big( (X_{\phi},V_{\phi})(0;t,x,v) \Big) \label{def_h_nonlinear}  \\
    &+  \int_{0}^{t} \mathbf{1}_{t_{\phi}^{\textnormal{inc}}(t,x,v) < s} \partial_{x} U \Big(s, X_{\phi}(s;t,x,v) \Big)  \partial_{v} f^{\infty} \Big(   (X_{\phi}, V_{\phi})(s;t,x,v) \Big)  \dd s,\nonumber  
\end{align}
where $\phi = \phi^{\infty} + U$ where $\phi^{\infty}$ is the stationary solution given by theorem \ref{theorem_equilibrium} and $(X_{\phi},V_{\phi})$ are the characteristics defined in \eqref{general-char}.
\item[c)]  For every $t \geq 0$, $U(t)$ is a strong solution to the non linear Poisson equation
\begin{align}
-\lambda^2 \partial_{xx} U(t) +   n_{e}\big( \phi^{\infty}+U(t) \big) - n_{e}\big(\phi^{\infty} \big) = \int_{\RR} h(t,\cdot, v) \dd v,  \label{non_linear_Poisson}
\end{align}
and $U(t)(0) = U(t)(1) = 0.$
\end{itemize}
\end{definition}
As for the linear analysis, we define two types of exit geometric conditions which depends on a generic time dependent potential.
\begin{definition}\label{egcc_1} Let $\phi \in \mathscr{C}(\RR^{+}; W^{2,\infty}(0,1))$. Let $K \subset Q$ and $J$ a subinterval of $\RR^{+}$. We say that the electric field $-\partial_{x} \phi$  satisfies the internal exit geometric condition in time $T > 0 $ with respect to $K$ on $J$ if
\begin{align}
&\underset{(t,x,v) \in J \times K } \sup \big( t_{\phi}^{\textnormal{out}}(t,x,v)-t  \big) \leq T,\\
& \forall (t,x,v) \in J \times K, \quad (X_{\phi}^{\textnormal{out}},V_{\phi}^{\textnormal{out}})(t,x,v) \in \Sigma^{\textnormal{out}}
\end{align}
\end{definition}

\begin{definition}\label{egcc_2}  Let $\phi \in \mathscr{C}(\RR^{+}; W^{2,\infty}(0,1))$. Let $K \subset Q$. We say that the electric field $-\partial_{x} \phi$ satisfies the initial exit geometric condition in time $T > 0$ with respect to $K$ if
\begin{align}
 &\underset{(x,v) \in  K } \sup t_{\phi}^{\textnormal{out}}(0,x,v)  \leq T,\\
 & \forall (x,v) \in K, \quad  (X_{\phi}^{\textnormal{out}},V_{\phi}^{\textnormal{out}})(0,x,v) \in \Sigma^{\textnormal{out}}.
\end{align}
\end{definition}
\subsection{The non linear elliptic estimates}\label{sec:nl_ell_est}
In this section we prove the  wellposedness and give elliptic estimates for the non linear Poisson equation
\begin{equation} \label{nlp}
\begin{cases}
-\lambda^2 \partial_{xx} W + n_{e}(\phi^{\infty} + W) - n_{e}(\phi^{\infty}) = \rho, \textnormal{ a.e in } (0,1),\\
W(0) = 0, \quad W(1) = 0,
\end{cases}
\end{equation}
where $\rho \in L^{1}(0,1)$ is a given source term.
We have the following.
\begin{proposition} Let $\rho \in L^{1}(0,1)$. Then the non linear Poisson problem \eqref{nlp} admits a unique strong solution $W \in W^{2,1}(0,1)$. In addition, the solution satisfies the estimates
\begin{align}
&\lambda^2 \| \partial_{x} W \|_{L^2(0,1)}  \leq \| \rho  \|_{L^{1}(0,1)} , \label{papa}\\
& \| n_{e}( \phi^{\infty} + W) - n_{e}(\phi^{\infty}) \|_{L^{1}(0,1)} \leq \| \rho \|_{L^{1}(0,1)}, \label{pepe} \\
&\lambda^2 \| \partial_{xx} W \|_{L^{1}(0,1)} \leq 2 \| \rho \|_{L^{1}(0,1)}, \label{pupu}\\
& \lambda^2 \| \partial_{x} W \|_{L^{\infty}(0,1)} \leq 2 \| \rho \|_{L^{1}(0,1)}. \label{pipi}
\end{align}
If $\rho \in L^{\infty}(0,1)$ then $W \in W^{2,\infty}(0,1)$ and we have the estimates
\begin{align}
& \| n_{e}( \phi^{\infty} + W) - n_{e}(\phi^{\infty}) \|_{L^{1}(0,1)} \leq \| \rho \|_{L^{\infty}(0,1)}, \label{ppeppe}\\
&\lambda^2 \| \partial_{x} W \|_{L^{\infty}(0,1)} \leq 2 \| \rho \|_{L^{\infty}(0,1)}, \label{popo}\\
& \| n_{e}( \phi^{\infty} + W) - n_{e}(\phi^{\infty}) \|_{L^{\infty}(0,1)} \leq \frac{2 M }{\lambda^2} \| \rho \|_{L^{\infty}(0,1)}, \label{uuu}\\
&\lambda^2 \| \partial_{xx} W \|_{L^{\infty}(0,1)} \leq  \left( \frac{2 M }{\lambda^2}+1 \right) \| \rho \|_{L^{\infty}(0,1)}. \label{ppp}
\end{align}
where $M$ is a positive constant that depends only on $\lambda^2$, $\phi_{\infty}$ and $\rho.$
\begin{proof} \emph{ Step 1: Existence and uniqueness}. The proof follows from a classical variational argument.
Let $ E : H^{1}_{0}(0,1) \longrightarrow \RR$ be the functional given for all $\psi \in H^{1}_{0}(0,1)$ by
\[
E(\psi) = \int_{0}^{1} \frac{\lambda^2}{2} \left \vert \partial_{x} \psi(x) \right \vert^2  + N_{e}(\phi^{\infty}(x) + \psi(x)) - ( n_{e}(\phi^{\infty}(x)) - \rho(x) ) \psi(x) \dd x.
\]
where the function $N_{e}$ is the function given for all $s \in \RR$ by  $N_{e}(s) = \displaystyle \int_{0}^{s} n_{e}(s') \dd s'$ and where we recall that $H^{1}_{0}(0,1)$ is equipped with the usual norm $\psi \in H^{1}_{0}(0,1) \longmapsto \| \psi \|_{H^{1}_{0}(0,1)} =\big(  \int_{0}^{1} \left \vert \partial_{x} \psi(x) \right \vert^2 \dd x \big)^{\frac{1}{2}}.$
The functional $E$ is well defined and of class $\mathscr{C}^{1}$ on $H^{1}_{0}$ because the space $H^{1}_{0}$ is continuously imbedded in $L^{\infty}(0,1)$ and $n_{e} \in \mathscr{C}^{1}(\RR)$. In addition note that since the function $n_{e}' $ is positive on $\RR$, the function $N_{e}$  is strictly convex and thus the functional $\psi \in H^{1}_{0}(0,1) \longmapsto \int_{0}^{1} N_{e}(\phi^{\infty}(x) + \psi(x)) - (n_{e}(\phi^{\infty}(x)) + \rho(x)) \psi(x) \dd x $ is strictly convex. Therefore $E$ is a strictly convex functional as the sum of two strictly convex functionals.  We then deduce that $W$ is a weak solution to \eqref{nlp} if and only if $W$ minimizes $E$ on $H^{1}_{0}(0,1)$. Let us therefore prove the existence of the minimizer. Using a Hölder inequality   we have for all $\psi \in H^{1}_{0}(0,1),$
\begin{align*}
E(\psi) \geq  \frac{\lambda^{2}}{2} \| \psi \|_{H^{1}_{0}(0,1)}^2 + \int_{0}^{1} N_{e}(\phi^{\infty}(x) + \psi(x)) \dd x   - \| n_{e} \circ \phi^{\infty} - \rho \|_{L^{1}(0,1)} \| \psi \|_{L^{\infty}(0,1)}.
\end{align*}
Since $N_{e}$ is a positive function and using the fact that for all $\psi \in H^{1}_{0}(0,1)$ we have $ \| \psi \|_{L^{\infty}(0,1)} \leq \| \psi \|_{H^{1}_{0}(0,1)}$ we get 
\[
E(\psi) \geq \frac{\lambda^{2}}{2} \| \psi \|_{H^{1}_{0}(0,1)}^2   - \| n_{e} \circ \phi^{\infty} + \rho \|_{L^{1}(0,1)} \| \psi \|_{H^{1}_{0}(0,1)}.
\]
Let $0 < \varepsilon < \lambda$.  Using an $\varepsilon$ dependent Young inequality for the second term we obtain
\begin{equation}
E(\psi) \geq  \frac{\lambda^2 - \varepsilon^2}{2} \| \psi \|_{H^{1}_{0}(0,1)}^{2}  - \frac{ \| n_{e} \circ \phi^{\infty} - \rho \|_{L^{1}(0,1)}^2}{2 \varepsilon^2}. \label{coercivity_estimate}
\end{equation}
It is a coercivity estimate for the functional $E$. By Mazur's Lemma,  since  $E$ is convex and continuous in $H^{1}_{0}(0,1)$, it is lower semi-continuous for the weak topology of $H^{1}_{0}(0,1).$ Consequently, the coercivity of the functional $E$ with its lower semi continuity for the weak topology of $H^{1}_{0}(0,1)$ yields the existence of a minimizer in $H^{1}_{0}(0,1)$. Since $E$ is striclty convex the minimizer is unique in $H^{1}_{0}(0,1)$. Let therefore $W = \underset{ \psi \in H^{1}_{0}(0,1)} {\textnormal{ argmin} }E(\psi)$. Since $E$ is $\mathscr{C}^{1}$ on $H^{1}_{0}(0,1)$, the Fréchet differential of $E$ evaluated at $W$ vanishes, it yields for all $\psi \in H_{0}^{1}(0,1),$
\begin{align}
\int_{0}^{1} \lambda^2 \partial_{x} W \partial_{x} \psi  + (n_{e}(\phi^{\infty} + W) - n_{e}(\phi^{\infty})) \psi \:  \dd x = \int_{0}^{1} \rho \psi \dd x. \label{weak_sol_V}
\end{align}
Since $n_{e}(\phi^{\infty} + W) - n_{e}(\phi^{\infty}) - \rho \in L^{1}(0,1)$, it is a consequence of the weak formulation \eqref{weak_sol_V} that $W \in W^{2,1}(0,1)$ and that $W$ is a strong solution to \eqref{nlp}. Arguing similarly,  if $\rho \in L^{\infty}(0,1)$ then $W \in W^{2,\infty}(0,1).$

\emph{Step 2: Estimates.}
We prove the first estimate \eqref{papa}. Since $W^{2,1}(0,1) \subset H^{1}(0,1)$,  we take $\psi = W$ in the weak formulation \eqref{weak_sol_V}. Then we obtain since $n_{e}$ is increasing that
$\displaystyle \int_{0}^{1}  (n_{e}(\phi^{\infty} + W) - n_{e}(\phi^{\infty})) W \dd x \geq 0$ and therefore

\[
\int_{0}^{1} \lambda^2 \vert \partial_{x} W  \vert ^2  \dd x \leq \int_{0}^{1} \rho W  \dd x \leq \| \rho \|_{L^{1}(0,1)} \| W \|_{L^{\infty}(0,1)}  \leq  \| \rho \|_{L^{1}(0,1)} \| \partial_{x} W \|_{L^{2}(0,1)}
\]
where we have used the Hölder inequality and the fact $\| W \|_{L^{\infty}(0,1)} \leq \| \partial_{x} W \|_{L^{2}(0,1)}.$ It yields the expected estimate.
We now prove the second estimate which follows the same lines as in Proposition \ref{estimate_Linf_E}. So we repeat the same argument.  Let $(\varphi_{n})_{n \in \NN} \subset \mathscr{C}^{1}(\RR)$ a sequence of functions such that for all $n \in \NN:$
\begin{align*}
    \varphi_n(0) = 0,\\
    \forall  u \in \RR, \varphi_{n}'(u) > 0, \:  |\varphi_n(u)| \leq 1\\
    \forall u \neq 0, \: \varphi_{n}(u) \longrightarrow \textnormal{sgn}(u) \textnormal{ as } n \rightarrow +\infty.
\end{align*}
Since $W \in W^{2,1}(0,1)$, for each $n \in \NN$, the function $\varphi_{n} \circ W$ belongs to $H^{1}(0,1)$ and it verifies $\varphi_{n} \circ W(0) = \varphi_{n} \circ W(1) = 0.$ We then use $\varphi_{n} \circ W$ as a test function in \eqref{weak_sol_V}. We thus get
\begin{align*}
\int_{0}^{1} \lambda^2 \left \vert  \partial_{x} W \right  \vert^2  \varphi_{n}'(W) \dd x + \int_{0}^{1} (n_{e}(\phi^{\infty} + W) - n_{e}(\phi^{\infty})) \varphi_{n} \circ W \dd x = \int_{0}^{1} \rho \varphi_{n} \circ W \dd x.
\end{align*}
Note that the first integral is non negative therefore we get
\[
\int_{0}^{1} (n_{e}(\phi^{\infty} + W) - n_{e}(\phi^{\infty})) \varphi_{n} \circ W \dd x \leq \int_{0}^{1} \rho \varphi_{n} \circ W \dd x \leq \| \rho \|_{L^{1}(0,1)} 
\]
where the last inequality was obtained using the Hölder inequality combined with the $L^{\infty}$ bound $\| \varphi_{n} \circ W \|_{L^{\infty}} \leq 1.$
Using the Lebesgue Dominated convergence theorem we have, because $n_{e}$ is increasing, 
\[
\int_{0}^{1} (n_{e}(\phi^{\infty} + W) - n_{e}(\phi^{\infty})) \varphi_{n} \circ W  \dd x  \underset{ n \rightarrow +\infty}{\longrightarrow} \int_{0}^{1} \big \vert  (n_{e}(\phi^{\infty} + W) - n_{e}(\phi^{\infty})) \big \vert \dd x.
\]
It yields the second estimate.
The third estimate is obtained using the strong form of the equation \eqref{nlp} together with the estimate \eqref{pepe}. As for the last estimate, since $\partial_{x} W \in W^{1,1}(0,1) \subset C^{0}[0,1]$ and $\int_{0}^{1} \partial_{x} W(x') \dd x' = 0$ because of the Dirichlet boundary condition, we have that there exists $x_{0} \in [0,1]$ such that for all $x \in [0,1]$, $\partial_{x} W(x) = \int_{x_{0}}^{x} \partial_{xx} W(x') \dd x'.$ Then $ \vert \partial_{x} W (w) \vert \leq \| \partial_{xx} W \|_{L^{1}(0,1)}$ and using the estimate \eqref{pupu} we conclude.

We now prove the estimates in the case when $\rho \in L^{\infty}(0,1)$. Note that \eqref{ppeppe} is obtained exactly as previously at the difference that we use rather the following bound for every $n \in \NN,$
\[
\int_{0}^{1} \rho \varphi_{n} \circ W \dd x \leq \| \rho \|_{L^{\infty}(0,1)} \int_{0}^{1} \left \vert  \varphi_{n} \circ W  \right \vert \dd x \leq \| \rho \|_{L^{\infty}(0,1)}
\]
since $| \varphi_{n} | \leq 1.$
From \eqref{ppeppe} and the strong from of the equation \eqref{nlp} we infer $\lambda^2 \| \partial_{xx} W \|_{L^{1}(0,1)}\leq 2 \| \rho \|_{L^{\infty}(0,1)}$ where we have used the continuous imbedding $L^{\infty}(0,1) \hookrightarrow L^{1}(0,1).$ Then,  we deduce because $\int_{0}^{1} \partial W (x) \dd x = 0$ that 
$\| \partial_{x} W \|_{L^{\infty}(0,1)} \leq \| \partial_{xx} W \|_{L^{1}(0,1)} \leq \frac{2}{\lambda^2} \| \rho \|_{L^{\infty}(0,1)}$ which is \eqref{popo}.
We now prove \eqref{uuu}. Since $W(0) = 0$ we have for all $x \in [0,1]$, $\vert W(x) \vert \leq \vert \int_{0}^{x}\partial_{x} W(x') \dd x' \vert \leq \| \partial_{x} W \|_{L^{\infty}(0,1)} \leq \frac{2}{\lambda^2} \| \rho \|_{L^{\infty}(0,1)}.$ Therefore by the mean value theorem we have 
\[
\left \vert n_{e}\big( \phi^{\infty}(x) + W(x) \big) - n_{e}( \phi^{\infty}(x) )  \right \vert \leq M \vert W(x) \vert \leq \frac{2 M}{\lambda^2} \| \rho \|_{L^{\infty}(0,1)},
\]
where $M = \underset{ s \in [-a,a] }{\sup}n_{e}'(s) $ where $a = \| \phi^{\infty} \|_{L^{\infty}(0,1)} + \frac{2}{\lambda^2} \| \rho \|_{L^{\infty}(0,1)}.$ It proves \eqref{uuu}. Lastly, the estimate \eqref{ppp} is obtained from the strong form of the equation \eqref{nlp} using  \eqref{uuu}.
\end{proof}
\end{proposition}
\subsection{Study of the characteristics}\label{sec:gen_characteristics_study}
In this section, we study the properties of the generic characteristics. In particular, we are interested in stability properties of the stationary characteristics with respect to small perturbation the equilibrium electric field. As for the linear case, we state a regularity result of the incoming time of a characteristic with respect to its starting point.  For fixed $t > 0$ and $ 0 < s < t $ we define the sets
\begin{align}
\mathcal{D}_{t} := \lbrace (x,v) \in Q \: : \: t_{\phi}^{\textnormal{inc}}(t,x,v) = 0 \rbrace, \label{D_t}\\
\mathcal{E}_{s} := \lbrace (x,v) \in Q \: : \: t_{\phi}^{\textnormal{inc}}(t,x,v) < s \rbrace. \label{E_s}
\end{align}
The set $\mathcal{D}_{t}$ is the set of points in $Q$ which are on characteristics that stay in $Q$ on interval the $(0,t]$.
The set $\mathcal{E}_{s}$ is the set of points in $Q$ which are on characteristics that reach  $\Sigma^{\textnormal{inc}} \cup \Sigma^{0}$ at a time smaller than $s$.
As for the linear case, we may justify the measurability of these sets.
We have the following.
\begin{lemma}[Regularity of the incoming time and measurability]\newline
Let $\phi \in \mathscr{C} \Big( \RR^{+} ; W^{2,\infty}(0,1) \Big)$ and fix $t > 0$.  We have
\begin{itemize}
\item[a)]
The three functions $(x,v) \in Q \mapsto t^{\textnormal{inc}}_{\phi}(t,x,v)$, $(x,v) \in Q \mapsto X^{\textnormal{inc}}_{\phi}(t,x,v)$, $(x,v) \in Q \mapsto V^{\textnormal{inc}}_{\phi}(t,x,v)$ are continuous at every point $(x,v) \in Q$ such that \newline
$ t^{\textnormal{inc}}_{\phi}(t,x,v) > 0$ and $(X^{\textnormal{inc}}_{\phi}(t,x,v), V^{\textnormal{inc}}_{\phi}(t,x,v)) \notin \Sigma^{0}.$ 
\item[b)]
The sets $\mathcal{D}_{t}$ and $\mathcal{E}_{s}$  for every $s \in (0,t)$ given in \eqref{D_t}, \eqref{E_s} are Borel sets.
\end{itemize}
\begin{proof}  The proof of a) is in every point similar to that of  the point a) of Lemma \ref{regularity_lemma}. For b), we have $\mathcal{D}_{t} := Q \setminus \lbrace (x,v) \in Q \: : \: t^{\textnormal{inc}}_{\phi}(t,x,v) > 0 \rbrace.$ Therefore, by stability of the Borel sigma algebra, it is sufficient to prove that $ \lbrace (x,v) \in Q \: : \: t^{\textnormal{inc}}_{\phi}(t,x,v) > 0 \rbrace$ is a Borel set. One has a decomposition
\begin{align*}
\lbrace (x,v) \in Q \: : \: t_{\phi}^{\textnormal{inc}}(t,x,v) > 0 \rbrace =  F_{t} \cup G_{t},\\
F_{t}  = \lbrace (x,v) \in Q : t^{\textnormal{inc}}_{\phi}(t,x,v) > 0, \:  (X^{\textnormal{inc}}_{\phi}(t,x,v), V^{\textnormal{inc}}_{\phi}(t,x,v)) \in \Sigma^{0} \rbrace,\\
G_{t} =  \lbrace (x,v) \in Q : t^{\textnormal{inc}}_{\phi}(t,x,v) > 0, \:  (X^{\textnormal{inc}}_{\phi}(t,x,v), V^{\textnormal{inc}}_{\phi}(t,x,v)) \notin \Sigma^{0} \rbrace.
\end{align*}
One may adapt the proof of Proposition 2.3 in \cite{Bardos} to prove that $F_{t}$ is a Borel set of measure zero. $G_{t}$ is an open set as a consequence of the continuity of $t^{\textnormal{inc}}_{\phi}(t,\cdot,\cdot)$ outside the points that reach transversally the boundary. The same argument of continuity yields the claim for the set $\mathcal{E}_{s}$. 
\end{proof}
\end{lemma}
We now come to a key lemma in the analysis. Unlike the linear analysis, the microscopic energy $(x,v) \in Q \mapsto \frac{v^2}{2} + \phi^{\infty}(x)$ is not conserved along the generic characteristics \eqref{general-char}. Since we shall ask the supports of the initial data $h_{0}$ and the equilibrium $f^{\infty}$ to be imbedded in $D^{+}_{r}.$  It is natural to track the evolution of the set  $D^{+}_{r}$  by the flow \eqref{flow_gen}. On that occasion, we recall that the equilibrium electric field $-\partial_{x} \phi^{\infty}$ verifies the exit geometric conditions  \eqref{egc1}, \eqref{egc2} at time $T_{r}= \frac{1}{r}$ with respect to $D^{+}_{r}$. We would like that the perturbed electric field $-\partial_{x} \phi$ does so, at least, at a larger time than $T_{r}$.  In that regard, we set for ease in the reading
\begin{align}
\tilde{T}_{r} = 2T_{r}\label{TT} = \frac{2}{r} \textnormal{ if } r > 0.
\end{align}
We will show that  under an appropriate smallness assumption, $-\partial_{x} \phi$ verifies the exit geometric conditions \eqref{egcc_1}-\eqref{egcc_2} in time $\tilde{T}_{r}.$ We have.

\begin{lemma}[Stability of the the microscopic energy and  exit geometric conditions]  \label{stability_micro_energy} Fix $r > 0.$ Let  $J$ a subinterval of $\RR^{+}$. For any $ U \in \mathscr{C}(\RR^{+}; W^{2,\infty}(0,1))$ such that
\begin{align}  \label{egc_cond}
\forall s \in J, \quad  \int_{s}^{s+\tilde{T}_{r} } \frac{\| \partial_{x} U(\tau) \|_{L^{\infty}(0,1)}^2}{2} \dd \tau + \tilde{T}_{r}  \| \phi^{\infty} \|_{L^{\infty}(0,1)}  <  r^{2}\Big( \exp(-\tilde{T}_{r}) - \frac{1}{8}\Big).
\end{align}
we have
\begin{itemize}
\item[a)] for all  $s \in J$ and $(x,v) \in D^{+}_{r}$, the characteristic \eqref{general-char} associated with the potential $\phi = U + \phi^{\infty}$ which started at time $s$ from $(x,v)$ verifies 
\begin{align}
\forall t \in [s,s+\tilde{T}_{r}], \quad V_{\phi}(t;s,x,v) > \frac{r}{2}. \label{stability_flow}
\end{align}
\item[b)]In addition, $\partial_{x} \phi$ verifies the internal exit geometric condition in time $\tilde{T}_{r}$ with respect to $D^{+}_{r}$ on $J$ and we have
\begin{align}
\forall t \in [s, t_{\phi}^{\textnormal{out}}(s,x,v)\big), \quad (X_{\phi},V_{\phi})(t;s,x,v) \in D^{+}_{\frac{r}{2}}.
\end{align}
Furthermore, if $ 0 \in J$ then $\partial_{x} \phi$ satisfies also the initial exit geometric condition  in time $\tilde{T}_{r}$ with respect to $D^{+}_{r}.$ 
\end{itemize}
\begin{proof} Fix $r > 0$, $J$ a subinterval of $\RR^{+}$ and consider $U \in \mathscr{C}(\RR^{+};W^{2,\infty}(0,1))$ which verifies \eqref{egc_cond}.
\emph{ Proof of a).} Let $(s,x,v) \in J \times D^{+}_{r}$. We have for all $t \in \RR^{+}$
\begin{align}
&\frac{\dd}{\dd s} X_{\phi}(t;s,x,v) = V_{\phi}(t;s,x,v) , \label{uno}\\
& \frac{\dd}{\dd s} V_{\phi}(t;s,x,v) = - \partial_{x} U(t;X_{\phi}(t;s,x,v)) - \partial_{x} \phi^{\infty}(X_{\phi}(t;s,x,v)). \label{dos}
\end{align}
We multiply \eqref{dos} by $V_{\phi}(t;s,x,v)$ to get,
\begin{align}
\frac{\dd}{\dd t} \Big( \frac{V_{\phi}^2(t;s,x,v)}{2} + \phi^{\infty}\Big(X_{\phi}(t;s,x,v) \Big) \Big) = - \partial_{x} U(t;X_{\phi}(t;s,x,v)) V_{\phi}(t;s,x,v).
\end{align}
For $t \geq s$, we get after an integration on $[s,t]$,
 \begin{align}
& \frac{V_{\phi}^2(t;s,x,v)}{2} + \phi^{\infty}\Big(X_{\phi}(t;s,x,v)\Big)\nonumber\\
 & = \frac{v^2}{2} + \phi^{\infty}(x) \textcolor{blue}{-} \int_{s}^{t} \partial_{x}U\big( t';X_{\phi}(t';s,x,v) \big) V_{\phi}(t';s,x,v) \dd t'. \label{energy_qcon}
 \end{align}
 Using the Young inequality, for $ab \geq - \frac{a^2 + b^2}{2}$ for all real numbers $a$ and $b$,  we obtain
  \begin{align}
 &\frac{V_{\phi}^2(t;s,x,v)}{2} + \phi^{\infty}\Big(X_{\phi}(t;s,x,v)\Big)\\
 & \geq \frac{v^2}{2} + \phi^{\infty}(x) - \int_{s}^{t} \frac{V_{\phi}^2(t';s,x,v)}{2} + \phi^{\infty}\Big(X_{\phi}(t';s,x,v)\Big)  \dd t' - B(t,s),
 \end{align}
 where $B(t,s) = \displaystyle  \int_{s}^{t} \frac{1}{2} \| \partial_{x} U(t') \|_{L^{\infty}(0,1)}^2 \dd t' + (t-s) \| \phi^{\infty}\|_{L^{\infty}(0,1)}.$ Since $(x,v) \in D^{+}_{r}$ we have $\frac{v^2}{2} + \phi^{\infty}(x) > r^2.$ Setting $I(t) = \frac{V_{\phi}^2(t;s,x,v)}{2} + \phi^{\infty}\Big(X_{\phi}(t;s,x,v)\Big) $ we thus get the integral inequality
 \begin{align}
 \forall t \geq s , \quad I(t) > r^2 - \int_{s}^{t} I(t') \dd t' -B(t,s).
 \end{align}
Since the function $t \longmapsto B(t,s)$ is increasing, using a Grönwall Lemma, we get

\begin{align}
\forall t \geq s, \quad I(t) >  r^2\exp(s-t) - B(t,s).
\end{align}
Remark that the function $t \in [s, s+\tilde{T}_{r}] \longmapsto r^2 \exp(s-t) - B(t,s)$ is decreasing, therefore we obtain
\begin{align}
\forall t \in [s,s+T_{r}], \quad I(t) >  r^2\exp(s-t) - B(t,s) \geq r^2 \exp(-\tilde{T}_{r}) - B(s+\tilde{T}_{r},s)
\end{align}
By assumption \eqref{egc_cond}, we have $r^2 \exp(-\tilde{T}_{r}) - B(s+\tilde{T}_{r},s) > \frac{r^2}{8}$. It yields
\begin{align}
\forall t \in [s,s+\tilde{T}_{r}], \quad \frac{V_{\phi}^2(t;s,x,v)}{2} + \phi^{\infty}\Big(X_{\phi}(t;s,x,v)\Big) > \frac{r^2}{8}. \label{energy_est}
\end{align}
Observe that $\phi^{\infty}$ is non positive on $[0,1]$.  Therefore, since $V_{\phi}(t;t,x,v) = v > 0$ we deduce thanks to \eqref{energy_est} and a standard continuity argument that $V_{\phi}(t;s,x,v) > 0$ for all $t \in [s,s+\tilde{T}_{r}]$. Thus for all $ t \in [s,s+\tilde{T}_{r}]$ we have $V_{\phi} (t;s,x,v) > \frac{r}{2}.$

\emph{ Proof of b).} From the previous point we have for all $t \in [s,s+\tilde{T}_{r}]$, $V_{\phi}(t;s,x,v) > \frac{r}{2}.$ By integration on $[s,s+\tilde{T}_{r}]$ we obtain $X_{\phi}(s+\tilde{T}_{r};s,x,v) > x + \frac{\tilde{T}_{r} \times r}{2} =   x+1 > 1$. By definition of the outgoing time we infer $t_{\phi}^{\textnormal{out}}(s,x,v) < s + \tilde{T}_{r}$. Since it holds for every $s \in J$ and $(x,v) \in D^{+}_{r}$, taking the supremum on $J \times D^{+}_{r}$ yields the internal exit geometric condition. Eventually, remark that  if $0 \in J$  then the previous arguments applied at $s = 0$ show that the initial exit geometric condition with respect to $D^{+}_{r}$ is also satisfied. Finally, remark that since $s < t_{\phi}^{\textnormal{out}}(s,x,v) < s + \tilde{T}_{r}$ we thus have for $t \in [s, t_{\phi}^{\textnormal{out}}(s,x,v)),$ $(X_{\phi},V_{\phi})(t;s,x,v) \in D^{+}_{\frac{r}{2}}.$
\end{proof}
\end{lemma}

To conclude this  study of the generic characteristics let us prove the following elementary lemma.
\begin{lemma} \label{non_vanishing_source_term}Let $\phi \in \mathscr{C}(\RR^{+}; W^{2,\infty}(0,1))$.   Let $P \subset Q$ and fix $(t,x,v)  \in \RR^{+}_{*} \times Q$. Assume there exists $t_{\phi}^{\textnormal{inc}}(t,x,v) <  s < t $ such that $(X_{\phi},V_{\phi})(s;t,x,v) \in P$. Then there exists $(x',v') \in P$ such  $(x,v) = (X_{\phi},V_{\phi})(t;s,x',v')$ and $t_{\phi}^{\textnormal{out}}(s,x',v') > t$.
\begin{proof} Let $P \subset Q$ and  fix $(t,x,v) \in \RR^{+}_{*} \times Q$. \newline
Assume there exists $ s \in (t_{\phi}^{\textnormal{inc}}(t,x,v), t)$ such that $(X_{\phi},V_{\phi})(s;t,x,v) \in P$.    Set $(x',v') = (X_{\phi},V_{\phi})(s;t,x,v) \in P$. Then, $(X_{\phi},V_{\phi})(t;s,x',v') = (x,v).$ Since $t_{\phi}^{\textnormal{inc}}(t,x,v) < s < t$, we have for all $s' \in [s,t]$, $ (X_{\phi},V_{\phi})(s';s,x',v') = (X_{\phi},V_{\phi})(s';t,x,v) \in Q$. Therefore, $t_{\phi}^{\textnormal{out}}(s,x',v') > t.$ 
\end{proof}
\end{lemma}
To study the contribution of the initial data and of the source term in the solution of the Vlasov equation \eqref{def_h_nonlinear} we shall make a great use of this lemma in its opposite version: let $P \subset Q$ and fix $(t,x,v) \in \RR^{+}_{*} \times Q$. If there is some $ 0 < s _{0} < t $ such that for all $s > s_{0}$ and all $(x',v') \in P,$  we have $t_{\phi}^{\textnormal{out}}(s,x',v') \leq t$ then for all $s$ such that $ t_{\phi}^{\textnormal{inc}}(t,x,v) < s < t$ we have $(X_{\phi},V_{\phi})(s;t,x,v) \notin P$. 

\subsection{Local stability estimates for \textbf{(PVP)}} \label{sec:loc_stab_est}
From now, we consider 
\[
(h,U) \in \mathscr{C}\big( \RR^{+}; L^{1}(Q)\big) \times  \mathscr{C}\big(\RR^{+}; W^{2,\infty}(0,1)\big)
\]
 a global mild-strong solution to \textbf{(PVP)} associated with an initial data $h_{0} \in \mathscr{A}$ and an equilibrium $(f^{\infty},\phi^{\infty})$ given by Theorem \ref{theorem_equilibrium}. In particular, we have for every $t \geq 0$ and a.e $(x,v) \in Q$
\begin{align}
&h(t,x,v) = \mathbf{1}_{t_{\phi}^{\textnormal{inc}}(t,x,v) = 0} h_{0} \Big( X_{\phi}\Big(0;t,x,v \Big), V_{\phi}\Big(0;t,x,v\Big) \Big) \label{formula_h_PVP}\\
    &+  \int_{0}^{t} \mathbf{1}_{ t_{\phi}^{\textnormal{inc}}(t,x,v)  < s} \partial_{x} U \Big(s, X_{\phi}\Big(s;t,x,v \Big) \Big)  \partial_{v} f^{\infty} \Big(   X_{\phi}\Big(s;t,x,v \Big), V_{\phi}\Big(s;t,x,v\Big) \Big)  \dd s.\nonumber
\end{align}
Thanks to this formula, we deduce the following linear type estimate
\begin{align}
\forall t \geq 0, \quad \| h(t) \|_{L^{1}(Q)} \leq \| h_{0} \|_{L^{1}(Q)} + \| \partial_{v} f^{\infty} \|_{L^{1}(Q)} \int_{0}^{t} \| \partial_{x} U(s') \|_{L^{\infty}(0,1)} \dd s'. \label{l1_estimate}
\end{align}
Thanks to the elliptic estimate  \eqref{pipi}, we deduce for every $t \geq 0$
\begin{align}
 \| h(t) \|_{L^{1}(Q)} \leq    \| h_{0} \|_{L^{1}(Q)}\exp \Big( \frac{2  \| \partial_{v} f^{\infty} \|_{L^{1}(Q)} t}{\lambda^2} \Big), \label{increase_exp} \\
 \| \partial_{x} U(t) \|_{L^{\infty}(0,1)} \leq \frac{2}{\lambda^2} \| h_{0} \|_{L^{1}(Q)} \exp \Big( \frac{2  \| \partial_{v} f^{\infty} \|_{L^{1}(Q)} t}{\lambda^2} \Big) \label{nl_est_dxU}.
\end{align}
Based on the representation formula \eqref{formula_h_PVP} and the stability analysis of the characteristics Lemma \ref{stability_micro_energy}, our goal now is now to show that the $L^{1}$ norm of $h(t)$ verifies a delayed Gronwall type inequality provided the initial fluctuation is small enough in $L^{1}$. Especially, we fix $r > 0$ and we are going to establish that for each $t^{\star} > \tilde{T}_{r}$, provided $\|h_{0}\|_{L^{1}(Q)}$ is small enough we have
\begin{align*}
&\forall t \in [0,t^{\star}], \quad \supp h(t) \subset D^{+}_{\frac{r}{2}},\\
&h(t,x,v) = \int_{t-\tilde{T}_{r}}^{t} \mathbf{1}_{ t_{\phi}^{\textnormal{inc}}(t,x,v)  < s} \partial_{x} U \Big(s, X_{\phi}\Big(s;t,x,v \Big) \Big)  \partial_{v} f^{\infty} \Big(   (X_{\phi}, V_{\phi})(s;t,x,v) \Big)  \dd s,
\end{align*}
for $t \in [\tilde{T}_{r},t^{\star}]$ and a.e $(x,v) \in D^{+}_{\frac{r}{2}}.$
The strategy is as follows:
\begin{itemize}
\item We prove thanks to the bounds \eqref{nl_est_dxU} that $\partial_{x} U$ verifies the smallness condition \eqref{egc_cond} on $J = [0,t^{\star}]$ provided $ \| h_{0} \|_{L^{1}(Q)}$ is small enough and $r >0 $ is chosen large enough. Then our key Lemma \ref{stability_micro_energy} applies.
\item  We then use the fact that the perturbed electric field $-\partial_{x} \phi$ satisfies the initial exit geometric condition at time $\tilde{T}_{r}$. So for $t \geq \tilde{T}_{r}$, in \eqref{formula_h_PVP} the contribution of the initial condition vanishes because either the condition $t_{\phi}^{\textnormal{inc}}(t,x,v) = 0$ is not met or the characteristics at time zero is outside the support of the initial data.
\item  Lastly, we use the fact that the perturbed electric field $-\partial_{x} \phi$ satisfies the internal exit geometric condition at time $\tilde{T}_{r}$ to prove that in $\eqref{formula_h_PVP}$ and for $t > \tilde{T}_{r}$, the piece of integral on $[0,t-\tilde{T}_{r}]$ vanishes because in \eqref{formula_h_PVP} the characteristic starting at time $t$ from $(x,v) \in D^{+}_{r}$ has spent a priori a time longer than $\tilde{T}_{r}$ in $D^{+}_{r}$ so it cannot be in the support of $\partial_{v} f^{\infty}$ for $s \in [0,t-\tilde{T}_{r}]$.\end{itemize}
To reach our goal, we somehow proceed in a reversed way compared to the strategy sketched above. We begin to give a sufficient condition on $\partial_{x} U$ to get the expected formula for $h$ as $t \geq \tilde{T}_{r}.$ We will show at the end of this section that this sufficient condition is fulfilled provided $\|h_{0}\|_{L^{1}(D^{+}_{r})}$ is small enough and $r > 0$ is large enough.
\begin{lemma}\label{non_linear_local_estimates} Let $r > 0$.
Suppose 
\begin{align}
\supp h_{0},  \quad \supp \partial_{v} f^{\infty} \subset D^{+}_{r}
\end{align}
and in addition, that $U \in \mathscr{C}\big( \RR^{+}; W^{2,\infty}(0,1) \big)$ verifies for some $t^{\star}> \tilde{T}_{r} $
\begin{align}
\forall \hat{s} \in [\tilde{T}_{r}; t^{\star} + \tilde{T}_{r}], \: \int_{\hat{s}-\tilde{T}_{r}}^{\hat{s}}  \frac{\| \partial_{x} U(\tau) \|_{L^{\infty}(0,1)}^2}{2} \dd \tau + \tilde{T}_{r}  \| \phi^{\infty} \|_{L^{\infty}(0,1)}  <  r^{2}\Big( \exp(-\tilde{T}_{r}) - \frac{1}{8}\Big). \label{shifted_egc}
\end{align} 
Then we have
\begin{align}
&\forall t \in [0,t^{\star}], \quad \supp h(t) \subset D^{+}_{\frac{r}{2}}, \label{supp_h}\\
&\forall t \in [\tilde{T}_{r},t^{\star}], \quad \| h(t) \|_{L^{1}\Big(D^{+}_{\frac{r}{2}} \Big)} \leq \| \partial_{v} f^{\infty} \|_{L^{1}(D^{+}_{r})} \int_{t - \tilde{T}_{r}}^{t}\| \partial_{x} U(s) \|_{L^{\infty}(0,1)} \dd s, \label{gronwall_h1}\\
&\forall t \in [\tilde{T}_{r},t^{\star}], \quad \| h(t) \|_{L^{1}\Big(D^{+}_{\frac{r}{2}} \Big)} \leq \frac{2\| \partial_{v} f^{\infty} \|_{L^{1}(D^{+}_{r})}}{\lambda^2}  \int_{t - \tilde{T}_{r}}^{t}\| h(s) \|_{L^{1}\Big(D^{+}_{\frac{r}{2}} \Big)} \dd s.\label{gronwall_h2}
\end{align}
\begin{proof}  
We begin to prove \eqref{supp_h}. For $t \in [0,t^{\star}]$ and a.e $(x,v) \in Q$ we have,
\begin{align*}
&h(t,x,v) = \mathbf{1}_{t_{\phi}^{\textnormal{inc}}(t,x,v) = 0} h_{0} \Big( X_{\phi}\Big(0;t,x,v \Big), V_{\phi}\Big(0;t,x,v\Big) \Big)\\
    &+  \int_{0}^{t} \mathbf{1}_{ t_{\phi}^{\textnormal{inc}}(t,x,v)  < s} \partial_{x} U \Big(s, X_{\phi}\Big(s;t,x,v \Big) \Big)  \partial_{v} f^{\infty} \Big(   X_{\phi}\Big(s;t,x,v \Big), V_{\phi}\Big(s;t,x,v\Big) \Big)  \dd s.\nonumber
\end{align*}
We denote for ease 
\[
h_{1}(t,x,v) = \mathbf{1}_{t_{\phi}^{\textnormal{inc}}(t,x,v) = 0} h_{0} \Big( X_{\phi}\Big(0;t,x,v \Big), V_{\phi}\Big(0;t,x,v\Big) \Big)\]
 and 
 \[
 h_{2}(t,x,v) = \int_{0}^{t} \mathbf{1}_{ t_{\phi}^{\textnormal{inc}}(t,x,v)  < s} \partial_{x} U \Big(s, X_{\phi}\Big(s;t,x,v \Big) \Big)  \partial_{v} f^{\infty} \Big(   (X_{\phi},V_{\phi})(s;t,x,v) \Big)  \dd s.
 \]
Let us study separately each term.
\newline
\emph{Study of $h_{1}$.}
Note that if $h_{1}(t,x,v) \neq 0$ therefore $t_{\phi}^{\textnormal{inc}}(t,x,v)= 0$ and \newline $(X_{\phi},V_{\phi})(0,t,x,v) \in D^{+}_{r}.$ Therefore there is $(x',v') \in D^{+}_{r}$ such that \newline $(x,v) = (X_{\phi},V_{\phi})(t;0,x',v')$ and $t_{\phi}^{\textnormal{out}}(0,x',v') > t.$ By assumption, $\partial_{x} U $ verifies \eqref{shifted_egc}  which means that it verifies \eqref{egc_cond} on the interval $J = [0,t^{\star}].$ Since $t \leq t^{\star}$ we deduce by application of Lemma \ref{stability_micro_energy} b) that $(x,v) = (X_{\phi},V_{\phi})(t;0,x',v') \in D^{+}_{\frac{r}{2}}$ provided $t < t_{\phi}^{\textnormal{out}}(0,x',v').$  This shows that $(x,v) \in D^{+}_{\frac{r}{2}}.$
\newline
\emph{Study of $h_{2}$.}
Observe that $h_{2}(t,x,v) \neq 0$ if there is a least one $s$ such  such that $t_{\phi}^{\textnormal{inc}}(t,x,v) < s < t $ and $(X_{\phi},V_{\phi})(s;t,x,v) \in D^{+}_{r}.$ By Lemma \ref{non_vanishing_source_term}, there is $(x',v') \in D^{+}_{r}$ such that $(x,v) = (X_{\phi},V_{\phi})(t;s,x',v')$ and $t_{\phi}^{\textnormal{out}}(s,x',v') > t$.  By assumption, $\partial_{x} U $ verifies \eqref{egc_cond} on the interval $J = [0,t^{\star}].$ Since $t \leq t^{\star}$ we deduce by application of Lemma \ref{stability_micro_energy} b) that $(x,v) = (X_{\phi},V_{\phi})(t;s,x',v') \in D^{+}_{\frac{r}{2}}$ provided $ t <t_{\phi}^{\textnormal{out}}(s,x',v')$. Therefore $(x,v) \in D^{+}_{\frac{r}{2}}$.
The first claim is proven.
\newline
We now prove the estimates \eqref{gronwall_h1} and \eqref{gronwall_h2}. As previously we study separately each term.
\newline
\emph{Study of $h_{1}$.}
By the preceding study, we know that if $h_{1}(t,x,v) \neq 0$ then there is $(x',v') \in D^{+}_{r}$ such that $(x,v) = (X_{\phi},V_{\phi})(t;0,x',v')$ and $t_{\phi}^{\textnormal{out}}(0,x',v') > t$. Since $\partial_{x} U$ verifies \eqref{egc_cond} on $[0,t^{\star}]$, according to Lemma \ref{stability_micro_energy} a),  we have that $-\partial_{x} \phi$ satisfies the initial exit geometric condition at time $\tilde{T}_{r}$ with respect to $D^{+}_{r}$. Therefore it $ t\geq \tilde{T}_{r}$ we have $t \geq t_{\phi}^{\textnormal{out}}(0,x',v')$ and thus $h_{1}(t,x,v) = 0.$

\emph{Study of $h_{2}$.} 
If $h_{2}(t,x,v) \neq 0$ then there is at least one $s$ such that $t_{\phi}^{\textnormal{inc}}(t,x,v) < s < t$ and  $(X_{\phi},V_{\phi})(s;t,x,v) \in D^{+}_{r}.$ By Lemma \ref{non_vanishing_source_term}, there is $(x',v') \in D^{+}_{r}$  such that $(x,v) = (X_{\phi},V_{\phi})(t;s,x',v')$ and $t_{\phi}^{\textnormal{out}}(s,x',v') > t$. Since $\partial_{x} U$ verifies  \eqref{shifted_egc}, it means that $\partial_{x} U$ verifies \eqref{egc_cond} on $J = [0,t^{\star}-\tilde{T}_{r}].$ According to Lemma \ref{stability_micro_energy} a), we deduce that $\partial_{x} \phi$ verifies the internal exit geometric condition in time $\tilde{T}_{r}$ with respect to $D^{+}_{r}$ on $[0,t^{\star}-\tilde{T}_{r}].$  For $s \in [t_{\phi}^{\textnormal{inc}}(t,x,v),t^{\star}-\tilde{T}_{r}]$ we then have $t_{\phi}^{\textnormal{out}}(s,x',v') \leq s + \tilde{T}_{r}.$ Therefore if $s \leq t-\tilde{T}_{r}$  then $t_{\phi}^{\textnormal{out}}(s,x',v') \leq t$. Thus,
\begin{align*}
h_{2}(t,x,v)= \int_{t-\tilde{T}_{r}}^{t} \mathbf{1}_{ t_{\phi}^{\textnormal{inc}}(t,x,v)  < s} \partial_{x} U \Big(s, X_{\phi}\Big(s;t,x,v \Big) \Big)  \partial_{v} f^{\infty} \Big(   (X_{\phi}, V_{\phi})(s;t,x,v) \Big)  \dd s.
\end{align*}
\emph{Conclusion}.
For $t \in [\tilde{T}_{r}, t^{\star}]$ we have for a.e $(x,v) \in D^{+}_{\frac{r}{2}},$
\begin{align*}
h(t,x,v)= \int_{t-\tilde{T}_{r}}^{t} \mathbf{1}_{ t_{\phi}^{\textnormal{inc}}(t,x,v)  < s} \partial_{x} U \Big(s, X_{\phi}\Big(s;t,x,v \Big) \Big)  \partial_{v} f^{\infty} \Big(   (X_{\phi},V_{\phi})(s;t,x,v) \Big)  \dd s.
\end{align*}
We now estimate $\| h(t) \|_{L^{1}\Big(D^{+}_{\frac{r}{2}} \Big)}$. Using a triangular inequality and the Fubini-Tonelli theorem we get
\begin{align*}
&\| h(t) \|_{L^{1}\Big(D^{+}_{\frac{r}{2}} \Big)}\\
&\leq \int_{t - \tilde{T}_{r}}^{t} \int_{D^{+}_{\frac{r}{2}}} \mathbf{1}_{ t_{\phi}^{\textnormal{inc}}(t,x,v)  < s} \left \vert  \partial_{x} U \Big(s, X_{\phi}\Big(s;t,x,v \Big) \Big) \right \vert  \left \vert \partial_{v} f^{\infty} \Big(   (X_{\phi}, V_{\phi})\Big(s;t,x,v\Big) \Big)  \right \vert \dd x \dd v \dd s.
\end{align*}
Using the $L^{\infty}$ bound \eqref{pipi}, we obtain
\begin{align*}
\| h(t) \|_{L^{1}\Big(D^{+}_{\frac{r}{2}} \Big)}\leq  \int_{t - \tilde{T}_{r}}^{t}  \| \partial_{x} U(s) \|_{L^{\infty}(0,1)} \int_{D^{+}_{\frac{r}{2}}} \mathbf{1}_{ t_{\phi}^{\textnormal{inc}}(t,x,v)  < s} \left \vert \partial_{v} f^{\infty} \Big( (X_{\phi}, V_{\phi})(s;t,x,v) \Big)  \right \vert \dd x \dd v \dd s.
\end{align*}
Using the measure preserving change of variable $(x,v) \mapsto (x',v') = (X_{\phi},V_{\phi})(s;t,x,v)$ we  obtain \eqref{gronwall_h1}.
Using the elliptic estimate \eqref{pipi} and the fact that $h$ is supported in $D^{+}_{\frac{r}{2}}$ on $[0,t^{\star}]$ we deduce \eqref{gronwall_h2} from \eqref{gronwall_h1}.
\end{proof}
\end{lemma}
As expected, we conclude this section by showing that the condition \eqref{shifted_egc}  is satisfied provided the norm of $h_{0}$ is chosen small enough and $r$ is large enough. We recall that $\vert \phi_{b} \vert = \| \phi^{\infty} \|_{L^{\infty}(0,1)}$ because $\phi^{\infty}$ is monotone decreasing with $\phi^{\infty}(0) = 0$ and $\phi^{\infty}(1) = \phi_{b}.$  We also recall that $r^{\star}(\phi_{b})$ is the unique zero of the function $\delta$ defined in \eqref{delta_r} and that for $r > r^{\star}(\phi_{b})$ we have $\delta(r) > 0.$

\begin{lemma}\label{ultimate_lemma} Let $r > r^{\star}(\phi_{b})$ and fix $t^{\star} > \tilde{T}_{r}$. Let $h_{0} \in \mathscr{A}$ such that $\supp h_{0} \subset D^{+}_{r}$, and let $(f^{\infty},\phi^{\infty})$ the equilibrium given by Theorem \ref{theorem_equilibrium} be such that $\supp \partial_{v} f^{\infty} \subset D^{+}_{r}$. There exists $\varepsilon_{0} \equiv \varepsilon_{0}\Big( t^{\star},r, \| \partial_{v} f^{\infty} \|_{L^{1}(D^{+}_{r})} \Big) > 0$ such that if
\begin{align}
\| h_{0} \|_{L^{1}(D^{+}_{r})} < \varepsilon_{0} \label{smallness_condition_initial_data}
\end{align}
then any global mild strong solution $(h,U) \in \mathscr{C}(\RR^{+}; L^{1}(Q)) \times \mathscr{C}(\RR^{+};W^{2,\infty}(0,1))$ to \textbf{(PVP)} associated with an initial data $h_{0}$ and a stationary solution $(f^{\infty},\phi^{\infty})$ as above satisfies \eqref{shifted_egc} and thus the conclusion of Lemma \ref{non_linear_local_estimates}.
\begin{proof} Fix $r >  r^{\star}$ and $t^{\star} > \tilde{T}_{r}$.  Let $(h,U) \in \mathscr{C}(\RR^{+}; L^{1}(Q)) \times \mathscr{C}(\RR^{+};W^{2,\infty}(0,1))$  be a global mild strong solution to \textbf{(PVP)} associated with an initial data $h_{0}$ and an equilibrium $(f^{\infty},\phi^{\infty})$ such that $\supp \partial_{v} f^{\infty} \subset D^{+}_{r}$. We obtain at first, thanks to the a  priori exponential growth estimate \eqref{nl_est_dxU}, for all $s \in [0,t^{\star}]$

\begin{align*}
&\int_{s}^{s + \tilde{T}_{r}}  \frac{\| \partial_{x} U(\tau) \|_{L^{\infty}(0,1)}^2}{2} \dd \tau \leq \frac{2}{\lambda^2} \| h_{0} \|_{L^{1}(D^{+}_{r})}^2  \exp\left( \frac{4 \| \partial_{v} f^{\infty} \|_{L^{1}(D^{+}_{r})} (t^{\star}+\tilde{T}_{r}) }{\lambda^2} \right)  \tilde{T}_{r}.
\end{align*}
A sufficient condition is then
\begin{align}
\frac{2}{\lambda^2} \| h_{0} \|_{L^{1}(D^{+}_{r})}^2  e^{\frac{4 \| \partial_{v} f^{\infty} \|_{L^{1}(D^{+}_{r})} (t^{\star}+\tilde{T}_{r}) }{\lambda^2} }  \tilde{T}_{r} < \delta(r). \label{sufficient_condition_h0}
\end{align}
where $\delta(r) > 0$ since $r > r^{\star}.$
For $\| h_{0} \|_{L^{1}(D^{+}_{r})}$ small enough it is trivially satisfied. Therefore there is $\varepsilon_{0} \equiv \varepsilon_{0}\Big( t^{\star},r, \| \partial_{v} f^{\infty} \|_{L^{1}(D^{+}_{r})} \Big) > 0$ such that if $\| h_{0} \|_{L^{1}(D^{+}_{r})} < \varepsilon_{0} $ we have
\[
\forall \hat{s} \in [\tilde{T}_{r}; t^{\star} + \tilde{T}_{r}], \quad \int_{\hat{s}-\tilde{T}_{r}}^{\hat{s}}  \frac{\| \partial_{x} U(\tau) \|_{L^{\infty}(0,1)}^2}{2} \dd \tau + \tilde{T}_{r}  \| \phi^{\infty} \|_{L^{\infty}(0,1)}  <  r^{2}\Big( \exp(-\tilde{T}_{r}) - \frac{1}{4}\Big).
\]
Thus \eqref{shifted_egc} is satisfied and Lemma \ref{non_linear_local_estimates} applies.
\end{proof}
\end{lemma}

\subsection{The continuation argument and the global exponential decay}\label{sec:continuation}
We now come to the final argument which consists in establishing that the estimate \eqref{gronwall_h1} holds globally in time. We fix $\phi_{b} < 0$ and consider $r > r^{\star}(\phi_{b})$ and $\mu$ satisfying the assumption of Theorem \ref{theorem_equilibrium} with in addition $\supp \mu \subset (r,+\infty)$. Let $\lambda > 0$ and consider the stationary solution $(f^{\infty},\phi^{\infty})$ given by Theorem \ref{theorem_equilibrium}. In particular we have $\supp \partial_{v} f^{\infty} \subset D^{+}_{r}.$  Applying Lemma \ref{ultimate_lemma} with $t^{\star} = 2\tilde{T}_{r}$ we know that there exists a positive constant $\varepsilon_{0} \equiv \varepsilon_{0}(r, \| \partial_{v} f^{\infty} \|_{L^{1}(D^{+}_{r})})$ such that for any $h_{0} \in \mathscr{A}$ such that
\begin{align}
\supp h_{0} \subset D^{+}_{r}, \quad \| h_{0} \|_{L^{1}(D^{+}_{r})} < \varepsilon_{0}
\end{align}
then any global mild strong solution $(h,U)$ a global mild solution to \textbf{(PVP)} verifies
\begin{align}
\forall \hat{s} \in [\tilde{T}_{r}, 3 \tilde{T}_{r}], \quad \int_{\hat{s}-\tilde{T}_{r}}^{\hat{s}}  \frac{\| \partial_{x} U(\tau) \|_{L^{\infty}(0,1)}^2}{2} \dd \tau < \delta(r) \label{zoro}
\end{align}
where $\delta(r)$ is given in \eqref{delta_r}.
Let us define
\begin{align}
t_{c} = \sup \Big \lbrace t > 2 \tilde{T}_{r} \:  : \: \forall \hat{s} \in [\tilde{T}_{r}, t + \tilde{T}_{r}] ,\:   \int_{\hat{s}-\tilde{T}_{r}}^{\hat{s}}  \frac{\| \partial_{x} U(\tau) \|_{L^{\infty}(0,1)}^2}{2} \dd \tau < \delta(r)  \Big \rbrace. \label{def_t_star}
\end{align}
Since \eqref{zoro} holds at $\hat{s} = 3\tilde{T}_{r}$, a continuity argument shows that $t_{c}$ is well-defined (it is either finite or infinite). Our goal is to prove that $t_{c} = +\infty$. We argue by contradiction and assume that $t_{c} < +\infty.$
By virtue of Lemma \ref{non_linear_local_estimates} we have
\begin{align}
&\forall t \in [0,t_{c}), \quad  \supp h(t) \subset D^{+}_{\frac{r}{2}}, \label{winnie}\\
&\forall t \in [\tilde{T}_{r}, t_{c}), \quad \| h(t) \|_{L^{1}\Big(D^{+}_{\frac{r}{2}} \Big)} \leq \| \partial_{v} f^{\infty} \|_{L^{1}(D^{+}_{r})} \int_{t - \tilde{T}_{r}}^{t}\| \partial_{x} U(s) \|_{L^{\infty}(0,1)} \dd s,\label{daffy_duck}\\
&\forall t \in [\tilde{T}_{r}, t_{c}), \quad \| h(t) \|_{L^{1}\Big(D^{+}_{\frac{r}{2}} \Big)} \leq \frac{2\| \partial_{v} f^{\infty} \|_{L^{1}(D^{+}_{r})}}{\lambda^2}  \int_{t - \tilde{T}_{r}}^{t}\| h(s) \|_{L^{1}\Big(D^{+}_{\frac{r}{2}} \Big)} \dd s. \label{blanche-neige}
\end{align}
Using \eqref{daffy_duck} and \eqref{def_t_star}, we obtain after a Cauchy-Schwarz inequality, for all $t \in [\tilde{T}_{r},t_{c} \big)$
\begin{align}
\| h(t) \|_{L^{1}\Big(D^{+}_{\frac{r}{2}} \Big)} \leq \| \partial_{v} f^{\infty} \|_{L^{1}(D^{+}_{r})} \sqrt{ 2 \tilde{T}_{r} \delta_{r} }.
\end{align}
By continuity of the map $t \in \RR^{+} \longmapsto \| h(t) \|_{L^{1}(Q)}$ we obtain,
\begin{equation} \label{estimate_h_t_star}
\| h(t_{c}) \|_{L^{1}\Big( D^{+}_{\frac{r}{2}} \Big) }\leq \| \partial_{v} f^{\infty} \|_{L^{1}(D^{+}_{r})} \sqrt{ 2 \tilde{T}_{r} \delta_{r} }.
\end{equation}
Let consider $h(t_{c})$ as an initial data for the Vlasov equation. A Grönwall lemma combined with \ref{estimate_h_t_star} yields for $t \geq t_{c}$
\begin{align}
\| h(t) \|_{L^{1}(Q)} \leq  \| \partial_{v} f^{\infty} \|_{L^{1}(D^{+}_{r})} \sqrt{ 2 \tilde{T}_{r} \delta_{r} }\exp \Big( \frac{2  \| \partial_{v} f^{\infty} \|_{L^{1}(Q)} (t-t_{c})}{\lambda^2} \Big).
\end{align}
Using the elliptic estimate \eqref{pipi} we obtain  for $t \geq t_{c},$
\begin{align}
\| \partial_{x} U (t) \|_{L^{\infty}(0,1)} \leq \frac{2}{\lambda^2} \| \partial_{v} f^{\infty} \|_{L^{1}(D^{+}_{r})} \sqrt{ 2 \tilde{T}_{r} \delta_{r} } \exp \Big( \frac{2  \| \partial_{v} f^{\infty} \|_{L^{1}(Q)} (t-t_{c})}{\lambda^2} \Big). \label{simba}
\end{align}
Using \eqref{simba}, we have  for $s  \in [t_{c}, t_{c}+ \tilde{T}_{r}]$
\begin{align}
&\int_{s}^{s + \tilde{T}_{r}}  \frac{\| \partial_{x} U(\tau) \|_{L^{\infty}(0,1)}^2}{2} \dd \tau \leq \int_{t_{c}}^{t_{c} + 2\tilde{T}_{r}} \frac{\| \partial_{x} U(\tau) \|_{L^{\infty}(0,1)}^2}{2} \dd \tau \leq\\
&\frac{4}{\lambda^2 } \| \partial_{v} f^{\infty} \|_{L^{1}(D^{+}_{r})}   \tilde{T}_{r} \delta_{r}  \left( \exp\left( \frac{8 \| \partial_{v} f^{\infty} \|_{L^{1}(D^{+}_{r})} \tilde{T}_{r} }{\lambda^2} \right) - 1 \right) . 
\end{align}
We therefore see that if 
\begin{align}
\frac{4}{\lambda^2 } \| \partial_{v} f^{\infty} \|_{L^{1}(D^{+}_{r})}   \tilde{T}_{r}  \left( \exp\left( \frac{12 \| \partial_{v} f^{\infty} \|_{L^{1}(D^{+}_{r})} \tilde{T}_{r} }{\lambda^2} \right) - 1  \right)  < 1 \label{mickey}
\end{align}
then for all $\hat{s} \in [t_{c} + \tilde{T}_{r}, t_{c}+ 2\tilde{T}_{r}]$,
\begin{align}
\int_{\hat{s}-\tilde{T}_{r}}^{\hat{s}}  \frac{\| \partial_{x} U(\tau) \|_{L^{\infty}(0,1)}^2}{2} \dd \tau < \delta(r)
\end{align}
which contradicts the maximality of $t_{c}.$ We remark that the inequality \eqref{mickey} depends only $\| \partial_{v} f^{\infty} \|_{L^{1}(D^{+}_{r})}$, $\lambda^2$ and $r$ but not on $h_{0}$. A continuity argument then shows that there exists a positive constant $\varepsilon_{\infty} \equiv \varepsilon_{\infty}(r, \lambda^2)$  such that if
\begin{align}
\| \partial_{v} f^{\infty} \|_{L^{1}(D^{+}_{r})} < \varepsilon_{\infty}
\end{align}
then \eqref{mickey} holds. Reducing $\varepsilon_{\infty}$ (if necessary) in such a way that $\varepsilon_{\infty} < \frac{\lambda^2}{4 T_{r}}$ we additionally obtain
\begin{align}
\frac{2 \| \partial_{v} f^{\infty} \|_{L^{1}(D^{+}_{r}) } \tilde{T}_{r} }{\lambda^2} < 1.
\end{align}
\emph{Conclusion.} We have proven that there exists a constant $\varepsilon_{\infty} \equiv \varepsilon_{\infty}(r,\lambda^2)$ such that
if $ \| \partial_{v} f^{\infty} \|_{L^{1}(D^{+}_{r})} \| < \varepsilon_{\infty}$ then \eqref{mickey} holds. Then we observe thanks to the proof of Lemma \ref{ultimate_lemma} that if  $h_{0}$ verifies
\begin{align}
\frac{2}{\lambda^2} \| h_{0} \|_{L^{1}(D^{+}_{r})}^2  e^{\frac{12 \tilde{T}_{r} \varepsilon_{\infty} }{\lambda^2} }  \tilde{T}_{r} < \delta(r). 
\end{align}
then Lemma \ref{ultimate_lemma} applies  for $t^{\star} = 2 \tilde{T}_{r}$. This inequality only depends on $r$ and $\lambda^2$. Thus, we have proven that there exists positive constants $\varepsilon_{\infty} \equiv \varepsilon(r,\lambda^2)$ and $\varepsilon_{0} \equiv \varepsilon_{0}(r,\lambda^2, \varepsilon_{\infty})$ such that if
\begin{align}
\| h_{0} \|_{L^{1}(D^{+}_{r})} < \varepsilon_{0}, \quad  \| \partial_{v} f^{\infty} \|_{L^{1}(D^{+}_{r})} < \varepsilon_{\infty}
\end{align}
then $t_{c} = +\infty$. Thus, any global mild strong solution to \textbf{(PVP)} verifies \eqref{winnie}-\eqref{blanche-neige} with $t_{c} = +\infty$. The delayed Grönwall Lemma \ref{delayed_Gronwall_lemma} thus applies and we obtain the estimate \eqref{Graal_1}. The estimate \eqref{Graal_2} is a consequence of the non linear elliptic estimate \eqref{pipi}. Theorem \ref{non_linear_stab_thm} is proven.

\section*{Acknowledgments}
The author thanks Ludovic Godard-Cadillac for his kind invitation at Institut des Mathématiques de Bordeaux where the early stage of this work was presented. The author is also grateful to Daniel HanKwan for his listening and his expertise. This work has been conducted with the financial support of the ANR project Muffin (ANR-19-CE46-0004).

\bibliographystyle{siamplain}
\bibliography{references}
\end{document}